\documentclass{amsart}
\usepackage{amssymb,amsmath,stmaryrd,mathrsfs}
\def\definetac{\newif\iftac}    % Can't define a \newif inside another \if!
\ifx\tactrue\undefined
  \definetac
  \ifx\state\undefined\tacfalse\else\tactrue\fi
\fi
\iftac\else\usepackage{amsthm}\fi
\usepackage[all,2cell]{xy}
\usepackage{enumitem}
\usepackage{xcolor}
\definecolor{darkgreen}{rgb}{0,0.45,0} 
\usepackage[pagebackref,colorlinks,citecolor=darkgreen,linkcolor=darkgreen]{hyperref}
\usepackage{mathtools}          % for all sorts of things
\usepackage{tikz}
\usepackage{braket}             % for \Set, etc.

\usepackage{url}                % for citations to web sites
\usepackage{xspace}             % put spaces after a \command in text

\makeatletter
\let\ea\expandafter

\def\mdef#1#2{\ea\ea\ea\gdef\ea\ea\noexpand#1\ea{\ea\ensuremath\ea{#2}\xspace}}
\def\alwaysmath#1{\ea\ea\ea\global\ea\ea\ea\let\ea\ea\csname your@#1\endcsname\csname #1\endcsname
  \ea\def\csname #1\endcsname{\ensuremath{\csname your@#1\endcsname}\xspace}}

\DeclareRobustCommand\widecheck[1]{{\mathpalette\@widecheck{#1}}}
\def\@widecheck#1#2{%
    \setbox\z@\hbox{\m@th$#1#2$}%
    \setbox\tw@\hbox{\m@th$#1%
       \widehat{%
          \vrule\@width\z@\@height\ht\z@
          \vrule\@height\z@\@width\wd\z@}$}%
    \dp\tw@-\ht\z@
    \@tempdima\ht\z@ \advance\@tempdima2\ht\tw@ \divide\@tempdima\thr@@
    \setbox\tw@\hbox{%
       \raise\@tempdima\hbox{\scalebox{1}[-1]{\lower\@tempdima\box
\tw@}}}%
    {\ooalign{\box\tw@ \cr \box\z@}}}

\newcount\foreachcount

\def\foreachletter#1#2#3{\foreachcount=#1
  \ea\loop\ea\ea\ea#3\@alph\foreachcount
  \advance\foreachcount by 1
  \ifnum\foreachcount<#2\repeat}

\def\foreachLetter#1#2#3{\foreachcount=#1
  \ea\loop\ea\ea\ea#3\@Alph\foreachcount
  \advance\foreachcount by 1
  \ifnum\foreachcount<#2\repeat}

\def\definescr#1{\ea\gdef\csname s#1\endcsname{\ensuremath{\mathscr{#1}}\xspace}}
\foreachLetter{1}{27}{\definescr}
\def\definecal#1{\ea\gdef\csname c#1\endcsname{\ensuremath{\mathcal{#1}}\xspace}}
\foreachLetter{1}{27}{\definecal}
\def\definebold#1{\ea\gdef\csname b#1\endcsname{\ensuremath{\mathbf{#1}}\xspace}}
\foreachLetter{1}{27}{\definebold}
\def\definebb#1{\ea\gdef\csname l#1\endcsname{\ensuremath{\mathbb{#1}}\xspace}}
\foreachLetter{1}{27}{\definebb}
\def\definefrak#1{\ea\gdef\csname f#1\endcsname{\ensuremath{\mathfrak{#1}}\xspace}}
\foreachletter{1}{9}{\definefrak} % \fi is something else!
\foreachletter{10}{27}{\definefrak}
\def\definebar#1{\ea\gdef\csname #1bar\endcsname{\ensuremath{\overline{#1}}\xspace}}
\foreachLetter{1}{27}{\definebar}
\foreachletter{1}{8}{\definebar} % \hbar is something else!
\foreachletter{9}{15}{\definebar} % \obar is something else!
\foreachletter{16}{27}{\definebar}
\def\definetil#1{\ea\gdef\csname #1til\endcsname{\ensuremath{\widetilde{#1}}\xspace}}
\foreachLetter{1}{27}{\definetil}
\foreachletter{1}{27}{\definetil}
\def\definehat#1{\ea\gdef\csname #1hat\endcsname{\ensuremath{\widehat{#1}}\xspace}}
\foreachLetter{1}{27}{\definehat}
\foreachletter{1}{27}{\definehat}
\def\definechk#1{\ea\gdef\csname #1chk\endcsname{\ensuremath{\widecheck{#1}}\xspace}}
\foreachLetter{1}{27}{\definechk}
\foreachletter{1}{27}{\definechk}
\def\defineul#1{\ea\gdef\csname u#1\endcsname{\ensuremath{\underline{#1}}\xspace}}
\foreachLetter{1}{27}{\defineul}
\foreachletter{1}{27}{\defineul}

\def\autofmt@n#1\autofmt@end{\mathrm{#1}}
\def\autofmt@b#1\autofmt@end{\mathbf{#1}}
\def\autofmt@l#1#2\autofmt@end{\mathbb{#1}\mathsf{#2}}
\def\autofmt@c#1#2\autofmt@end{\mathcal{#1}\mathit{#2}}
\def\autofmt@s#1#2\autofmt@end{\mathscr{#1}\mathit{#2}}
\def\autofmt@f#1\autofmt@end{\mathsf{#1}}
\def\autofmt@u#1\autofmt@end{\underline{\smash{\mathsf{#1}}}}
\def\autofmt@U#1\autofmt@end{\underline{\underline{\smash{\mathsf{#1}}}}}
\def\autofmt@h#1\autofmt@end{\widehat{#1}}
\def\autofmt@r#1\autofmt@end{\overline{#1}}
\def\autofmt@t#1\autofmt@end{\widetilde{#1}}
\def\autofmt@k#1\autofmt@end{\check{#1}}

\def\auto@drop#1{}
\def\autodef#1{\ea\ea\ea\@autodef\ea\ea\ea#1\ea\auto@drop\string#1\autodef@end}
\def\@autodef#1#2#3\autodef@end{%
  \ea\def\ea#1\ea{\ea\ensuremath\ea{\csname autofmt@#2\endcsname#3\autofmt@end}\xspace}}
\def\autodefs@end{blarg!}
\def\autodefs#1{\@autodefs#1\autodefs@end}
\def\@autodefs#1{\ifx#1\autodefs@end%
  \def\autodefs@next{}%
  \else%
  \def\autodefs@next{\autodef#1\@autodefs}%
  \fi\autodefs@next}

\DeclareSymbolFont{bbold}{U}{bbold}{m}{n}
\DeclareSymbolFontAlphabet{\mathbbb}{bbold}
\newcommand{\bbDelta}{\ensuremath{\mathbbb{\Delta}}\xspace}
\newcommand{\bbone}{\ensuremath{\mathbbb{1}}\xspace}
\newcommand{\bbtwo}{\ensuremath{\mathbbb{2}}\xspace}
\newcommand{\bbthree}{\ensuremath{\mathbbb{3}}\xspace}

\mdef\delbar{\overline{\partial}}

\newcommand{\dual}[1]{D{#1}}
\newcommand{\dualr}[1]{D_r{#1}}
\newcommand{\duall}[1]{D_\ell{#1}}
\mdef\hf{\textstyle\frac12 }
\mdef\thrd{\textstyle\frac13 }
\mdef\qtr{\textstyle\frac14 }

\newcommand{\op}{^{\mathrm{op}}}

\SelectTips{cm}{}
\newdir{ >}{{}*!/-10pt/@{>}}    % extra spacing for tail arrows in XYpic
\newcommand{\pushoutcorner}[1][dr]{\save*!/#1+1.4pc/#1:(1,-1)@^{|-}\restore}
\newcommand{\pullbackcorner}[1][dr]{\save*!/#1-1.2pc/#1:(-1,1)@^{|-}\restore}

\mdef\Id{\mathrm{Id}}
\mdef\id{\mathrm{id}}
\alwaysmath{ell}
\alwaysmath{infty}
\alwaysmath{odot}
\def\frc#1/#2.{\frac{#1}{#2}}   % \frc x^2+1 / x^2-1 .
\mdef\ten{\mathrel{\otimes}}

\mdef\sqten{\mathrel{\boxtimes}}

\DeclareRobustCommand\widecheck[1]{{\mathpalette\@widecheck{#1}}}
\def\@widecheck#1#2{%
    \setbox\z@\hbox{\m@th$#1#2$}%
    \setbox\tw@\hbox{\m@th$#1%
       \widehat{%
          \vrule\@width\z@\@height\ht\z@
          \vrule\@height\z@\@width\wd\z@}$}%
    \dp\tw@-\ht\z@
    \@tempdima\ht\z@ \advance\@tempdima2\ht\tw@ \divide\@tempdima\thr@@
    \setbox\tw@\hbox{%
       \raise\@tempdima\hbox{\scalebox{1}[-1]{\lower\@tempdima\box
\tw@}}}%
    {\ooalign{\box\tw@ \cr \box\z@}}}

\DeclareMathOperator\colim{colim}

\newcommand{\too}[1][]{\ensuremath{\overset{#1}{\longrightarrow}}}
\newcommand{\ot}{\ensuremath{\leftarrow}}

\let\toto\rightrightarrows
\let\into\hookrightarrow

\mdef\we{\overset{\sim}{\longrightarrow}}
\mdef\leftwe{\overset{\sim}{\longleftarrow}}

\let\xto\xrightarrow
\let\xot\xleftarrow
\def\rightarrowtailfill@{\arrowfill@{\Yright\joinrel\relbar}\relbar\rightarrow}
\newcommand\xrightarrowtail[2][]{\ext@arrow 0055{\rightarrowtailfill@}{#1}{#2}}

\def\twoheadrightarrowfill@{\arrowfill@{\relbar\joinrel\relbar}\relbar\twoheadrightarrow}
\newcommand\xtwoheadrightarrow[2][]{\ext@arrow 0055{\twoheadrightarrowfill@}{#1}{#2}}

\def\slashedarrowfill@#1#2#3#4#5{%
  $\m@th\thickmuskip0mu\medmuskip\thickmuskip\thinmuskip\thickmuskip
   \relax#5#1\mkern-7mu%
   \cleaders\hbox{$#5\mkern-2mu#2\mkern-2mu$}\hfill
   \mathclap{#3}\mathclap{#2}%
   \cleaders\hbox{$#5\mkern-2mu#2\mkern-2mu$}\hfill
   \mkern-7mu#4$%
}
\def\rightslashedarrowfill@{%
  \slashedarrowfill@\relbar\relbar\mapstochar\rightarrow}
\newcommand\xslashedrightarrow[2][]{%
  \ext@arrow 0055{\rightslashedarrowfill@}{#1}{#2}}
\mdef\hto{\xslashedrightarrow{}}
\mdef\htoo{\xslashedrightarrow{\quad}}

\def\toiso{\xto{\smash{\raisebox{-.5mm}{$\scriptstyle\sim$}}}}

\long\def\my@drawfill#1#2;{%
\@skipfalse
\fill[#1,draw=none] #2;
\@skiptrue
\draw[#1,fill=none] #2;
}
\newif\if@skip
\newcommand{\skipit}[1]{\if@skip\else#1\fi}
\newcommand{\drawfill}[1][]{\my@drawfill{#1}}

\newif\ifhyperref
\@ifpackageloaded{hyperref}{\hyperreftrue}{\hyperreffalse}
\iftac
  \let\your@state\state
  \def\state#1{\gdef\currthmtype{#1}\your@state{#1}}
  \let\your@staterm\staterm
  \def\staterm#1{\gdef\currthmtype{#1}\your@staterm{#1}}
  \let\defthm\newtheorem
  \def\currthmtype{}
  \ifhyperref
    \def\autoref#1{\ref*{label@name@#1}~\ref{#1}}
  \else
    \def\autoref#1{\ref{label@name@#1}~\ref{#1}}
  \fi
  \AtBeginDocument{%
    \let\old@label\label%
    \def\label#1{%
      {\let\your@currentlabel\@currentlabel%
        \edef\@currentlabel{\currthmtype}%
        \old@label{label@name@#1}}%
      \old@label{#1}}
  }
\else
  \ifhyperref
    \def\defthm#1#2{%
      \newtheorem{#1}{#2}[section]%
      \expandafter\def\csname #1autorefname\endcsname{#2}%
      \expandafter\let\csname c@#1\endcsname\c@thm}
  \else
    \def\defthm#1#2{\newtheorem{#1}[thm]{#2}}
    \ifx\SK@label\undefined\let\SK@label\label\fi
    \let\old@label\label
    \let\your@thm\@thm
    \def\@thm#1#2#3{\gdef\currthmtype{#3}\your@thm{#1}{#2}{#3}}
    \def\currthmtype{}
    \def\label#1{{\let\your@currentlabel\@currentlabel\def\@currentlabel%
        {\currthmtype~\your@currentlabel}%
        \SK@label{#1@}}\old@label{#1}}
    \def\autoref#1{\ref{#1@}}
  \fi
\fi

\newtheorem{thm}{Theorem}[section]

\defthm{cor}{Corollary}
\defthm{prop}{Proposition}
\defthm{lem}{Lemma}
\defthm{sch}{Scholium}
\defthm{assume}{Assumption}
\defthm{claim}{Claim}
\defthm{conj}{Conjecture}
\defthm{hyp}{Hypothesis}
\defthm{fact}{Fact}
\iftac\theoremstyle{plain}\else\theoremstyle{definition}\fi
\defthm{defn}{Definition}
\defthm{notn}{Notation}
\iftac\theoremstyle{plain}\else\theoremstyle{remark}\fi
\defthm{rmk}{Remark}
\defthm{eg}{Example}
\defthm{egs}{Examples}
\defthm{ex}{Exercise}
\defthm{ceg}{Counterexample}
\defthm{warn}{Warning}

\def\thmqedhere{\expandafter\csname\csname @currenvir\endcsname @qed\endcsname}

\setitemize[1]{leftmargin=2em}
\setenumerate[1]{leftmargin=*}

\iftac
  \let\c@equation\c@subsection
\else
  \let\c@equation\c@thm
\fi
\numberwithin{equation}{section}

\@ifpackageloaded{mathtools}{\mathtoolsset{showonlyrefs,showmanualtags}}{}

\alwaysmath{alpha}
\alwaysmath{beta}
\alwaysmath{gamma}
\alwaysmath{Gamma}
\alwaysmath{delta}
\alwaysmath{Delta}
\alwaysmath{epsilon}
\mdef\ep{\varepsilon}
\alwaysmath{zeta}
\alwaysmath{eta}
\alwaysmath{theta}
\alwaysmath{Theta}
\alwaysmath{iota}
\alwaysmath{kappa}
\alwaysmath{lambda}
\alwaysmath{Lambda}
\alwaysmath{mu}
\alwaysmath{nu}
\alwaysmath{xi}
\alwaysmath{pi}
\alwaysmath{rho}
\alwaysmath{sigma}
\alwaysmath{Sigma}
\alwaysmath{tau}
\alwaysmath{upsilon}
\alwaysmath{Upsilon}
\alwaysmath{phi}
\alwaysmath{Pi}
\alwaysmath{Phi}
\mdef\ph{\varphi}
\alwaysmath{chi}
\alwaysmath{psi}
\alwaysmath{Psi}
\alwaysmath{omega}
\alwaysmath{Omega}

\makeatother

\UseAllTwocells
\tikzset{lab/.style={auto,font=\scriptsize}} % arrow labels
\usetikzlibrary{arrows}

\newcommand{\homr}[3]{#2\mathrel{\rhd_{[#1]}} #3}
\newcommand{\homl}[3]{#3 \mathrel{\lhd_{[#1]}} #2}
\newcommand{\homre}[2]{#1\mathrel{\rhd} #2}
\newcommand{\homle}[2]{#2 \mathrel{\lhd} #1}
\newcommand{\homrbare}{\rhd}
\newcommand{\homrbareop}{\rhd\op}
\newcommand{\homlbare}{\lhd}
\newcommand{\homlbareop}{\lhd\op}
\newcommand{\homrbaree}[1]{\rhd_{[#1]}}
\newcommand{\homlbaree}[1]{\lhd_{[#1]}}

\newcommand{\homreop}[2]{#1\mathrel{\rhd\op} #2}
\newcommand{\homleop}[2]{#2 \mathrel{\lhd\op} #1}

\newcommand{\tr}{\ensuremath{\operatorname{tr}}}
\newcommand{\tw}{\ensuremath{\operatorname{tw}}}
\newcommand{\D}{\sD}
\newcommand{\E}{\sE}

\newcommand{\symm}{\mathfrak{s}}

\autodefs{\cDER\cCat\cCAT\cMONCAT\bSet\cProf\bCat\cPro\cMONDER\cBICAT\ncomp\nid\sProf\ncyl\fC\fZ\fT\nhocolim\nholim\cPsNat\ncoll}
\def\cPDER{\ensuremath{\mathcal{PD}\mathit{ER}}\xspace}

\def\ho{\mathscr{H}\!\mathit{o}\xspace}

\let\oast\varoast

\let\oldboxtimes\boxtimes
\def\boxtimes{\mathrel{\oldboxtimes}}

\newcommand{\fib}{\mathsf{fib}}
\newcommand{\cof}{\mathsf{cof}}

\def\shift#1#2{{#1}^{#2}}

\def\DT#1#2#3#4#5#6#7{%
  \xymatrix{{#1} \ar[r]^-{#2} &
    {#3} \ar[r]^-{#4} &
    {#5} \ar[r]^-{#6} &
    {#7}
  }}

\def\DTl#1#2#3#4#5#6#7{%
  \xymatrix@C=3pc{{#1} \ar[r]^-{#2} &
    {#3} \ar[r]^-{#4} &
    {#5} \ar[r]^-{#6} &
    {#7}
  }}

\newsavebox{\tvabox}
\savebox\tvabox{\hspace{1mm}\begin{tikzpicture}[>=latex',baseline={(0,-.18)}]
  \draw[->] (0,.1) -- +(1,0);
  \node at (.5,0) {$\scriptscriptstyle\bot$};
  \draw[->] (1,-.1) -- +(-1,0);
  \draw[->] (1,-.2) -- +(-1,0);
\end{tikzpicture}\hspace{1mm}}
\newcommand{\tva}{\usebox{\tvabox}}

\title{The additivity of traces in monoidal derivators}
\author{Moritz Groth, Kate Ponto and Michael Shulman}
\thanks{The first author was partially supported by the Deutsche Forschungsgemeinschaft within the graduate program 
`Homotopy and Cohomology' (GRK 1150) and by the Dutch Science Foundation (NWO).  The second author was partially 
supported by NSF grant DMS-1207670. The third author was partially supported by an NSF postdoctoral fellowship and 
NSF grant DMS-1128155, and appreciates the hospitality of the University of Kentucky.
Any opinions, findings, and conclusions or recommendations expressed in this material are those of the authors and 
do not necessarily reflect the views of the National Science Foundation.}
\date{\today}
\begin{document}

\begin{abstract}
  Motivated by traces of matrices and Euler characteristics of topological spaces, we expect abstract traces in a symmetric monoidal category to be ``additive''.
  When the category is ``stable'' in some sense, additivity along cofiber sequences is a question about the interaction of stability and the monoidal structure.

  May proved such an additivity theorem when the stable structure is a triangulation, based on new axioms for monoidal triangulated categories.
  In this paper we use stable derivators instead, which are a different model for ``stable homotopy theories''.
  We define and study monoidal structures on derivators, providing a context to describe the interplay between stability 
  and monoidal structure using only ordinary category theory and universal properties.
  We can then perform May's proof of the additivity of traces in a closed monoidal stable derivator without needing extra axioms, 
  as all the needed compatibility is automatic.
\end{abstract}

\maketitle

\setcounter{tocdepth}{1}
\tableofcontents

\section{Introduction}

The Euler characteristic for finite CW complexes is
determined by its value on a point and its additivity on subcomplexes:
if $A$ is a subcomplex of $X$, then 
\[\chi(A)+\chi(X/A)=\chi(X).\]
For generalizations of the Euler characteristic, such as the Lefschetz number of an endomorphism 
of a finite CW complex, axiomatizations are more complicated (see~\cite{ab:lefschetz, b:lefschetz}); 
but a similar additivity on subcomplexes is an essential component.

The classical Euler characteristic can also be generalized to Euler characteristics for dualizable
objects in a symmetric monoidal category---and further, to traces of (twisted) endomorphisms of such, which include Lefschetz numbers and topological transfers.
In this generality, many fundamental properties are easy to prove, but additivity is not among them.
Despite this, for most specific examples, additivity remains important.

In this paper, we will show that if our symmetric monoidal category has the additional structure of 
a \emph{stable derivator} (see below), then the Euler characteristic is automatically additive, relative to the appropriate generalization of subcomplexes.
Similarly, if an endomorphism of a dualizable object $X$ restricts to a ``subcomplex'' $A$ of $X$, then its trace 
on $X$ is the sum of the traces of the induced endomorphisms of $A$ and of $X/A$.
Using a suggestive notation, we write this as
\[ \tr(\phi_A) + \tr(\phi_{X/A}) = \tr(\phi_X). \]

Essentially this same theorem, in slightly less generality, was proven by~\cite{may:traces}, who worked with triangulated categories and stable model categories.
The point of this paper is not so much the truth of the additivity theorem, which has been known since~\cite{may:traces}, 
but the fact that derivators are a convenient context in which to state it, prove it, and generalize it.
More generally, derivators are a convenient context to study the interaction between stability and monoidal structure.

The first descriptions of the interactions between stability and monoidal structure the authors are 
aware of were given by Margolis \cite{margolis:spectra} and Hovey-Palmieri-Strickland \cite{HPS:axiomatic}.  
They focused on biexactness properties of (closed) symmetric monoidal structures and the sign issue related to permutations of sphere coordinates. 
With the goal of proving additivity of traces, May \cite{may:traces} augmented the axioms of~\cite{HPS:axiomatic} 
with three more compatibility conditions relating to pushout products.
These ideas were further developed in \cite{kn:quiver}. Finally, the first author and Jan \v{S}\v{t}ov\'{i}\v{c}ek indicate in \cite{gs:tilting} a relation between the approach offered here and the one of Keller--Neeman \cite{kn:quiver}.

\subsection{Why derivators?}
\label{sec:why-derivators}

Intuitively, an additivity theorem for traces should be true in any ``stable homotopy theory''.
(The case of ordinary Euler characteristics and Lefschetz numbers corresponds to the classical stable homotopy theory of spectra.)
There are various formal axiomatizations of ``stable homotopy theories'': in addition to stable derivators, one 
may consider stable model categories, stable $(\infty,1)$-categories, or triangulated categories.
Derivators are less well-known than these other options, but we find that they have many advantages.
For the purpose of the sort of abstract arguments we will need in this paper, each of the other notions contains either \emph{too little} information 
(triangulated categories) or \emph{too much} (stable model categories and stable $(\infty,1)$-categories) when compared with derivators.

The fact that triangulated categories often contain too little information is well-known, and was demonstrated in~\cite{may:traces}.
The major contribution of \emph{ibid.}\ was to describe ``compatibility axioms'' between a triangulation and a 
monoidal structure, which hold in the homotopy category of any stable, monoidal, enriched model category, and which imply the additivity of Euler characteristics.
Unfortunately, these axioms do not imply the more general result for \emph{traces} 
(which in fact fails to be true for \emph{some} morphisms between distinguished triangles~\cite{ferrand:nonadd}).

One of the main obstacles to proving additivity is that a triangulation
is merely a ``remnant'' of the homotopy theory visible in the homotopy category, so
its axioms assert that certain objects and morphisms \emph{exist} but do not characterize them uniquely.
For instance, any morphism $A\to X$ (regarded as a ``subcomplex inclusion'') can be extended to some distinguished triangle, 
whose third object may be regarded as the ``quotient'' $X/A$, but the quotient map $X\to X/A$ is not uniquely determined.
In examples, there are particular good choices of these objects and morphisms arising from \emph{homotopy limit} 
and \emph{colimit} constructions, but their universal properties are not visible to the triangulated category.
Therefore, new axioms which build on previous ones are often very complicated,
such as May's axiom (TC5b), and this approach becomes unwieldy when generalizing further. 

Due to these problems, May proved additivity for traces (as opposed to Euler characteristics) by working 
directly with a stable monoidal model category, where homotopy limits and colimits can be constructed 
using ordinary limits and colimits along with fibrations and cofibrations.
For instance, if we represent a subcomplex inclusion by a cofibration $A\to X$, then its ``quotient'' in 
the sense of stable homotopy theory is just its ordinary category-theoretic quotient $X/A$ in the model category.
However, the fact that model categories contain too \emph{much} information for this sort of proof is also visible in May's 
argument, which requires careful management of fibrant and cofibrant replacement functors.

One way to avoid these problems is the modern theory of $(\infty,1)$-categories, where homotopy limits 
and colimits have $(\infty,1)$-categorical universal properties, with no need for fibrant and cofibrant 
replacements, and all equivalences are invertible.
This makes for a beautiful and powerful theory, but quite a complicated one, as witnessed by the 
length of~\cite{lurie:higher-topoi,lurie:ha}.
Additionally, one must deal directly with homotopy coherence, which can be quite tedious and obscure the important ideas.
% While an individual $(\infty,1)$-category has ``forgotten'' all the extra information contained 
% in a model category, the \emph{theory} of $(\infty,1)$-categories retains too much information about each individual $(\infty,1)$-category.

While these challenges are certainly not  insurmountable, \emph{derivators} provide a very effective way to sidestep them.
Derivators were developed independently (under various names) by Grothendieck~\cite{grothendieck:derivateurs}, 
Heller~\cite{heller:htpythys}, and Franke~\cite{franke:triang}, and have been studied further by Maltsiniotis~\cite{m:introder}, 
Cisinski~\cite{cisinski:idcm}, and the present authors~\cite{groth:ptstab,gps:stable}; see the website~\cite{grothendieck:derivateurs} for a comprehensive bibliography. 

A derivator is the ``minimal modification'' of a homotopy category which remedies the problem of homotopy limits and colimits.
In addition to a homotopy category, a derivator remembers information about ``homotopy coherent diagrams'': 
for each small category $A$ we have a category $\D(A)$, regarded as the homotopy category of $A$-shaped homotopy coherent diagrams in \D.
These categories are related by restriction and left and right (homotopy) Kan extension functors, 
which characterize homotopy limits and colimits by ordinary universal properties; thus May's 
compatibility \emph{axioms} for triangulated monoidal categories become \emph{lemmas} about stable monoidal derivators.
On the other hand, there are no weak equivalences or fibrant/cofibrant replacements: homotopy meaningfulness is already built in by passage to homotopy categories.
Finally, derivators require only ordinary category-theoretic machinery, and form a well-behaved 2-category in which 2-categorical constructions are homotopically meaningful.

Since model categories give rise to derivators, our results about derivators apply \textit{a fortiori} to model categories as well.
Specifically, Cisinski~\cite{cisinski:idcm} showed that the homotopy category of any model category can be enhanced to a derivator
(an easier proof for combinatorial model categories can be found in~\cite{groth:ptstab}).
In this paper we will show (\autoref{thm:mmc-cmder-noncomb}) that this extends to monoidal structures: the homotopy derivator of a (cofibrantly generated) monoidal model category is a closed monoidal derivator.
In~\cite{gps:stable} we sketched a proof that complete and cocomplete $(\infty,1)$-categories give rise to derivators; we expect that monoidal structures also carry through this construction, but that has not yet been verified.
(Note, though, that all locally presentable $(\infty,1)$-categories --- which includes most of those arising in practice --- can be presented by a combinatorial model category, and in many cases a monoidal structure can also be presented on the model-categorical level.)
But modulo this gap, we can regard derivators as merely a convenient abstract context in which to express a 
\emph{calculus of homotopy Kan extensions} that is valid in any model category or $(\infty,1)$-category.

\subsection{An outline of this paper}
\label{sec:outline}

In \S\ref{sec:derivators} we recall the definition of stable derivators, provide examples and establish notation.
Our conventions follow~\cite{gps:stable}, and we will quote many results from \textit{ibid.}\ 
and~\cite{groth:ptstab}, but most of the present paper should be readable independently.

In \S\ref{sec:monoidal-derivators} we introduce the first of the two main new definitions of this paper: 
\emph{monoidal derivators}, and more generally cocontinuous two-variable morphisms of derivators.
In \S\ref{sec:stablemonoidal} we describe the interaction of monoidal structure with suspensions and 
cofiber sequences; this corresponds to May's axioms (TC1) and the first half of (TC2).
Then from cofiber sequences we generalize to homotopy pushouts, which requires a preliminary 
discussion of ends and coends and their application to tensor products of functors in \S\ref{sec:ends-coends}.
Using this, we discuss pushout products in \S\S \ref{sec:mon-triang}--\ref{sec:rotated}, and prove May's axioms (TC3) and (TC4).

The second of our new definitions is that of a \emph{closed} monoidal derivator, and more generally that of a two-variable adjunction of derivators.
We introduce this in \S\S\ref{sec:tvas}--\ref{sec:cycling-tvas}, prove the second half of (TC2), and also prove that monoidal model categories give rise to closed monoidal derivators.
Then in \S\S\ref{sec:duality}--\ref{prof-dual} we define \emph{duality} in a closed monoidal 
derivator and in its bicategory of profunctors, and prove versions of May's axiom (TC5).
Finally, in \S\ref{sec:introadditivity} we complete the proof of the additivity of traces, closely following May's.

In the appendices we complete a few proofs that are more tedious and require familiarity 
with more of the machinery developed in~\cite{groth:ptstab,gps:stable}: in Appendix~\ref{sec:coends} 
we give two additional characterizations of coends in derivators, and in Appendix~\ref{sec:bicatunit} 
we complete the construction of the bicategory of profunctors in a monoidal derivator.

\section{Review of derivators}
\label{sec:derivators}

In this section we recall some basic definitions and facts from the theory of derivators and fix some notation and terminology.
Let \cCat and \cCAT denote the 2-categories of small and large categories, respectively.

\begin{defn}
  A \textbf{prederivator} is a 2-functor $\D\colon\cCat\op\to\cCAT$.
\end{defn}

For a functor $u\colon A\to B$, we write $u^*\colon\D(B) \to \D(A)$ for its image under the 2-functor \D, and refer to it as \textbf{restriction} along $u$.
If $\bbone$ denotes the terminal category, we call $\D(\bbone)$ the \textbf{underlying category} of \D.

We refer to the objects of $\D(A)$ as \textbf{(coherent, $A$-shaped) diagrams} in \D.
Any such $X\in\D(A)$ has an \textbf{underlying (incoherent) diagram}, which is an ordinary diagram in 
$\D(\bbone)$, i.e.\ an object of the functor category $\D(\bbone)^A$.
For each $a\in A$, the underlying diagram of $X$ sends $a$ to $a^*X$.
(Here $a$ also denotes the functor $a\colon\bbone\to A$ whose value on the object of $\bbone$ is $a$.)
We may also write $X_a$ for $a^*X$.

We will occasionally refer to a coherent diagram as having the \textbf{form} of, or \textbf{looking like}, its 
underlying diagram, and proceed to draw that underlying diagram using objects and arrows in the usual way.
It is very important to note, though, that a coherent diagram is not determined by its underlying diagram, not even up to isomorphism.

\begin{eg}
  Any (possibly large) category \bC gives rise to a \emph{represented} prederivator~$y(\bC)$ defined by
  \[ y(\bC)(A) \coloneqq \bC^A. \]
  Its underlying category is \bC itself.
\end{eg}

The preceding example is useful to think about when generalizing from ordinary categories to derivators, but the next two examples are those of primary interest.

\begin{eg}
  Suppose \bC is a category equipped with a class \bW of ``weak equivalences''; for instance, \bC could be a \emph{Quillen model category} (see e.g.~\cite{hovey:modelcats}).
  We write $\bC[\bW^{-1}]$ for the category obtained by formally inverting the weak equivalences.
  If $A$ is a small category, then $\bC^A$ also has a class of weak equivalences, namely the pointwise ones $\bW^A$.
  We then have the \emph{derived} or \emph{homotopy prederivator} $\ho(\bC)$ defined by
  \[ \ho(\bC)(A) \coloneqq (\bC^A)[(\bW^A)^{-1}]. \]
  Its underlying category is $\bC[\bW^{-1}]$.
\end{eg}

\begin{eg}
  Suppose \cC is an $(\infty,1)$-category (e.g.\ a quasicategory as in~\cite{lurie:higher-topoi,joyal}); then it has a homotopy category $h\cC$ obtained by identifying equivalent 1-morphisms.
  Any small category $A$ may be regarded as an $(\infty,1)$-category, so that we have a functor $(\infty,1)$-category $\cC^A$.
  We then have the \emph{homotopy prederivator} $\ho(\cC)$ defined by
  \[ \ho(\cC)(A) \coloneqq h(\cC^A). \]
  Its underlying category is $h\cC$.
\end{eg}

A prederivator remembers just enough information about coherent diagrams that we can characterize homotopy limits and colimits by ordinary universal properties.
In fact, it is convenient to consider all homotopy Kan extensions as well.
In ordinary category theory, Kan extensions can be constructed ``pointwise'' out of limits and colimits, but in the world of derivators this is not the case.
The closest we can come is to ask such Kan extensions to satisfy a ``pointwise-ness'' property, ensuring that the objects occurring in their underlying diagrams are given by the appropriate homotopy limits and colimits.

This leads to the definition of a \emph{derivator}: a prederivator which has all pointwise homotopy 
Kan extensions, and which moreover satisfies some straightforward stack-like properties.

\begin{defn}
  A \textbf{derivator} is a prederivator \D with the following properties.
  \begin{itemize}[leftmargin=4em]
  \item[(Der1)] $\D\colon \cCat\op\to\cCAT$ takes coproducts to products.  In particular, $\D(\emptyset)$ is the terminal category.
  \item[(Der2)] For any $A\in\cCat$, the family of functors $a^*\colon \D(A) \to\D(\bbone)$, as $a$ ranges over the objects of $A$, is jointly conservative (isomorphism-reflecting).
  \item[(Der3)] Each functor $u^*\colon \D(B) \to\D(A)$ has both a left adjoint $u_!$ and a right adjoint $u_*$.
  \item[(Der4)] For any functors $u\colon A\to C$ and $v\colon B\to C$ in \cCat, let $(u/v)$ denote their comma category, with projections $p\colon(u/v) \to A$ and $q\colon(u/v)\to B$.
    If $B=\bbone$ is the terminal category, then the canonical transformation
    \[ q_! p^* \to q_! p^* u^* u_! \to q_! q^* v^* u_! \to v^* u_! \]
    is an isomorphism.
    Similarly, if $A=\bbone$ is the terminal category, then the canonical transformation
    \[ u^* v_* \to p_* p^* u^* v_* \to p_* q^* v^* v_* \to p_* q^* \]
    is an isomorphism.
  \end{itemize}
  We say that a derivator is \textbf{strong} if it satisfies:
  \begin{itemize}[leftmargin=4em]
  \item[(Der5)] For any $A$, the induced functor $\D(A\times \bbtwo) \to \D(A)^\bbtwo$ is full and essentially surjective, 
   where $\bbtwo=(0\to 1)$ is the category with two objects and one nonidentity arrow between them.
  \end{itemize}
\end{defn}

Following the established terminology of the theory of~$(\infty,1)$-categories, we refer to the functors $u_!$ and $u_*$ in (Der3) simply as \textbf{left} and \textbf{right Kan extensions}, respectively, as opposed to calling them \emph{homotopy} Kan extensions.
This is meant to simplify the terminology and does not result in a risk of ambiguity since actual ``categorical'' Kan extensions are meaningless in the context of an abstract derivator.
Similarly, when $B$ is the terminal category $\bbone$, we call them \textbf{colimits} and \textbf{limits}.
Axiom (Der4) expresses the fact that Kan extensions can be computed ``pointwise'' from limits and colimits, as in~\cite[X.3.1]{maclane}.

Note that (Der1) and (Der3) together imply that each category $\D(A)$ has (actual) small coproducts and products, including an initial and terminal object.

\begin{egs}
  A large category \bC is both complete and cocomplete if and only if its represented prederivator $y(\bC)$ is a derivator.
  If \bC is a model category, then $\ho(\bC)$ is a derivator (see~\cite{cisinski:idcm,groth:ptstab}); 
  the functors $u_!$ and $u_*$ in $\ho(\bC)$ are the left and right derived functors, respectively, of the corresponding functors in $y(\bC)$.
  And if \cC is a complete and cocomplete $(\infty,1)$-category, then $\ho(\cC)$ is a derivator; see~\cite{gps:stable} for a sketch of a proof.
  All of these derivators are strong.
\end{egs}

\begin{egs}\label{eg:shifted}
  For any derivator \D and category $B\in\cCat$, we have a ``shifted'' derivator $\shift\D B$ defined by $\shift\D B(A) \coloneqq \D(B\times A)$ \cite{groth:ptstab}.
  This is technically very convenient: it enables us to ignore extra ``parameter'' categories $B$ by shifting them into the (universally quantified) derivator under consideration.
  Note that $\shift{y(\bC)}{B} \cong y(\bC^B)$ and $\shift{\ho(\bC)}{B} \cong \ho(\bC^B)$. 
  
  \label{eg:opposite}
  Any derivator \D has an \emph{opposite} derivator, defined by $\D\op(A) \coloneqq \D(A\op)\op$.
  Note that $y(\bC)\op = y(\bC\op)$ and $\ho(\bC\op) = \ho(\bC)\op$ and $(\shift\D B)\op = \shift{(\D\op)}{B\op}$.

  Both $\shift\D B$ and $\D\op$ are strong if \D is.
\end{egs}

In this paper we are particularly concerned with \emph{stable} derivators, which represent ``stable'' homotopy theories such as those of spectra and unbounded chain complexes.
We recall the relevant basic definitions from~\cite{groth:ptstab,gps:stable}.

\begin{defn}
  A derivator is \textbf{pointed} if $\D(\bbone)$ has a zero object (an object which is both initial and terminal).
\end{defn}

In a pointed derivator, we expect to be able to define the \emph{cofiber} of a map $f\colon x\to y$ by ``taking the pushout of $f$ along the unique map $x\to 0$'', yielding a map with domain $y$.
\begin{equation}\label{eq:cofibersquare}
  \vcenter{\xymatrix@-.5pc{
      x\ar[r]^f\ar[d] &
      y\ar[d]^{\cof(f)}\\
      0\ar[r] &
      z. \pushoutcorner
      }}
\end{equation}
This is how we would describe the operation in an ordinary category, but in a derivator we have to be more careful to keep track of coherence of diagrams.
First of all, if $f$ is a morphism in $\D(\bbone)$, then we have to ``lift'' it to an object of $\D(\bbtwo)$ --- this is what axiom (Der5) allows us to do.
Second, we have to obtain from $f\in \D(\bbtwo)$ a coherent diagram representing the span of which we want to take the pushout:
\begin{equation}\label{eq:cofiberspan}
  \vcenter{\xymatrix@-.5pc{
      x\ar[r]^f\ar[d] &
      y\\
      0
      }}
\end{equation}
To be precise, let $\Box$ denote the category $\bbtwo\times\bbtwo$, which we write as 
\[\vcenter{\xymatrix@-.5pc{
    (0,0)\ar[r]\ar[d] &
    (0,1)\ar[d]\\
    (1,0)\ar[r] &
    (1,1).
  }}\]
Let $\ulcorner$ and $\lrcorner$ denote the full subcategories $\Box \setminus \{(1,1)\}$ and  $\Box \setminus \{(0,0)\}$, respectively, with inclusions $i_\ulcorner\colon \mathord{\ulcorner} \into\Box$ and $i_\lrcorner\colon \mathord{\lrcorner} \into\Box$.

To obtain~\eqref{eq:cofiberspan} we take the right Kan extension of $f\in\D(\bbtwo)$ along the inclusion $(0,-)\colon \bbtwo \to \mathord\ulcorner$ as the top horizontal arrow.
This is called \emph{extension by zero}, and depends only on the fact that $(0,-)$ is the inclusion of a \emph{sieve} in $\ulcorner$.
We could now take the colimit of this span, i.e.\ take its left Kan extension along the unique functor $\mathord\ulcorner \to\bbone$.
However, this would produce only the object $z$ in~\eqref{eq:cofibersquare}; to obtain the whole square~\eqref{eq:cofibersquare} 
as a coherent diagram, we can instead left Kan extend along the inclusion $i_\ulcorner : \mathord\ulcorner \to\Box$.
(A commutative square that is left Kan extended from $\mathord\ulcorner$ is called \textbf{cocartesian}.)
We can then restrict along the functor $(-,1)\colon \bbtwo\to\Box$.
In other words, the \textbf{cofiber functor} $\cof\colon \D(\bbtwo)\to\D(\bbtwo)$ in a pointed derivator is the composite
\[ \D(\bbtwo) \xto{(0,-)_*} \D(\ulcorner) \xto{(i_{\ulcorner})_!} \D(\Box) \xto{(-,1)^*} \D(\bbtwo). \]

\begin{rmk}
  Of course, there are many verifications necessary in this construction, such as confirming $(0,-)_*$ 
  doesn't change the objects $x$ and $y$ and the morphism~$f$ and that the new object is a zero object; 
  that the domain of $\cof(f)$ is the same as the codomain of $f$; that the colimit of the span~\eqref{eq:cofiberspan} 
  is the same as the object~$z$ in~\eqref{eq:cofibersquare}; and so on.
  This all follows from the calculus of manipulating Kan extensions called \emph{homotopy exactness}, which is described in~\cite{groth:ptstab,gps:stable}.
  For the most part, we will not need to invoke this calculus explicitly in this paper, as it will suffice to quote results from~\cite{groth:ptstab,gps:stable}.
  The only exceptions are the proofs of properties of coends in \autoref{thm:coendadj} and the Appendices.
\end{rmk}

Dually to the cofiber, the \textbf{fiber} of $f\colon x\to y$ in a pointed derivator is obtained by pulling back along the unique map $0\to y$:
\begin{equation}\label{eq:fibersquare}
  \vcenter{\xymatrix@-.5pc{
      w\ar[r]^{\fib(f)}\ar[d] \pullbackcorner &
      x\ar[d]^{f}\\
      0\ar[r] &
      y
      }}
\end{equation}
More precisely, the fiber functor is the composite
\[ \D(\bbtwo) \xto{(-,1)_!} \D(\lrcorner) \xto{(i_{\lrcorner})_*} \D(\Box) \xto{(0,-)^*} \D(\bbtwo). \]
We have an adjunction $(\cof,\fib)\colon \D(\bbtwo) \rightleftarrows \D(\bbtwo)$.  (A commutative square that is right Kan extended from $\lrcorner$ is 
called \textbf{cartesian}.)

The \textbf{suspension} and \textbf{loop space} of an object $x\in \D(\bbone)$ are obtained by pushing out or pulling back over two copies of $0$:
\begin{equation}\label{eq:susploopsquares}
  \vcenter{\xymatrix@-.5pc{
      x\ar[r]\ar[d] &
      0\ar[d]\\
      0\ar[r] &
      \Sigma x \pushoutcorner
    }}
  \qquad
  \vcenter{\xymatrix@-.5pc{
      \Omega x \pullbackcorner \ar[r]\ar[d] &
      0\ar[d]\\
      0\ar[r] &
      x
      }}
\end{equation}
Precisely, the suspension functor is the composite
\[ \D(\bbone) \xto{(0,0)_*} \D(\mathord{\ulcorner}) \xto{(i_\ulcorner)_!} \D(\Box) \xto{(1,1)^*} \D(\bbone) \]
and dually for the loop space.
We have a further adjunction $(\Sigma,\Omega)\colon \D(\bbone) \rightleftarrows \D(\bbone)$.

\begin{rmk}\label{rmk:sign}
  If we restrict a cocartesian square as on the left of~\eqref{eq:susploopsquares} along the automorphism 
  $\sigma\colon\Box\to\Box$ which swaps $(0,1)$ and $(1,0)$, we obtain a \emph{different} cocartesian 
  square (with the same underlying diagram).  This defines an isomorphism of $\Sigma x$ which we denote by ``$-1$''  and think of as multiplication by -1.  This is explained further in \cite[\S5]{gps:stable}.
\end{rmk}

It is shown in~\cite{gps:stable} that the following are equivalent for a pointed derivator:
\begin{enumerate}
\item $\Sigma\dashv \Omega$ is an adjoint equivalence.
\item $\cof\dashv \fib$ is an adjoint equivalence.
\item A coherently commutative square (i.e.\ an object of $\D(\Box)$) is cartesian if and only if it is cocartesian.\label{item:stable3}
\end{enumerate}
A derivator satisfying these conditions is called \textbf{stable}, and a square as in~\ref{item:stable3} is called \textbf{bicartesian}.

In any pointed derivator, the cocartesian square at the left of~\eqref{eq:susploopsquares} can further be extended to a coherent diagram
\begin{equation}
  \vcenter{\xymatrix@-.5pc{
      x\ar[r]^f\ar[d] &
      y\ar[r]\ar[d]^g &
      0\ar[d]\\
      0\ar[r] &
      z\ar[r]_h &
      w,
    }}\label{eq:cofiberseq}
\end{equation}
i.e.\ an object of $\D(\boxbar)$, in which both squares are cocartesian.
A coherent diagram of this form is called a \textbf{cofiber sequence}.
By~\cite[Prop.~3.13]{groth:ptstab}, the outer rectangle is also cocartesian, hence induces an isomorphism $\Sigma x\cong w$.
In a \emph{stable} derivator, we say that the induced string of composable arrows
\begin{equation}
  \DT x f y g z {h'} {\Sigma x}\label{eq:dt}
\end{equation}
(or any string of arrows isomorphic to it) is a {\bf distinguished triangle}.
In \cite[Theorem~4.16]{groth:ptstab} it is shown that this choice of distinguished triangles in a strong, stable derivator 
\D makes $\D(\bbone)$ into a \emph{triangulated category} in the sense of Verdier.
In particular, it is \emph{additive} --- that is, finite products and coproducts coincide, and are written as direct sums 
$x\oplus y$ --- and thus we can add morphisms $f\colon x\to y$ and $g\colon x\to y$ by taking the composite
\begin{equation}
  \vcenter{\xymatrix{
      x\ar[r]^-{\Delta} \ar@(dr,dl)[rrr]_{f+g} &
      x\oplus x\ar[r]^-{f\oplus g} &
      y\oplus y\ar[r]^-{\nabla} &
      y.
      }}
\end{equation}
This way of adding morphisms is, of course, essential to make sense of the additivity of traces.

\section{Monoidal derivators}
\label{sec:monoidal-derivators}

We will start by defining monoidal prederivators.  
For the following definition, recall that a \textbf{pseudonatural transformation} $F\colon \D_1\to \D_2$ consists of functors 
$F_A\colon \D_1(A)\to \D_2(A)$ and isomorphisms 
$F_Bu^*\cong u^*F_A$ for all $u\colon A\to B$, satisfying the obvious axioms.
Prederivators form a 2-category \cPDER, whose morphisms are pseudonatural transformations and whose 2-cells are modifications.
Moreover, it is cartesian monoidal, with products computed pointwise in \cCAT.

\begin{defn}
  A \textbf{monoidal prederivator} is a pseudomonoid object in \cPDER.
\end{defn}

Thus, a monoidal prederivator \D has a product morphism $\otimes\colon \D\times \D\xto{} \D$, a unit
$\lS\colon y(\bbone)\xto{} \D$,
and coherence isomorphisms % $\alpha$, $\lambda$, and $\rho$
expressing associativity and unitality. % left and right unitality respectively.
Since the monoidal structure of \cPDER is pointwise, this is equivalent to asking that \D lift to a 2-functor valued 
in the 2-category \cMONCAT of monoidal categories and strong monoidal functors.
Similarly, we have \textbf{braided} and \textbf{symmetric} monoidal prederivators, and notions of \textbf{strong}, 
\textbf{lax}, and \textbf{colax} monoidal morphisms of monoidal prederivators.

\begin{eg}
  If \bC is a monoidal category, then its represented prederivator $y(\bC)$ is monoidal, with the pointwise product on each category $\bC^A$.
  If \D is a monoidal prederivator, so are $\D\op$ and $\shift\D B$ (since $(-)\op$ and $\shift{(-)}B$ are product-preserving endo-2-functors of \cPDER).
\end{eg}

\begin{eg}\label{thm:mmc-mpder}
  If \bC is a monoidal model category, then its homotopy category $\ho(\bC)(\bbone)$ is a monoidal category by~\cite[4.3.2]{hovey:modelcats}.
  However, it does not automatically follow that $\bC^A$ is a monoidal model category, even when it admits a model structure.
  The pointwise tensor product on $\bC^A$ does preserve weak equivalences between pointwise cofibrant diagrams, 
  and every diagram admits a pointwise weak equivalence from a pointwise cofibrant one (apply a cofibrant replacement functor of \bC pointwise).
  This is sufficient to ensure that the pointwise tensor product has a left derived functor (see e.g.~\cite{dhks:holim}).
  The associativity and unit constraints and axioms can then be constructed exactly as for \bC itself (see~\cite[4.3.2]{hovey:modelcats} or~\cite[Prop.~15.4]{shulman:htpylim}).
  Thus, $\ho(\bC)$ is a monoidal prederivator.
\end{eg}

A \emph{monoidal derivator} is not just a monoidal prederivator which is a derivator; we intend to require that the monoidal product 
is ``cocontinuous'' in each variable separately.
First recall the following definition of cocontinuity for one-variable morphisms of derivators, which is modeled 
on the classical notion of cocontinuity for functors (generalized to Kan extensions).

\begin{defn}
  A morphism of derivators $F\colon \D_1 \to \D_2$ is \textbf{cocontinuous} if for any $u\colon A\to B$ in \cCat, the canonical transformation
  \begin{equation}
    u_! F_A \to u_! F_A u^* u_! \toiso u_! u^* F_B u_! \to F_B u_!\label{eq:cocont-trans}
  \end{equation}
  is an isomorphism.
\end{defn}

The transformation \eqref{eq:cocont-trans} is the composite of unit and counit of the $u_!\dashv u^*$ adjunctions with the 
natural transformation $\alpha$ in the diagram
\[\vcenter{\xymatrix{
    \D_1(B)\ar[r]^{u^*} \ar[d]_{F_B} \drtwocell\omit{\alpha} &
    \D_1(A)\ar[d]^{F_A}\\
    \D_2(B)\ar[r]_{u^*} &
    \D_2(A).
  }}
\]
This is an instance of a general notion in 2-category theory called a \textbf{mate}, which we will use frequently.
See~\cite{ks:r2cats} for a comprehensive theory and~\cite[Appendix A]{gps:stable} for a brief summary.

\begin{warn}\label{warning:intccts-nope}
  There is an obvious generalization of cocontinuity to morphisms of two variables, but it does not work.
  Specifically, given a morphism $\oast\colon\D_1\times\D_2\to\D_3$, objects $X\in\D_1(A),\,Y\in\D_2(B),$ and a functor $u\colon A\to B$, 
  one can form the following canonical mate transformation
  \begin{equation}
    u_!(X\oast_A u^*Y) \to u_!(u^* u_! X \oast_A u^* Y) \toiso u_! u^* (u_! X \oast_B Y) \to u_!X \oast_B Y.\label{eq:int-mder-nope}
  \end{equation}
  However, this morphism will \emph{not} be an isomorphism in the typical examples.
  For example, if $u\colon\bbone\to\bbtwo$ classifies the object~0, then the domain of~\eqref{eq:int-mder-nope} is $X\oast Y_0\to X\oast Y_0$ while its target is $X\oast Y_0\to X\oast Y_1$.
\end{warn}

In order to bypass this problem and define monoidal derivators we first have to define \emph{external variants} of monoidal structures (or, more generally, morphisms of two variables).
As we have defined it, a monoidal structure on a prederivator consists only of \emph{pointwise} or \emph{internal} monoidal product functors, which in the represented case $y(\bC)$ are
\begin{align}
  \otimes_A \colon \bC^A \times \bC^A \to \bC^A
  \qquad\text{defined by}\qquad (X\otimes_A Y)_a = X_a \otimes Y_a.
\end{align}
However, in the represented case there are also \emph{external} monoidal product functors
\begin{align}
  \otimes \colon \bC^A \times \bC^B \to \bC^{A\times B}
  \qquad\text{defined by}\qquad
  (X \otimes Y)_{(a,b)} = X_a \otimes Y_b.
\end{align}
In fact, such external products exist in any monoidal prederivator.
More generally, given any morphism of prederivators $\oast \colon \D_1\times\D_2 \to \D_3$, 
with components denoted $\oast_A\colon\D_1(A) \times \D_2(A) \to \D_3(A)$, we can define a functor
\begin{equation}
  \oast \colon \D_1(A)\times \D_2(B)\xto{} \D_3(A\times B)\label{eq:external}
\end{equation}
to be the composite
\begin{equation}
  \xymatrix@C=3pc{ \D_1(A)\times \D_2(B)\ar[r]^-{\pi_B^* \times \pi_A^*}&
    \D_1(A\times B)\times \D_2(A\times B)
    \ar[r]^-{\oast_{A\times B}}&\D_3(A\times B).}\label{eq:extdef}
\end{equation}
(In general, we write $\pi_A$ to denote a projection functor in which the category $A$ is projected away.
This includes the map $\pi_A\colon A\to \bbone$ to the terminal category, but also projections such as $A\times B\to B$ and $B\times A\to B$.)

The functors~\eqref{eq:external} form a pseudonatural transformation from the 2-functor
\begin{align*}
  (\D_1 \mathrel{\overline\times} \D_2) \colon \cCat\op\times\cCat\op &\too \cCAT\\
  (A,B)&\mapsto \D_1(A) \times \D_2(B)
\end{align*}
to the 2-functor
\begin{align*}
  (\D_3 \circ\mathord\times) \colon \cCat\op\times\cCat\op &\too \cCAT\\
  (A,B) &\mapsto \D_3(A\times B).
\end{align*}
In particular, for $X\in \D_1(A)$ and $Y\in \D_2(B)$ and functors $u\colon A' \to A$ and $v\colon B'\to B$ we have isomorphisms
\begin{equation}
  (u\times v)^*(X\oast Y) \cong u^* X \oast v^* Y\label{eq:extern-psnat}
\end{equation}
which are natural with respect to morphisms in $\D_1(A)$ and  $\D_1(B)$.
Moreover, the internal product $\oast_A$ can be recovered from the external one $\oast$ as the following composite
\[ \xymatrix{ \D_1(A) \times \D_2(A) \ar[r]^-{\oast} & \D_3(A\times A) \ar[r]^-{\Delta_A^*} & \D_3(A)} \]
where $\Delta_A\colon A\to A\times A$ is the diagonal functor of $A$.
More precisely, we have the following theorem.

\begin{thm}
  For prederivators $\D_1$, $\D_2$, and $\D_3$, there is an equivalence
  \begin{equation}
    \cPDER(\D_1\times \D_2, \D_3) \simeq \cPsNat(\D_1 \mathrel{\overline\times} \D_2, \D_3 \circ \mathord\times).
  \end{equation}
\end{thm}
As above $\cPDER$ is the category of prederivators and modifications, and $\cPsNat$ is the category of natural transformations and modifications.

\begin{proof}
  Note that $\D_1 \times \D_2 \cong (\D_1\mathrel{\overline\times}\D_2)\circ\Delta$, where $\Delta\colon\cCat\op \to \cCat\op\times\cCat\op$ is the diagonal.
  Thus the desired equivalence is between pseudonatural transformations
  \begin{equation}
    \vcenter{\xymatrix@C=3pc{
        \cCat\op\times\cCat\op\ar[r]^-{\D_1 \mathrel{\overline\times} \D_2} \ar[d]_{\times} \drtwocell\omit &
        \cCAT\ar@{=}[d]\\
        \cCat\op\ar[r]_{\D_3} &
        \cCAT
      }}
    \quad\text{and}\quad
    \vcenter{\xymatrix@C=3pc{
        \cCat\op\times\cCat\op\ar[r]^-{\D_1 \mathrel{\overline\times} \D_2} &
        \cCAT\ar@{=}[d]\\
        \cCat\op\ar[r]_{\D_3} \ar[u]^{\Delta} \urtwocell\omit &
        \cCAT.
      }}
  \end{equation}
  But $\Delta$ is right 2-adjoint to $\mathord\times$ (passing from \cCat to $\cCat\op$ reverses the handedness of adjunctions), 
  so this is just a categorified version of the mate correspondence; see for instance~\cite[Prop.~3.5]{lauda:faaa}.
  (Another way to say this is that because the monoidal structure of $\cCat$ is cartesian, 
  the induced Day convolution monoidal structure on $\cPDER = [\cCat\op,\cCAT]$ coincides with its pointwise monoidal structure.)
\end{proof}

At this point, it is worth pausing to emphasize our notation.
If we have a morphism of prederivators $\oast \colon \D_1\times\D_2 \to \D_3$, then:
\begin{itemize}
\item The internal product of $X\in\D_1(A)$ and $Y\in\D_2(A)$ is denoted with a subscript as $X\oast_A Y \in\D_3(A)$, and
\item The external product of $X\in\D_1(A)$ and $Y\in\D_2(B)$ is denoted without a subscript as $X \oast Y\in\D_3(A\times B)$.
\end{itemize}
More generally, we can have operations that are partly internal and partly external, such as
\begin{equation}\label{eq:combintext}
  \D_1(A\times B) \times \D_2(B\times C) \to \D_3(A\times B\times C).
\end{equation}
(This particular operation takes $X$ and $Y$ to $(1_A\times \Delta_B\times 1_C)^*(X\oast Y)$.)
We always include subscripts for indexing categories that are treated internally and leave them off for those treated externally; 
thus~\eqref{eq:combintext} would be denoted $(X,Y)\mapsto X\oast_B Y$.
This notational convention is sufficient to determine the type of any such expression unless one of the indexing categories 
appears more than once somewhere (and in that case context usually disambiguates).

\begin{rmk}
  There are two alternative ways of viewing~\eqref{eq:combintext}.
  On the one hand, by the functoriality of shifting, $\oast$ induces a morphism of prederivators
  \begin{equation}\label{eq:paramintern}
    \shift{\D_1}{B} \times \shift{\D_2}{B} \to \shift{\D_3}{B}
  \end{equation}
  of which~\eqref{eq:combintext} is an \emph{external} component.
  On the other hand, we can lift the definition~\eqref{eq:extdef} to a morphism of prederivators
  \begin{equation}
    \xymatrix@+.5pc{ \shift{\D_1}A \times \shift{\D_2}C\ar[r]^-{\pi_C^* \times \pi_A^*}&
      \shift{\D_1}{A\times C}\times \shift{\D_2}{A\times C}
      \ar[r]^-\oast &
      \shift{\D_3}{A\times C}}\label{eq:paramextern}
  \end{equation}
  and~\eqref{eq:combintext} is the \emph{internal} component of this morphism at $B$.
\end{rmk}

\begin{rmk}\label{rmk:extern-assoc}
  In the special case of a monoidal prederivator, the associativity for $\otimes\colon\D\times\D\to\D$, 
  together with the compatibility between the restriction functors and the monoidal product, induce associativity isomorphisms for the external product
  \begin{equation}
   (X\otimes Y) \otimes Z \cong X\otimes (Y\otimes Z)\label{eq:extern-assoc}
  \end{equation}
  for $X\in\D(A)$, $Y\in \D(B)$, and $Z\in\D(C)$.

  Recall that the unit consists of functors $\lS_\bC\colon\bbone \cong  \bbone^\bC\to \D(\bC)$ and so
  we have unit isomorphisms
  \begin{equation}
   X\otimes \lS_{\bbone} \cong X \qquad\text{and}\qquad \lS_{\bbone}\otimes X \cong X\label{eq:extern-unit}
  \end{equation}
  (where in all cases we decline to notate restriction along the associativity and unit isomorphisms of \cCat).
  Because $\lS_{\bbone}$ behaves as a unit object for all of \D in this way, we often write it simply as $\lS$.

  These isomorphisms satisfy appropriate versions of the axioms for a monoidal category.
  Conversely, from a coherent family of isomorphisms~\eqref{eq:extern-assoc} and~\eqref{eq:extern-unit} 
  we can reconstruct coherent associativity and unitality isomorphisms for the internal components making \D a monoidal prederivator.
  (See~\cite{shulman:frbi} for a general theorem along these lines.)
  The more general operations such as~\eqref{eq:combintext} are similarly associative and unital.
\end{rmk}

We can now define two-variable cocontinuity correctly.

\begin{defn}
  A two-variable morphism of derivators $\oast \colon \D_1\times\D_2 \to \D_3 $ is \textbf{cocontinuous in the first variable} if the canonical mate-transformation
  \begin{equation}
    \small (u\times 1)_! (X \oast Y) \to
    (u\times 1)_! (u^* u_! X \oast Y) \toiso
    (u\times 1)_! (u\times 1)^* (u_! X \oast Y) \to
    u_! (X) \oast Y\label{eq:coctsmate}
  \end{equation}
  is an isomorphism for all $X\in \D_1(A)$, $Y\in\D_2(C)$, and $u\colon A\to B$.
\end{defn}

By (Der2), (Der4), the functoriality of mates, and the pseudonaturality of the external product~\eqref{eq:extern-psnat}, 
it suffices for~\eqref{eq:coctsmate} to be an isomorphism when $B=\bbone$ and $C=\bbone$.
There is of course a dual notion, and a combined notion of $\oast$ being \textbf{cocontinuous in each variable} (separately).

\begin{defn}
  A \textbf{monoidal derivator} is a monoidal prederivator that is a derivator and whose product
  $\otimes \colon \D\times \D \xto{}\D$
  is cocontinuous in each variable.
  A monoidal derivator is \textbf{braided} or \textbf{symmetric} if and only if it is so as a monoidal prederivator.
\end{defn}

\begin{eg}\label{eg:repr-2vccts}
  If $\oast\colon \bC_1\times \bC_2 \to \bC_3$ is an ordinary two-variable functor between complete and cocomplete categories, 
  then for $B=C=\bbone$, a diagram $X\in(\bC_1)^A$, and an object $Y\in\bC_2 = (\bC_2)^\bbone$, the transformation~\eqref{eq:coctsmate} 
  is easily identified with the canonical map
  \[ \colim_A(X\otimes Y) \too (\colim_A X) \otimes Y\]
  and similarly for all left Kan extensions.
  Thus, the induced two-variable morphism $y(\bC_1) \times y(\bC_2) \to y(\bC_3)$ is cocontinuous in each variable, 
  as defined above, if and only if the original functor $\oast$ was so, in the ordinary sense.
  In particular, if~\bC is a complete and cocomplete monoidal category, then $y(\bC)$ is a monoidal derivator 
  if and only if the tensor product of~\bC is cocontinuous in each variable in the usual sense.
\end{eg}

\begin{eg}\label{thm:mmc-mder}
  If $\oast\colon \bC_1\times \bC_2 \to \bC_3$ is a two-variable Quillen left adjoint, then for any cofibrant object $Y\in\bC_2$, 
  the induced functor $(-\oast Y)$ is left Quillen, hence induces a cocontinuous morphism of derivators.
  Thus, the induced two-variable morphism of derivators is cocontinuous in each variable.
  In particular, if \bC is a monoidal model category, then $\ho(\bC)$ is a monoidal derivator.
\end{eg}

\begin{egs}
  If \D is a monoidal derivator, then so is $\shift{\D}{B}$.  
  However, the tensor product of $\D\op$ need not be cocontinuous in each variable.
  This fails already in the represented case: colimits in $\D\op$ are limits in $\D$, and $\otimes$ need not (and rarely does) preserve limits in either variable.
\end{egs}

\section{Stable monoidal derivators}
\label{sec:stablemonoidal}

In \S\ref{sec:derivators} we recalled the definition of stable derivators, and in the previous section we defined monoidal derivators.
Most of this paper will be spent showing that any stable \emph{and} monoidal 
derivator displays  the structure we expect from examples such as the categories of spectra
and unbounded chain complexes.
Our proof of additivity closely follows May's, so we will eventually show that a stable monoidal derivator satisfies versions of all of his axioms.
In this section we begin with the first two. 

From now on, let $\D$ be a derivator that is both monoidal and stable.
The following is a version for non-symmetric monoidal categories of one of the compatibility conditions of~\cite{HPS:axiomatic}, which in turn is equivalent to May's (TC1).

\begin{thm}\label{thm:better-tc1}
  For $x,y\in\D(\bbone)$, there are natural isomorphisms
  \begin{equation}
    \Sigma x \otimes y \;\toiso\; \Sigma(x\otimes y) \;\toiso\; x\otimes \Sigma y\label{eq:tc1-isos}
  \end{equation}
  which commute with the associativity and unit isomorphisms in an obvious way.
  If \D is symmetric, they also commute with the symmetry, in the sense that the following square commutes
  \begin{equation}\label{eq:tc1-symm}
    \vcenter{\xymatrix@-.5pc{
        \Sigma (x\otimes y)\ar[r]\ar[d] &
        \Sigma x\otimes y\ar[d]\\
        \Sigma (y\otimes x)\ar[r] &
        y\otimes \Sigma x.
      }}
  \end{equation}
  Moreover, the induced composite
  \begin{equation}
    \Sigma\Sigma(x\otimes y) \toiso \Sigma(\Sigma x\otimes y)
    \toiso \Sigma x \otimes \Sigma y
    \toiso \Sigma (x\otimes \Sigma y)
    \toiso \Sigma \Sigma(x\otimes y)\label{eq:tc1-betterswap}
  \end{equation}
  is multiplication by $-1$.
\end{thm}

\begin{proof}
  By definition, $\Sigma y$ comes with a cocartesian square of the form
  \[\vcenter{\xymatrix@-.5pc{
      y\ar[r]\ar[d] &
      0\ar[d]\\
      0\ar[r] &
      \Sigma y.
    }}\]
  Applying the external product $\D(\bbone) \times \D(\Box) \to \D(\Box)$ to $x$ and this square, we obtain a square of the form
  \[\vcenter{\xymatrix@-.5pc{
      x\otimes y\ar[r]\ar[d] &
      0\ar[d]\\
      0\ar[r] &
      x\otimes \Sigma y.
    }}\]
  We have used the fact that the tensor product preserves zero objects since it is cocontinuous in each variable.
  Cocontinuity also implies that this square is again cocartesian, hence induces the second isomorphism in~\eqref{eq:tc1-isos}; the first is similar.
  Commutativity with the associativity, unit, and symmetry isomorphisms follows since this entire construction was functorial and these isomorphisms are natural.

  For the second part, note that an isomorphism $\Sigma\Sigma z \cong w$ can be determined by giving a coherent diagram of the form
  \begin{equation}
    \vcenter{\xymatrix{
        z\ar[r]\ar[d] &
        0\ar[d]\\
        0\ar[r] &
        \bullet \ar[r] \ar[d] &
        0 \ar[d]\\
        & 0 \ar[r] & w
      }}
  \end{equation}
  in which both squares are cocartesian.
  Now from $x$ and $y$ we can obtain cocartesian squares
  \begin{equation}
    \vcenter{\xymatrix@-.5pc{
        x\ar[r]\ar[d] &
        0_2\ar[d]\\
        0_1\ar[r] &
        \Sigma x
      }}
    \qquad\text{and}\qquad
    \vcenter{\xymatrix@-.5pc{
        y\ar[r]\ar[d] &
        0_4\ar[d]\\
        0_3\ar[r] &
        \Sigma y
      }}
  \end{equation}
  which we denote $X,Y\in\D(\Box)$.
  (All objects denoted $0_k$ are zero objects; we have given them subscripts to aid in telling them apart.)
  The external product $X\otimes Y \in \D(\Box\times \Box)$ is a hypercube, which we have shown as the solid arrows in \autoref{fig:tc1-hypercube}.
  (We have abbreviated $a\otimes b$ by $a.b$ for conciseness.)
  \begin{figure}
    \centering
    \begin{tikzpicture}[->,xscale=1.6]
      \node (xy) at (0.4,7.2) {$x.y$};
      \node (o2y) at (6.4,7.2) {$0_2.y$};
      \node (o1y) at (1,6) {$0_1.y$};
      \node (sxy) at (7,6) {$\Sigma x.y$};
      \node (xo4) at (2.2,5) {$x.0_4$};
      \node (o2o4) at (4.2,5) {$0_2.0_4$};
      \node (o1o4) at (2.8,3.8) {$0_1.0_4$};
      \node (sxo4) at (4.8,3.8) {$\Sigma x.0_4$};
      \node (xsy) at (2.2,3) {$x.\Sigma y$};
      \node (o2sy) at (4.2,3) {$0_2.\Sigma y$};
      \node (o1sy) at (2.8,1.8) {$0_1.\Sigma y$};
      \node (sxsy) at (4.8,1.8) {$\Sigma x.\Sigma y$};
      \node (xo3) at (0.4,1.2) {$x.0_3$};
      \node (o2o3) at (6.4,1.2) {$0_2.0_3$};
      \node (o1o3) at (1,0) {$0_1.0_3$};
      \node (sxo3) at (7,0) {$\Sigma x.0_3$};
      \draw (xy) -- (o2y); \draw (o2y) -- (sxy);
      \draw (xy) -- (o1y);
      \draw (xo4) -- (o2o4); \draw (o1o4) -- (sxo4);
      \draw (xo4) -- (o1o4); \draw (o2o4) -- (sxo4);
      \draw (xsy) -- (o2sy); \draw (o1sy) -- (sxsy);
      \draw (xsy) -- (o1sy); \draw (o2sy) -- (sxsy);
      \draw (xo3) -- (o2o3); \draw (o1o3) -- (sxo3);
      \draw (xo3) -- (o1o3); \draw (o2o3) -- (sxo3);
      \draw (xy) -- (xo4); \draw (o2y) -- (o2o4); 
      \draw (xo3) -- (xsy); \draw (o2o3) -- (o2sy); 
      \draw (xy) -- (xo3); \draw (o2y) -- (o2o3);
      \draw (sxy) -- (sxo3);
      \draw (xo4) -- (xsy); \draw (o2o4) -- (o2sy);
      \node (o5) at (0.6,4.5) {$0_5$};
      \node (o6) at (1.8,6.7) {$0_6$};
      \begin{scope}[densely dotted]
        \draw (xy) -- (o5); \draw (o5) -- (o1y); \draw (o5) -- (xo3);
        \draw (xy) -- (o6); \draw (o6) -- (o2y); \draw (o6) -- (xo4);
      \end{scope}
      \draw[white,line width=5pt,-] (o1y) -- (sxy);     \draw (o1y) -- (sxy);  
      \draw[white,line width=5pt,-] (o1y) -- (o1o4);    \draw (o1y) -- (o1o4); 
      \draw[white,line width=5pt,-] (sxy) -- (sxo4);    \draw (sxy) -- (sxo4); 
      \draw[white,line width=5pt,-] (sxo3) -- (sxsy);   \draw (sxo3) -- (sxsy);
      \draw[white,line width=5pt,-] (o1o3) -- (o1sy);   \draw (o1o3) -- (o1sy);
      \draw[white,line width=5pt,-] (o1y) -- (o1o3);    \draw (o1y) -- (o1o3); 
      \draw[white,line width=5pt,-] (o1o4) -- (o1sy);   \draw (o1o4) -- (o1sy);
      \draw[white,line width=5pt,-] (sxo4) -- (sxsy);   \draw (sxo4) -- (sxsy);
    \end{tikzpicture}
    \caption{The hypercube for (TC1)}
    \label{fig:tc1-hypercube}
  \end{figure}
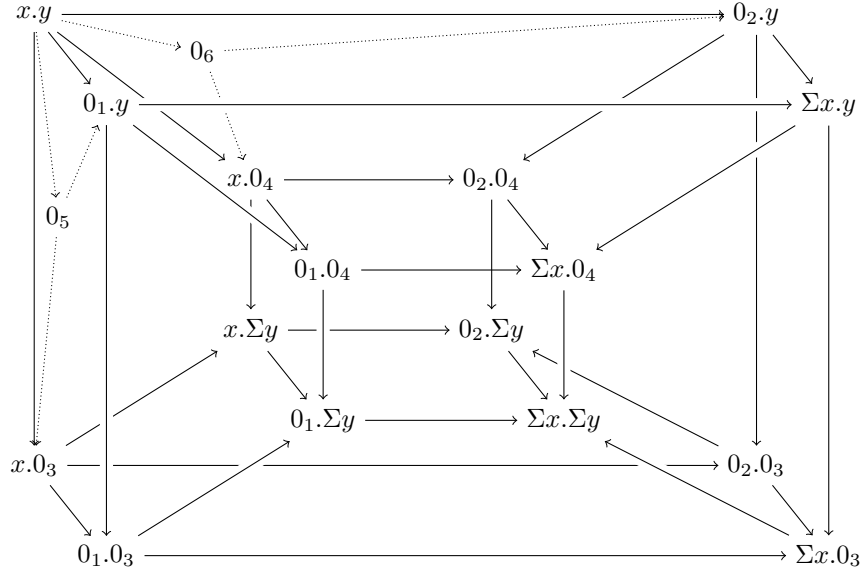
  Of course, since $\otimes$ preserves zero objects in each variable, there are a lot of zero objects in this diagram.
  However, we have notated them all as tensor products rather than merely as ``$0$'', since for purposes of identifying minus signs 
  (see \autoref{rmk:sign}) it is important to distinguish between squares and their transposes.

  Now by restriction from \autoref{fig:tc1-hypercube}, we obtain the two coherent diagrams shown in \autoref{fig:tc1-1a}.
  \begin{figure}
    \centering
    \begin{equation}\label{eq:tc1-1}
      \vcenter{\xymatrix@-.5pc{
          x. y\ar[r]\ar[d] &
          0_2. y\ar[d]\\
          0_1. y\ar[r] &
          \Sigma x. y \ar[r] \ar[d] &
          \Sigma x . 0_4 \ar[d]\\
          & \Sigma x . 0_3 \ar[r] &
          \Sigma x . \Sigma y
        }}
      \qquad
      \vcenter{\xymatrix@-.5pc{
          x. y\ar[r]\ar[d] &
          x. 0_4\ar[d]\\
          x. 0_3\ar[r] &
          x. \Sigma y \ar[r]\ar[d] &
          0_2 . \Sigma y \ar[d]\\
          & 0_1 . \Sigma y \ar[r] &
          \Sigma x . \Sigma y.
        }}
    \end{equation}
    \caption{Two ways that $\Sigma x\otimes\Sigma y$ is $\Sigma\Sigma(x\otimes y)$}
    \label{fig:tc1-1a}
  \end{figure}
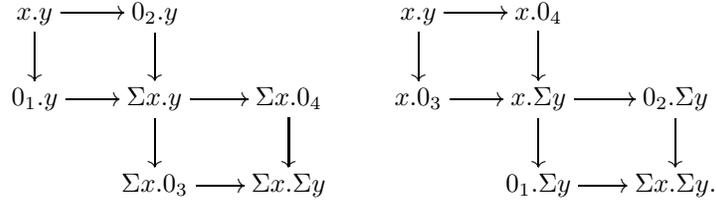
  These two diagrams induce, respectively, the composite of the first two arrows in~\eqref{eq:tc1-betterswap} and the composite of the last two arrows therein.
  Thus, it will suffice to show that these two diagrams become isomorphic after transposing one of the squares in one of them, 
  by an isomorphism which restricts to the identity on $x.y$ and $\Sigma x.\Sigma y$.

  Let $A$ denote the category $\Box\times \Box$ extended with two new objects $5$ and $6$ as shown with the dotted arrows in \autoref{fig:tc1-hypercube}, 
  with inclusion $j\colon \Box\times\Box \into A$.
  Let $Z = j_*(X\otimes Y)$; then by the dual of Franke's lemma, a detection lemma for cartesian squares \cite[Proposition~3.10]{groth:ptstab}, we conclude that the squares
  \begin{equation}
    \vcenter{\xymatrix@-.5pc{
        0_5\ar[r]\ar[d] &
        0_1.y\ar[d]\\
        x.0_3\ar[r] &
        0_1.0_3
      }}
    \qquad\text{and}\qquad
    \vcenter{\xymatrix@-.5pc{
        0_6\ar[r]\ar[d] &
        0_2.y\ar[d]\\
        x.0_4\ar[r] &
        0_2.0_4
      }}
  \end{equation}
  are cartesian, so that the objects labeled $0_5$ and $0_6$ in $Z$ are zero objects.

  Now by restriction from $Z$ along a suitable functor, we obtain a coherent diagram looking like the solid arrows in \autoref{fig:tc1-1}, with the middle object $\Sigma(x. y)$ also omitted
  (but with the composite arrows from each of $0_5$ and $0_6$ to each of $0_2.0_3$ and $0_1.0_4$ included, so that the shape is a full subcategory of the shape of the whole of \autoref{fig:tc1-1}).
  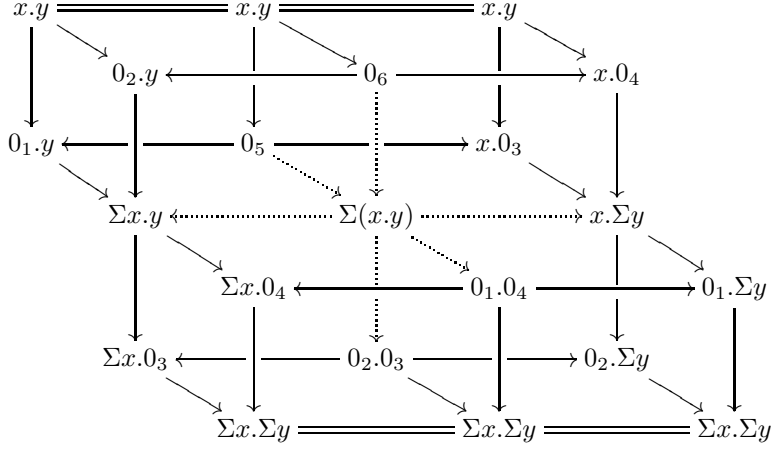
\begin{figure}
    \centering
    \[\xymatrix@-1pc{
      x. y \ar[dd] \ar[dr] &&
      x. y \ar@{=}[ll] \ar@{=}[rr] \ar'[d][dd] \ar[dr] &&
      x. y \ar[dr] \ar'[d][dd]\\
      & 0_2. y \ar[dd] &&
      0_6 \ar[ll] \ar[rr] \ar@{.>}[dd] &&
      x. 0_4 \ar[dd]\\
      0_1. y\ar[dr] &&
      0_5 \ar'[l][ll] \ar'[r][rr] \ar@{.>}[dr] &&
      x. 0_3  \ar[dr]\\
      & \Sigma x. y \ar[dd] \ar[dr] &&
      \Sigma(x. y) \ar@{.>}[ll] \ar@{.>}[rr] \ar@{.>}[dr] \ar@{.>}'[d][dd] &&
      x. \Sigma y \ar'[d][dd] \ar[dr]\\
      && \Sigma x. 0_4 \ar[dd] &&
      0_1. 0_4 \ar[ll] \ar[rr] \ar[dd] &&
      0_1. \Sigma y \ar[dd]\\
      & \Sigma x. 0_3 \ar[dr] &&
      0_2. 0_3 \ar'[l][ll] \ar'[r][rr] \ar[dr] &&
      0_2. \Sigma y \ar[dr] \\
      && \Sigma x. \Sigma y &&
      \Sigma x. \Sigma y \ar@{=}[ll] \ar@{=}[rr] &&
      \Sigma x. \Sigma y
      }\]
    \caption{The minus sign in (TC1)}
    \label{fig:tc1-1}
  \end{figure}
  Finally, we perform a left Kan extension to add the object denoted $\Sigma(x.y)$ and the dotted arrows connecting to it.
  Since all the solid horizontal arrows are isomorphisms, and the left, middle, and right faces are composed of cocartesian squares, the horizontal dotted arrows are also isomorphisms.
  However, the left and right faces of \autoref{fig:tc1-1} are the left- and right-hand diagrams in~\autoref{fig:tc1-1a}, with one square transposed on the right side.
  This completes the proof.
\end{proof}

In the symmetric case, we easily deduce May's (TC1).
Recall that $\lS$ means $\lS_{\bbone}$.

\begin{cor}[TC1]\label{thm:tc1}
  Suppose \D is symmetric.
  Then there is a natural isomorphism $\alpha\colon \Sigma x\toiso x \otimes \Sigma\lS$ such that the composite
  \begin{equation}\label{eq:tc1-swap}
    \xymatrix{\Sigma\Sigma\lS \ar[r]^-{\alpha} &
      \Sigma\lS \otimes \Sigma\lS \ar[r]^-{\symm}&
      \Sigma\lS \otimes \Sigma\lS \ar[r]^-{\alpha^{-1}}&
      \Sigma\Sigma\lS
    }
  \end{equation}
  is multiplication by $-1$, where $\symm$ denotes the symmetry isomorphism.
\end{cor}
\begin{proof}
  Take $y=\lS$ and let $\alpha$ be the second isomorphism in~\eqref{eq:tc1-isos} (composed with a unit isomorphism).
  Naturality of the unit isomorphism and~\eqref{eq:tc1-symm} then implies that the composite
  \begin{equation}
    \Sigma\lS \otimes x \xto{\symm} x\otimes \Sigma\lS \xto{\alpha^{-1}} \Sigma x\label{eq:tc1-other}
  \end{equation}
  is an instance of the first isomorphism in~\eqref{eq:tc1-isos} (again composed with a unit isomorphism).
  In fact, we can recover the general case of~\eqref{eq:tc1-isos} from $\alpha$ and~\eqref{eq:tc1-other}, using associativity and unitality, e.g.
  \[ \Sigma(x\otimes y) \xto{\alpha} (x\otimes y) \otimes \Sigma \lS
  \toiso x\otimes (y\otimes \Sigma\lS) \xto{\alpha^{-1}} x\otimes\Sigma y.
  \]
  (May used this as a \emph{definition} of the isomorphisms~\eqref{eq:tc1-isos} in terms of his $\alpha$.)
  Finally, using~\eqref{eq:tc1-symm} we can identify~\eqref{eq:tc1-swap} with~\eqref{eq:tc1-betterswap}.
\end{proof}

In fact, in \autoref{thm:better-tc1} we did not need stability of \D, only pointedness.
The next axiom could also be stated in a pointed version by using cofiber sequences, but for simplicity we use the notation of the stable case.

\begin{thm}[part of (TC2)]\label{thm:tc2}
  For any distinguished triangle \[\DT x f y g z h {\Sigma x}\]
  in $\D(\bbone)$ and any $w\in\D(\bbone)$, each of the following triangles is distinguished.
  \begin{gather*}
    \DT{x\otimes w}{f\otimes 1}{y\otimes w}{g\otimes 1}{z\otimes w}{h\otimes 1}{\Sigma(x\otimes w),}\\
    \DT{w\otimes x}{1\otimes f}{w\otimes y}{1\otimes g}{w\otimes z}{1\otimes h}{\Sigma(w\otimes x).}
  \end{gather*}
\end{thm}
\begin{proof}
  By assumption, we have a cofiber sequence
  \[\vcenter{\xymatrix@-.5pc{
      x\ar[r]\ar[d] &
      y\ar[r]\ar[d] &
      0\ar[d]\\
      0\ar[r] &
      z\ar[r] &
      \Sigma x.
    }}\]
  Taking the external product with $w$ on both sides, and arguing as in the proof of \autoref{thm:tc1}, we obtain the triangles above.
\end{proof}

See \autoref{thm:tc2b} for the rest of May's axiom TC2.

\section{Ends, coends, and the tensor product of functors}
\label{sec:ends-coends}

Theorems \ref{thm:better-tc1} and \ref{thm:tc2} describe some compatibility between a stable and monoidal structure, 
but these results are only the beginning of the structure displayed in well-known examples.
In \S\S\ref{sec:mon-triang}--\ref{sec:rotated} we will describe the relationship between stability and the \emph{pushout product}.
As preparation, in this section we introduce the canceling tensor product of diagrams, of which the pushout product is a special case.

As described in \S\ref{sec:monoidal-derivators},  a monoidal derivator has both an external product and an internal product. 
However, if \bC is a \emph{cocomplete} monoidal category, there is a third type of operation induced on the values of its represented prederivator: 
the \emph{(canceling) tensor product of functors}
\begin{equation}\label{eq:tpf}
  \otimes_{[A]}\colon\bC^{A\op}\times \bC^{A} \to \bC
  \qquad\text{defined by}\qquad
  X \otimes_{[A]} Y \coloneqq \int^{a\in A} X_a \otimes Y_a.
\end{equation}
This requires not only the monoidal structure but the notion of \emph{coend} in \bC (denoted by an integral sign with a variable at the top).

In classical category theory, there are several ways to define coends.
For a generalization to derivators, we choose one which is particularly easy to work with.
In Appendix~\ref{sec:coends}, however, we will show that our definition agrees with two other natural ones.

Let $A$ be a small category; its \textbf{twisted arrow category} $\tw(A)$ is defined to be the category of elements of the hom-functor $A(-,-)\colon A\op\times A\to \bSet$.
Thus, its objects are morphisms $a\xto{f} b$ in $A$, while its morphisms from $a_1\xto{f_1} b_1$ to $a_2\xto{f_2} b_2$ 
are pairs of morphisms $b_1 \xto{h} b_2$ and $a_2 \xto{g} a_1$ such that $f_2 = h f_1 g$ (that is, ``two-sided factorizations'' of $f_2$ through $f_1$).
It comes with a projection
\[(s,t)\colon\tw(A) \to A\op\times A,\]
where for $a\xto{f} b$ we have $s(f) = a$ and $t(f) = b$.
We will also be interested in its opposite category $\tw(A)\op$, which of course comes with a projection
\[(t\op,s\op)\colon\tw(A)\op \to A\op\times A.\]

\begin{defn}\label{def:coend}
  If \D is a derivator, then the \textbf{coend} of $X\in \D(A\op\times A)$ is defined to be
  \[ \int^A X \coloneqq (\pi_{\tw(A)\op})_! (t\op,s\op)^* X, \]
  which is an object of the underlying category $\D(\bbone)$.
  Dually, the \textbf{end} of $X$, denoted $\int_A X$, is its coend in $\D\op$.
\end{defn}

Put differently, the coend functor $\int^A$ is the composite
\begin{equation}
  \D(A\op\times A) \xto{(t\op,s\op)^*} \D(\tw(A)\op)  \xto{(\pi_{\tw(A)\op})_!} \D(\bbone).
\end{equation}
This definition also makes sense ``with parameters'', i.e.\ if $X\in \D(A\op\times A\times B)$, then we have $\int^A X \in\D(B)$.
The naturality in parameters is then phrased concisely by defining the coend as the following morphism of derivators
\begin{equation}
  \shift{\D}{A\op\times A} \xto{(t\op,s\op)^*} \shift{\D}{\tw(A)\op} \xto{(\pi_{\tw(A)\op})_!} \D.\label{eq:paramcoend}
\end{equation}

When $\D=y(\bC)$ is a represented derivator, it is easy to verify that this definition reproduces the usual notion of coend.

We have the usual Fubini-type theorem (which again has obvious variants with parameters).

\begin{lem}\label{lem:Fubini}
  Let \D be a derivator,
  \[ \symm \colon (A\op\times A)\times (B\op\times B) \xto{\cong} (A\times B)\op\times (A\times B) \]
  be the canonical isomorphism, and  $X\in\D((A\times B)\op\times (A\times B))$.
  Then there are natural isomorphisms
  \[\int^A\int^B \symm^\ast X\cong \int^{A\times B}X\cong \int^B\int ^A \symm^\ast X. \]
\end{lem}
\begin{proof}
  This follows immediately from the observation that there is a canonical isomorphism between $\tw(A\times B)$ and $\tw(A)\times \tw(B)$, 
  which is compatible with the source and target maps in the sense that the following diagram commutes
  \[
   \raisebox{\depth-\height}{\xymatrix{
   \tw(A)\op\times \tw(B)\op\ar[d]\ar[r]&\tw(A\times B)\op\ar[d]\\
   A\op\times A\times B\op\times B \ar[r]_-{\symm}&  (A\times B)\op\times (A\times B).
   }}\qedhere 
  \]
\end{proof}

We also mention an ``adjointness'' property which will be useful later.

\begin{lem}\label{thm:coendadj}
  Let $f\colon A\to B$.
  \begin{enumerate}
  \item For $X\in\D(B\op\times A)$, we have a natural isomorphism\label{item:coendadj1}
    \begin{equation}
      \int^A (f\op\times 1)^* X \;\cong\; \int^B (1\times f)_! X.
    \end{equation}
  \item For $X\in\D(A\op\times B)$, we have a natural isomorphism\label{item:coendadj2}
    \begin{equation}
      \int^A (1\times f)^* X\;\cong\; \int^B (f\op\times 1)_! X.
    \end{equation}
\end{enumerate}
\end{lem}
\begin{proof}
  We prove~\ref{item:coendadj1};~\ref{item:coendadj2} is analogous.
  Since $(t\op,s\op)\colon \tw(B)\op \to B\op\times B$ is a fibration, the pullback square below is homotopy exact.
  \[\vcenter{\xymatrix@C=3pc{
      \tw(A)\op \ar[r]^{(t\op,s\op)} \ar@(d,l)[ddr] \ar@{.>}[dr]
      &A\op\times A \ar[dr]^{f\op\times 1}\\
      &\tw(f)\op\ar[r]\ar[d] \pullbackcorner &
      B\op\times A\ar[d]^{1\times f}\\
      &\tw(B)\op\ar[r]_{(t\op,s\op)} \ar[d] &
      B\op\times B\\
      & \bbone
    }}\]
  Thus, by unraveling definitions, it will suffice to show that the induced functor $\tw(A)\op \to \tw(f)\op$ is homotopy final.
  Using the characterization of homotopy exact squares in \cite[Theorem 3.15]{gps:stable} 
  and the definition of $\tw(f)\op$ as a pullback, this is equivalent to asking that for any $a\in A$, $b\in B$, and $\phi\colon f a \to b$, 
  the following category $K_{a,b,\phi}$ is homotopy contractible:
  \begin{itemize}[leftmargin=2em]
   \item Its objects consist of a pair of objects $a_1,a_2\in A$, morphisms $\alpha\colon a\to a_1$ and $\xi\colon a_1\to a_2$ in $A$, 
    and $\psi\colon  f a_2 \to b$ in $B$, such that $\psi \circ f\xi \circ f\alpha = \phi$.
  \item A morphism from $(a_1,a_2,\alpha,\xi,\psi)$ to $(a_1',a_2',\alpha',\xi',\psi')$ consists of a pair of morphisms 
   $\zeta_1\colon a_1 \to a_1'$ and $\zeta_2\colon a_2' \to a_2$ in $A$ such that $\zeta_1 \circ \alpha = \alpha'$, $\psi \circ f \zeta_2 = \psi'$, and $\xi = \zeta_2 \circ \xi' \circ \zeta_1$.
  \end{itemize}
  Let $L_{a,b,\phi}$ be the full subcategory of $K_{a,b,\phi}$ whose objects have $a_1=a$ and $\alpha = 1_a$.
  Then $L_{a,b,\phi}$ is coreflective; the coreflection of $(a_1,a_2,\alpha,\xi,\psi)$ is $(a,a_2,1_a,\xi\circ\alpha,\psi)$.
  Thus, since the coreflection is a right adjoint and hence homotopy final, $K_{a,b,\phi}$ is homotopy contractible if and only if $L_{a,b,\phi}$ is.

  However, $L_{a,b,\phi}$ has a terminal object, namely $(a,a,1_a,1_a,\phi)$.
  Thus, since the inclusion $\bbone \to L_{a,b,\phi}$ of this terminal object is a right adjoint and $\bbone$ is homotopy contractible, so is $L_{a,b,\phi}$.
\end{proof}

We will be particularly interested in taking coends of diagrams of the form $X\oast Y$, where 
$\oast\colon \D_1\times \D_2 \to \D_3$ is a two-variable morphism, and $X\in\D_1(B\op)$ and $Y\in\D_2(B)$.
In the represented case, this yields~\eqref{eq:tpf}.
We denote such a \textbf{(canceling) tensor product of functors} by
\begin{equation}\label{eq:hoten}
  X \oast_{[B]} Y \coloneqq \int^B (X\oast Y).
\end{equation}
The brackets around $[B]$ indicate that it is ``canceled'', no longer appearing at all as an indexing category in the result.
Functorially, $\oast_{[B]}$ is the composite
\begin{equation}\label{eq:hotenfunctor}
  \D_1(B\op) \times \D_2(B) \xto{\oast} \D_3(B\op\times B) \xto{\int^B} \D_3(\bbone).
\end{equation}

The canceling tensor product over one category $B$ can also be combined with internal or external products over other categories.
For instance, we have a functor
\begin{equation}
  \D_1(A\times B\op) \times \D_2(A\times B) \to \D_3(A)\label{eq:combintcanc}
\end{equation}
which is internal in $A$ and canceling in $B$; thus we write it $X \oast_{A,[B]} Y$.
As with~\eqref{eq:combintext}, we can view~\eqref{eq:combintcanc} either as the tensor product~\eqref{eq:hotenfunctor} over $B$ constructed from the induced morphism
\begin{equation}
  \shift{\D_1}{A} \times \shift{\D_2}{A} \to \shift{\D_3}{A}
\end{equation}
or as the internal component at $A$ of the lifted version of~\eqref{eq:hotenfunctor}
\begin{equation}\label{eq:tensor}
  \xymatrix{\shift{\D_1}{B\op} \times \shift{\D_2}{B} \ar[r]^-{\oast} &
    \shift{\D_3}{B\op\times B} \ar[r]^-{\int^B} & \D_3. }
\end{equation}
Similarly, we have a functor
\begin{equation}
  \D_1(A\times B\op) \times \D_2(B\times C) \too \D_3(A\times C)\label{eq:combextcanc}
\end{equation}
denoted $(X,Y) \mapsto X\oast_{[B]} Y$, in which $A$ and $C$ are treated externally.
We leave it to the reader to write down the most general form with all three of internal, external, and canceled indexing categories. 

\autoref{thm:coendadj} specializes to give an adjointness property for the tensor product.

\begin{cor}\label{thm:tpadj}
  Let $f\colon A\to B$, and let $\oast\colon  \D_1 \times \D_2 \to \D_3$ be cocontinuous in both variables.
  \begin{enumerate}
  \item If $X\in\D_1(B\op)$ and $Y\in \D_2(A)$, we have a natural isomorphism\label{item:tpadj1}
    \begin{equation}
      (f\op)^* X \oast_{[A]} Y \;\cong\; X \oast_{[B]} f_! Y.
    \end{equation}
  \item If $X\in\D_1(A\op)$ and $Y\in \D_2(B)$, we have a natural isomorphism\label{item:tpadj2}
    \begin{equation}
      X \oast_{[A]} f^* Y \;\cong\; (f\op)_! X \oast_{[B]} Y.
    \end{equation}
  \end{enumerate}
\end{cor}

\begin{proof}
  Pseudonaturality of the external product $\oast$ implies
  \[(f\op)^*X \oast Y \cong (f\op\times 1)^*(X\oast Y),\]
  while its cocontinuity in the second variable implies
  \[X\oast f_! Y \cong (1\times f)_! (X\oast Y).\]
  Thus~\ref{item:tpadj1} follows from \autoref{thm:coendadj}\ref{item:coendadj1}.
  Similarly,~\ref{item:tpadj2} follows from \autoref{thm:coendadj}\ref{item:coendadj2}.
\end{proof}

We remarked in \S\ref{sec:monoidal-derivators} that associativity and unitality of a monoidal structure on a derivator 
can equivalently be expressed in terms of the internal or the external products.
The corresponding notion of ``associativity'' and ``unitality'' for the canceling tensor product of functors can be 
concisely expressed using a \emph{bicategory of profunctors}, as follows.

\begin{thm}\label{thm:bicategory}
  If \D is a monoidal derivator, then there is a bicategory $\cProf(\D)$ described as follows:
  \begin{itemize}
  \item Its objects are small categories.
  \item Its hom-category from $A$ to $B$ is $\D(A\times B\op)$.
  \item Its composition functors are the external-canceling tensor products
    \[ \otimes_{[B]} \colon \D(A\times B\op) \times \D(B\times C\op) \too \D(A\times C\op). \]
  \item The identity 1-cell of a small category $B$ is
    \begin{equation}
      \lI_B\;=\;(t,s)_! \lS_{\tw(B)} \;\cong\; (t,s)_! (\pi_{\tw(B)})^* \lS_{\bbone} \; \in \D(B\times B\op).\label{eq:unit}
    \end{equation}
  \end{itemize}
\end{thm}

Note that in the definition of $\lI_B$ we use $\tw(B)$, rather than its opposite as we did in \autoref{def:coend}.

\begin{rmk}
  On the one hand, the definition of $\cProf(\D)$ can be regarded as a simple generalization of the usual 
  bicategory of profunctors enriched in a monoidal category (with the objects restricted to unenriched categories rather than enriched ones).
  Indeed, the tensor product of functors is exactly how classical profunctors are composed.
  It is likewise straightforward to verify that
  $(\lI_B)_{(b_2,b_1)} \in\D(\bbone)$ is a coproduct of $B(b_1,b_2)$ copies of the unit object $\lS_\bbone \in \D(\bbone)$, 
  so that in the represented case $\D=y(\bC)$, this definition of $\lI_B$ agrees with the usual identity profunctor of (the free \bC-enriched category on) the ordinary category~$B$.
\end{rmk}

\begin{rmk}
  On the other hand, the construction of $\cProf(\D)$ can also be regarded as analogous to the construction of a bicategory 
  from an indexed monoidal category (a.k.a.\ monoidal fibration) in~\cite{shulman:frbi,ps:indexed}.
  Indeed, a monoidal derivator \emph{is} an indexed monoidal category, where the indexing is over the cartesian monoidal category $\cCat$.
  However, instead of the Beck-Chevalley condition for \emph{pullback} squares, as assumed in \emph{ibid.}, 
  a monoidal derivator satisfies the Beck-Chevalley condition for \emph{comma} squares, by (Der4).
  This is a manifestation of the fact that the objects of \cCat are ``directed''.
  As a consequence, we must replace the diagonal maps $\Delta_A\colon A\to A\times A$ used in \emph{ibid.}\ 
  with either $(t,s)\colon\tw(B) \to B\times B\op$, as used in the definition of the identity 1-cells, or $(t\op,s\op)\colon\tw(B)\op \to B\op\times B$, as used in the definition of composition.
  Similarly, $\pi_A\colon A\to \bbone$ is replaced by $\pi_{\tw(A)}$ or $\pi_{\tw(A)\op}$.
\end{rmk}

What remains for the proof of \autoref{thm:bicategory} is to construct coherent associativity and unitality isomorphisms.
Associativity is easy.

\begin{lem}\label{thm:bicatassoc}
  For any diagrams
  \[X\in \D(A\times B\op), \qquad Y \in \D(B\times C\op), \qquad \text{and} \qquad Z \in \D(C\times D\op), \]
  we have a natural isomorphism
  \[ X\otimes_{[B]} (Y\otimes_{[C]} Z) \;\cong\; (X\otimes_{[B]} Y)\otimes_{[C]} Z \]
  which satisfies the pentagon identity.
\end{lem}

\begin{proof}
  Since the external product $\otimes$ preserves colimits in each variable, it also preserves coends; thus we have a natural isomorphism
  \[ (X\otimes_{[B]} Y) \otimes Z \;=\; \left(\int^B X\otimes Y \right) \otimes Z
  \;\cong\; \int^B( (X\otimes Y) \otimes Z). \]
  Therefore, we have a natural isomorphism
  \begin{equation}
   (X\otimes_{[B]} Y) \otimes_{[C]} Z \;=\; \int^C (X\otimes_{[B]} Y) \otimes Z \;\cong\; \int^C\int^B( (X\otimes Y) \otimes Z).\label{eq:assoc23a}
  \end{equation}
  Similarly, we have 
  \begin{equation}
   X\otimes_{[B]} (Y \otimes_{[C]} Z) \;\cong\; \int^B\int^C( X\otimes (Y \otimes Z)).\label{eq:assoc23b}
  \end{equation}
  Thus, composing~\eqref{eq:assoc23a} and~\eqref{eq:assoc23b} with the associativity isomorphism of $\otimes$ 
  and \autoref{lem:Fubini}, %the Fubini isomorphism %$\int^B\int^C \cong \int^C \int^B$ (), 
  we obtain a natural isomorphism
  \[ (X\otimes_{[B]} Y) \otimes_{[C]} Z \;\cong\;X\otimes_{[B]} (Y \otimes_{[C]} Z). \]
  The pentagon identity for this associativity isomorphism follows from the pentagon identity for $\otimes$, 
  together with the pasting coherence properties of the mates that define the   isomorphisms~\eqref{eq:assoc23a} and~\eqref{eq:assoc23b} and the Fubini isomorphisms.
\end{proof}

\begin{rmk}\label{rmk:tpadj-assoc}
  We also have the following compatibility between the associativity and the isomorphisms of \autoref{thm:tpadj}.
  Suppose $X\in \D(A\times B\op)$, $Y\in \D(B\times C\op)$, and $Z\in \D(C'\times D\op)$, while $f\colon C\to C'$.
  Then the following square of isomorphisms commutes:
  \begin{equation}
  \vcenter{\xymatrix{
      (1\times f\op)_!(X \otimes_{[B]} Y) \otimes_{[C']} Z \ar[r]\ar[d] &
      (X\otimes_{[B]} Y) \otimes_{[C]} (f\times 1)^* Z\ar[dd]\\
      \big(X\otimes_{[B]} (1\times f\op)_! Y\big) \otimes_{[C']} Z \ar[d]\\
      X\otimes_{[B]} \big((1\times f\op)_! Y \otimes_{[C']} Z\big) \ar[r] &
      X\otimes_{[B]} \big(Y \otimes_{[C]} (f\times 1)^* Z\big)
      }}
  \end{equation}
  and similarly in other analogous situations.
  We leave the proof to the reader.
\end{rmk}

The construction of the unitality isomorphism and its coherence is rather involved, so we defer it to Appendix~\ref{sec:bicatunit}.
However, we note for future reference that it actually admits the following generalization.

\begin{lem}\label{thm:bicatunit}
  For any diagrams
  \[X \in \D(A\times B\op) \qquad\text{and}\qquad Y\in \D(C) \]
  we have a natural isomorphism
  \[ X\otimes_{[B]} \big((t,s)_! (\pi_{\tw(B)})^* Y\big) \cong X\otimes Y. \]
\end{lem}

The ordinary unitality isomorphism is the special case when $Y = \lS_{\bbone}$.
Of course, there is also a dual version.

\section{Pushout products in stable monoidal derivators}
\label{sec:mon-triang}

One of the primary uses we have for the tensor product of functors is the construction of pushout products.  
Given diagrams $X\in \D(\bbtwo\op)$ and $X'\in \D(\bbtwo)$ of the form $(y\xot{f} x)$ and $(x'\xto{f'} y')$
their external product $X\otimes X' \in \D(\bbtwo\op\times\bbtwo)$ looks like
\[\vcenter{\xymatrix@-.5pc{
    y\otimes x'\ar@{<-}[r]\ar[d] &
    x\otimes x'\ar[d]\\
    y\otimes y'\ar@{<-}[r] &
    x\otimes y'
  }}\]
while its restriction to $\tw(\bbtwo)\op$ has the form
\begin{equation}
  \vcenter{\xymatrix@-.5pc{
      y\otimes x'\ar@{<-}[r] &
      x\otimes x'\ar[d]\\
      & x\otimes y'.
    }}\label{eq:vdef}
\end{equation}
Thus, the canceling tensor product $X\otimes_{[\bbtwo]} X'$ is the pushout of this cospan, which is the usual definition of a \emph{pushout product} in a monoidal category.

In a stable derivator, this pushout product and the various maps to and from it fit into several distinguished triangles.   
This compatibility is May's axiom (TC3).  
Of course, in a triangulated category the construction above is unavailable, 
so May's version of the following axiom asserts only that there is some object $v$, without specifying it further.

\begin{thm}[TC3]\label{thm:tc3}
  Suppose we have $X=(y\xot{f} x)$ and $X'=(x'\xto{f'} y')$, giving rise to distinguished triangles
  \begin{gather}
    \DT x f y g z h {\Sigma x,}\\
    \DT {x'} {f'} {y'} {g'} {z'} {h'} {\Sigma x'.}
  \end{gather}
  Then if we write $v\coloneqq X\otimes_{[\bbtwo]} X'$, there exist morphisms $p_1$, $p_2$, $p_3$, $j_1$, $j_2$, and $j_3$, along with cofiber sequences giving rise to distinguished triangles
  \begin{gather}
    \DT{y\otimes x'}{p_1}{v}{j_1}{x\otimes z'}{f\otimes h'}{\Sigma(y\otimes x'),}\\
    \DT{\Sigma^{-1}(z\otimes z')}{p_2}{v}{j_2}{y\otimes y'}{-g\otimes g'}{z\otimes z',}\\
    \DT{x\otimes y'}{p_3}{v}{j_3}{z\otimes x'}{h\otimes f'}{\Sigma(x\otimes y'),}
  \end{gather}
  and a coherent diagram of the form of \autoref{fig:tc3}.
\end{thm}

May's version of (TC3) follows from this when \D is strong, since then we can lift any pair of morphisms in $\D(\bbone)$ to objects of $\D(\bbtwo\op)$ and $\D(\bbtwo)$.

\begin{figure}
  \centering
   \begin{equation}\label{eq:tc3}
    \xymatrix@R=3pc@C=5pc{
      \Sigma^{-1}(y\otimes z')
      \ar[dr]^(.7){\Sigma^{-1}(g\otimes 1_{z'})}
      \ar@(dl,ul)[d]_{\Sigma^{-1}(1_y\otimes h')} &
      x\otimes x' \ar[dl]_(.3){f\otimes 1_{x'}} \ar[dr]^(.3){1_x\otimes f'} &
      \Sigma^{-1}(z\otimes y')
      \ar[dl]_(.7){\Sigma^{-1}(1_z\otimes g')}
      \ar@(dr,ur)[d]^{\Sigma^{-1}(h\otimes 1_{y'})} \\
      y\otimes x'
      \ar@(dl,ul)[dd]_{g\otimes 1_{x'}}
      \ar[ddr]_(.3){1_y\otimes f'}
      \ar[dr]^(.4){p_1} &
      \Sigma^{-1}(z\otimes z')
      \ar[ddl]_(.7){\Sigma^{-1}(1_z\otimes h')}
      \ar[d]^{p_2}
      \ar[ddr]^(.7){\Sigma^{-1}(h\otimes 1_{z'})} &
      x\otimes y'
      \ar[dl]_(.4){p_3}
      \ar[ddl]^(.3){f\otimes 1_{y'}}
      \ar@(dr,ur)[dd]^{1_x\otimes g'} \\
      & v \ar[dl]^(.6){j_3} \ar[d]^{j_2} \ar[dr]_(.6){j_1} \\
      z\otimes x' \ar@(dl,ul)[d]_{1_z\otimes f'} \ar[dr]^(.3){h\otimes 1_{x'}} &
      y\otimes y' \ar[dl]^(.3){g\otimes 1_{y'}} \ar[dr]_(.3){1_y\otimes g'} &
      x\otimes z' \ar[dl]_(.3){-1_x\otimes h'} \ar@(dr,ur)[d]^{f\otimes 1_{z'}} \\
      z\otimes y' &
      \Sigma(x\otimes x') &
      y\otimes z'
    }
  \end{equation}
  \caption{The (TC3) diagram from~\cite{may:traces}}
  \label{fig:tc3}
\end{figure}
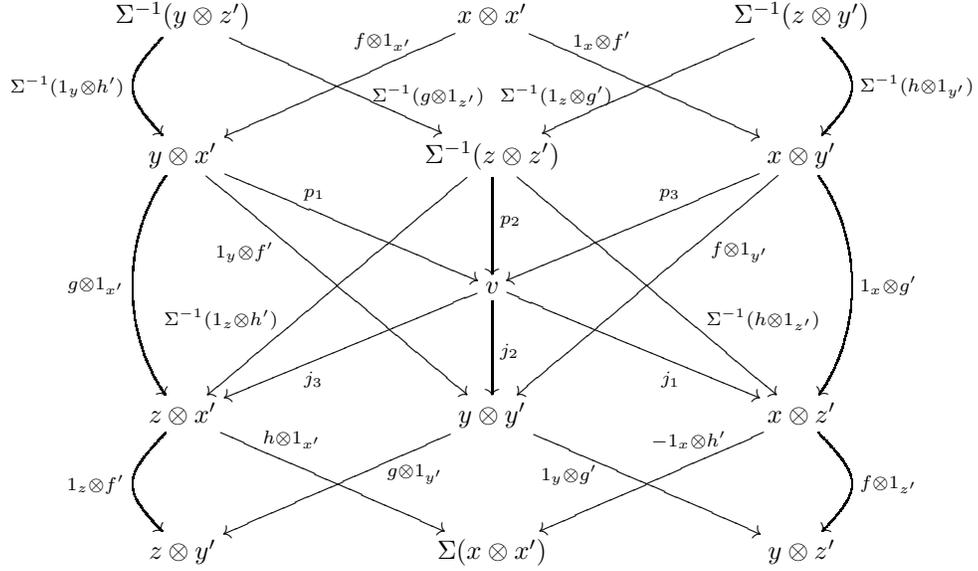

\begin{figure}
  \centering
   \begin{equation}\label{eq:tc3pf}
    \xymatrix@R=3pc{
      (0_2 \ot \Sigma^{-1} y) \ar[dr] \ar@(dl,ul)[d] &
      (x \ot 0_1) \ar[dl] \ar[dr] &
      (\Sigma^{-1}z\ot \Sigma^{-1}z) \ar[dl] \ar@(dr,ur)[d] \\
      (y\ot 0_1) \ar@(dl,ul)[dd] \ar[ddr] \ar[dr] &
      (0_2 \ot \Sigma^{-1}z) \ar[ddl] \ar[d] \ar[ddr] &
      (x\ot x) \ar[dl] \ar[ddl] \ar@(dr,ur)[dd] \\
      & (y \ot x) \ar[dl] \ar[d] \ar[dr] \\
      (z\ot 0_3) \ar@(dl,ul)[d] \ar[dr] &
      (y\ot y) \ar[dl] \ar[dr] &
      (0_4 \ot x) \ar[dl] \ar@(dr,ur)[d] \\
      (z\ot z) &
      (\Sigma x\ot 0_3) &
      (0_4 \ot y)
    }
  \end{equation}
  \caption{Diagram from the proof of \autoref{thm:tc3}}
  \label{fig:tc3pf}
\end{figure}
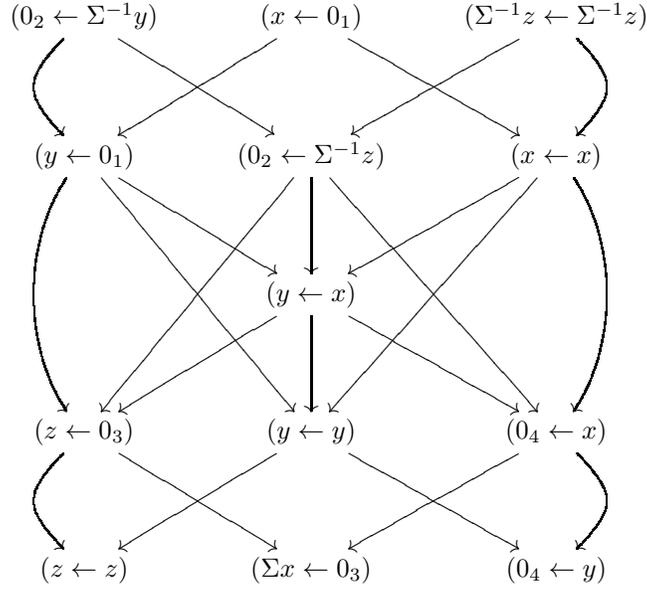

\begin{rmk}
  Since this theorem is our first significant use of many of the properties of derivators in this paper, we have given a very detailed and explicit proof.
  In later proofs that make use of similar techniques, we will leave more details for the reader to fill in.
\end{rmk}

\begin{proof} 
  %The basic idea of the following proof is to break the symmetry between $X$ and $X'$.
  We will construct from $X$ an object $U\in \D(C\times \bbtwo\op)$, where $C$ is the poset that is the shape of \autoref{fig:tc3}, 
  such that the tensor product $U\otimes_{[\bbtwo]} X'\in\D(C)$ is the desired \autoref{fig:tc3} itself.
  Replacing $X'$ in the diagram above with  the two ``representable'' $\bbtwo$-diagrams
  \[ \lS \to \lS \qquad\text{and}\qquad 0 \to \lS, \]
  Yoneda-type arguments suggest that $U$ should look like \autoref{fig:tc3pf} 
  (as in earlier proofs, each object denoted $0_k$ is a zero object).
%
%  Yoneda-type arguments suggest that the object $U$ should be determined by its tensor products with the two ``representable'' $\bbtwo$-diagrams
%  \[ \lS \to \lS \qquad\text{and}\qquad 0 \to \lS. \]
%  By substituting these two diagrams for $x'\xto{f'} y'$ in \autoref{fig:tc3}, we can conclude that that $U$ should look like \autoref{fig:tc3pf} 
%  (as in earlier proofs, each object denoted $0_k$ is a zero object).

  We can obtain such an object $U\in \D(C\times \bbtwo\op)$ starting from 
  %by restricting a known object along some functor out of $C\times \bbtwo\op$.
  %Recall that by definition of the triangulation, we have a diagram
  the diagram 
  \begin{equation}
    \vcenter{\xymatrix{
        x\ar[r]^f\ar[d] &
        y\ar[r]\ar[d]^g &
        0\ar[d]\\
        0 \ar[r] &
        z \ar[r]_h &
        \Sigma x
      }}
  \end{equation}
  in which both squares are bicartesian.
  Using left and right extensions by zero and Kan extension, we can extend this to the diagram of the form shown in \autoref{fig:A}.
  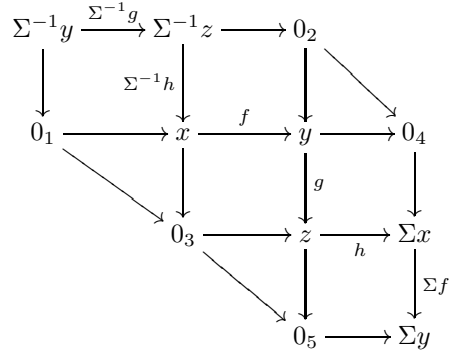
\begin{figure}
    \centering
    \begin{equation}
      \label{eq:A}
      \vcenter{\xymatrix{
          \Sigma^{-1} y \ar[r]^{\Sigma^{-1} g }\ar[d] &
          \Sigma^{-1} z\ar[r]\ar[d]_{\Sigma^{-1}h} &
          0_2 \ar[d] \ar[dr]\\
          0_1 \ar[r]\ar[dr] &
          x\ar[r]^f\ar[d] &
          y\ar[r]\ar[d]^g &
          0_4\ar[d]\\
          & 0_3 \ar[r] \ar[dr] &
          z \ar[r]_h \ar[d] &
          \Sigma x\ar[d]^{\Sigma f}\\
          && 0_5 \ar[r]
          & \Sigma y
        }}
    \end{equation}
    \caption{The origin of $U$}
    \label{fig:A}
  \end{figure}
  Again, all the objects labeled $0_k$ are zero objects.
  The lack of minus signs means that we have identified the objects labeled $\Sigma^{-1} y$, $\Sigma^{-1} z$, $\Sigma x$, and $\Sigma y$ in \autoref{fig:A} using the respective bicartesian squares
  \begin{equation}\label{eq:Asquares}
    \vcenter{\xymatrix{
        \Sigma^{-1} y\ar[r]\ar[d] &
        0_2\ar[d]\\
        0_1\ar[r] &
        y,
      }}
    \qquad
    \vcenter{\xymatrix{
        \Sigma^{-1} z\ar[r]\ar[d] &
        0_2\ar[d]\\
        0_3\ar[r] &
        z,
      }}
    \qquad
    \vcenter{\xymatrix{
        x\ar[r]\ar[d] &
        0_4\ar[d]\\
        0_3\ar[r] &
        \Sigma x,
      }}
    \qquad
    \vcenter{\xymatrix{
        y\ar[r]\ar[d] &
        0_4\ar[d]\\
        0_5\ar[r] &
        \Sigma y
      }}
  \end{equation}
  and not their transposes (see \autoref{rmk:sign}).

  Now, if $A$ is the shape of \autoref{fig:A}, then there is a functor $q\colon C\times \bbtwo\op\to A$ such that restricting \autoref{fig:A} along $q$ produces the diagram in \autoref{fig:tc3pf}.
  The subscripts indicate which zero objects in \autoref{fig:tc3pf} map to which zero objects in \autoref{fig:A}, and this fact uniquely determines $q$ since its domain and codomain are posets.
  We leave it to the reader to verify that $q$ exists.  % does actually exist. %, i.e.\ that every arrow in \autoref{fig:tc3pf} corresponds to some arrow in \autoref{fig:A}.
  %(The objects $0_5$ and $\Sigma y$ are not used yet; they will be important later.)

%  Thus, we have our $U\in\D(C\times \bbtwo\op)$ that looks like \autoref{fig:tc3pf}; we must now verify that 
  With this choice of $U$ we will show $U\otimes_{[\bbtwo]} X'$ looks like \autoref{fig:tc3}.
  By naturality of coends in parameters, for any $c\in C$, the object $(U\otimes_{[\bbtwo]} X')_c$ can be identified with $U_c \otimes_{[\bbtwo]} X'$. 
%  where $U_c \in \D(\bbtwo\op)$ is the arrow occurring at the appropriate spot in $U$.
  Thus, for instance, the object at the top middle of $U\otimes_{[\bbtwo]} X'$ is the tensor product of $(x\ot 0)\in \D(\bbtwo\op)$ with $X' = (x'\to y') \in\D(\bbtwo)$.
  By the argument above, this is the pushout of
  \[\vcenter{\xymatrix@-.5pc{
      x\otimes x'\ar@{<-}[r] &
      0\otimes x' \mathrlap{\;\cong 0}\ar[d]\\
      & 0\otimes y'\mathrlap{\;\cong 0}
    }}\]
  which is canonically isomorphic to $x\otimes x'$. %, as it should be.
  The same argument identifies the objects $y\otimes x'$, $z\otimes x'$, and $\Sigma(x\otimes x') \cong (\Sigma x) \otimes x'$ of $U\otimes_{[\bbtwo]} X'$.

  Next, consider the object at the bottom left-hand corner of $U\otimes_{[\bbtwo]} X'$: the tensor product of $(z\ot z)$ with $(x'\to y')$.
  This is the pushout of
  \[\vcenter{\xymatrix@C=1.5pc{
      z\otimes x'\ar@{<-}[r] &
      z\otimes x' \ar[d]^{1_z \otimes f'}\\
      & z\otimes y',
    }}\]
  which is canonically isomorphic to $z\otimes y'$ by \cite[Prop.~3.12(2)]{groth:ptstab}.
  The same argument identifies the objects $y\otimes y'$, $x\otimes y'$, and $\Sigma^{-1}(z\otimes y')$.

  Now consider the object at the bottom right-hand corner of $U\otimes_{[\bbtwo]} X'$: the tensor product of $(0\ot y)$ with $(x'\to y')$.
  This is the pushout of
  \[\vcenter{\xymatrix@C=1.5pc{
      \mathllap{0\cong\;} 0\otimes x'\ar@{<-}[r] &
      y\otimes x' \ar[d]^{1_y \otimes f'}\\
      & y\otimes y',
    }}\]
  which is to say the cofiber of $1_y \otimes f'$.
  But $(y\otimes -)$ is cocontinuous, and $z'$ is the cofiber of $f'$, so this is canonically isomorphic to $y\otimes z'$.
  The same argument identifies the objects $x\otimes z'$, $\Sigma^{-1}(z\otimes z')$, and $\Sigma^{-1}(y\otimes z')$.

  This completes our identification of the objects in $U\otimes_{[\bbtwo]} X'$ with their corresponding objects in \autoref{fig:tc3} --- except for $v$, 
  of course, which we simply define to be the appropriate object in $U\otimes_{[\bbtwo]} X'$.
  This is the tensor product of $(y\ot x)$ with $(x'\to y')$, which is the pushout of~\eqref{eq:vdef} as stated.

 The arrows $p_1$, $p_2$, $p_3$, $j_1$, $j_2$, and $j_3$ are defined by restriction from $U\otimes_{[\bbtwo]} X'$.
  %Now we need to identify the arrows occurring in $U\otimes_{[\bbtwo]} X'$ --- apart from $p_1$, $p_2$, $p_3$, $j_1$, $j_2$, and $j_3$, of course, which we make correct by definition.
  Naturality of coends implies that for any morphism $\gamma\colon c\to c'$ in $C$, the arrow $(U\otimes_{[\bbtwo]} X')_\gamma \in \D(\bbtwo)$ 
  can be identified with $U_\gamma \otimes_{[\bbtwo]} X'$, which is a left Kan extension along the projection $(\bbtwo \times \mathord\ulcorner) \to\bbtwo$.
  For instance, the arrow towards the top left of $U\otimes_{[\bbtwo]} X'$ %which should be $f\otimes 1_{x'}$ 
  from $x\otimes x'\to  y\otimes x'$ is obtained from the solid-arrow 
  $(\bbtwo \times\mathord\ulcorner)$-diagram below by extending to the dotted arrows:
  \[ \vcenter{\xymatrix@-.5pc{
      x\otimes x' \ar@{<-}[rr] \ar[dr]|{f\otimes 1_{x'}} \ar@{.>}[dd] &&
      0_1\otimes x' \mathrlap{\;\cong 0} \ar'[d][dd] \ar[dr] \\
      & y\otimes x' \ar@{<-}[rr] \ar@{.>}[dd]
      && 0_1\otimes x' \mathrlap{\;\cong 0} \ar[dd]\\
      x\otimes x' \ar@{<.}'[r][rr] \ar@{.>}[dr]_{f \otimes 1_{x'}}
      && 0_1\otimes y' \mathrlap{\;\cong 0} \ar[dr] \\
      & y\otimes x' \ar@{<.}[rr]
      && 0_1\otimes y' \mathrlap{\;\cong 0.}
    }}\]
  Up to isomorphism, the solid arrow diagram is obtained by restriction from a $\bbthree$-diagram
  \[ 0 \to x\otimes x' \xto {f\otimes 1_{x'}} y\otimes x' \]
  along a particular functor $\bbtwo\times\ulcorner \to \bbthree$.
  However, this functor extends to a functor $\bbtwo\times\Box \to \bbthree$, restriction along which produces 
  some diagram looking like the entire cube above, with the morphism $f\otimes 1_{x'}$ as labeled.
  By \cite[Prop.~3.12(2)]{groth:ptstab} 
  the front and back face of this cube are cocartesian, and hence by \cite[Corollary~2.6]{groth:ptstab} 
  it must be \emph{the} left Kan extension which computes the desired morphism in \autoref{fig:tc3}.
  Thus, that morphism must be $f\otimes 1_{x'}$, as desired.

  Similar arguments easily identify the morphisms  $g\otimes 1_{x'}$, $h\otimes 1_{x'}$, $\Sigma^{-1} (h\otimes 1_{y'})$, 
  $f\otimes 1_{y'}$, $g\otimes 1_{y'}$, $\Sigma^{-1}(g\otimes 1_{z'})$, $\Sigma^{-1}(h\otimes 1_{z'})$, and $f\otimes 1_{z'}$. % --- all those whose primed factor is an identity.
  In some cases, we must restrict from a $\Box$-diagram instead of a $\bbthree$-diagram, such as for $g\otimes 1_{x'}$ where we have
  \begin{equation}
    \vcenter{\xymatrix@-.5pc{
        \mathllap{0\cong\;}0_1\otimes x'\ar[r]\ar[d] &
        x\otimes x'\ar[d]^{f\otimes 1_{x'}}\\
        \mathllap{0\cong\;}0_3\otimes x'\ar[r] &
        y\otimes x'.
      }}
  \end{equation}
  But since any zero object is uniquely isomorphic to any other zero object, this makes no essential difference.

  Now consider the morphism  $1_x\otimes f'$.
  The above argument implies that it is the left Kan extension to $\bbtwo$ of the solid arrow diagram
  \begin{equation}
  \vcenter{\xymatrix@-.5pc{
      x\otimes x' \ar@{<-}[rr] \ar[dr] \ar@{.>}[dd]&&
      0_1\otimes x' \mathrlap{\;\cong 0} \ar'[d][dd] \ar[dr] \\
      & x\otimes x' \ar@{<-}[rr]\ar@{.>}[dd]|(.3){1_x\otimes f'}
      && x\otimes x' \ar[dd]^{1_x\otimes f'}\\
      x\otimes x'\ar@{.>}[dr]_-{1_x\otimes f'}
      && 0_1\otimes y' \mathrlap{\;\cong 0} \ar[dr]\ar@{.>}'[l][ll] \\
      & x\otimes y'
      && x\otimes y'.\ar@{.>}[ll]
    }}\label{eq:1timesfprime}
  \end{equation}
  The solid arrow part of~\eqref{eq:1timesfprime} is obtained by restriction from a diagram of the form
  \[ 0 \too x\otimes x' \xto{1_x\otimes f'} x\otimes y'.\]
  We can obtain an entire cube of the form~\eqref{eq:1timesfprime}, with the morphism $1_x\otimes f'$ as labeled, by restricting along a further functor $\bbtwo\times\Box \to \bbthree$.
  As before that this cube is cocartesian, hence the desired morphism is $1_x\otimes f'$.
  The same argument applies for~$1_y\otimes f'$ and~$1_z\otimes f'$.

  The morphisms $\Sigma^{-1}(1_z\otimes g')$, $1_x\otimes g'$, and $1_y\otimes g'$ are likewise similar to each other.
  For~$1_y\otimes g'$, we use the following cube
  \[ \vcenter{\xymatrix@-.5pc{
      y\otimes x' \ar@{<-}[rr] \ar[dr] \ar@{.>}[dd]_{1_y\otimes f'}&&
      y\otimes x' \ar'[d][dd]^{1_y\otimes f'} \ar[dr] \\
      & \mathllap{0\cong\;} 0_4 \otimes x' \ar@{<-}[rr]\ar@{.>}[dd]
      && y\otimes x' \ar[dd]^{1_y\otimes f'}\\
      y\otimes y'\ar@{.>}[dr]_-{1_y\otimes g'}&& y\otimes y' \ar[dr]\ar@{.>}'[l][ll] \\
      & y\otimes z'&& y\otimes y'\ar@{.>}[ll]^{1_y\otimes g'}
    }}\]
  which can be obtained by restriction from the cocartesian square
  \begin{equation}
    \vcenter{\xymatrix{
        y\otimes x'\ar[r]\ar[d]_{1_y\otimes f'} &
        0_4 \otimes x' \mathrlap{\;\cong 0}\ar[d]\\
        y\otimes y'\ar[r]_{1_y\otimes g'} &
        y\otimes z'.
      }}
  \end{equation}
  The other two are analogous.

%  Now there are three morphisms left: $\Sigma^{-1}(1_y\otimes h')$, $\Sigma^{-1}(1_z\otimes h')$, and $-1_x\otimes h'$.
  For $\Sigma^{-1}(1_y\otimes h')$, the relevant cube is
  \[ \vcenter{\xymatrix@-.5pc{
      \mathllap{0\cong\;} 0_2\otimes x' \ar@{<-}[rr] \ar[dr] \ar@{.>}[dd]&&
      \Sigma^{-1}y\otimes x' \ar'[d][dd] \ar[dr] \\
      & y \otimes x' \ar@{<-}[rr]\ar@{.>}[dd]
      && 0_1\otimes x' \mathrlap{\;\cong 0}\ar[dd]\\
      \Sigma^{-1}y\otimes z'\ar@{.>}[dr]_-?
      && \Sigma^{-1}y\otimes y' \ar[dr]\ar@{.>}'[l][ll] \\
      & y\otimes x'&& 0_1\otimes y'\mathrlap{\;\cong 0.}\ar@{.>}[ll]
    }}\]
  As in all the previous cases, the back and the front faces are cocartesian.
  However, in this case  the top face is also cocartesian by definition of~$U$ as a restriction of \autoref{fig:A}.
  But then  the bottom square is cocartesian by \cite[Prop.~3.13]{groth:ptstab}. 
  Using the definition of the triangulation, we can use the back and the bottom faces in order to identify the morphism labeled by~``?'' as~$\Sigma^{-1}(1_y\otimes h')$.
  The remaining two cases $\Sigma^{-1}(1_z\otimes h')$ and $-1_x\otimes h'$ can be established in a similar way.
  
  %We have not yet explained t
  The one minus sign in \autoref{fig:tc3}
  appears on the morphism $-1_x\otimes h'$, but it should really not be regarded as fundamentally attached to that morphism; rather, 
  it is merely a marker that the square with vertices $v$, $x\otimes z'$, $z\otimes x'$, and $\Sigma(x\otimes x')$ anticommutes rather than commutes.
  In identifying the morphisms labeled $h\otimes 1_{x'}$ and $1_x\otimes h'$, we have identified their common codomain 
  with $\Sigma(x\otimes x')$ in two different ways, which differ by the transposition automorphism $\sigma$ of $\Box$.  
  To see this, let us restrict $U$ to the $(\lrcorner\times \bbtwo\op)$-shaped diagram corresponding to these two morphisms.
  After taking the external product with $X'$ and pulling back to $\tw(\bbtwo)\op$, we obtain the solid arrow diagram in \autoref{fig:tc3ac}, 
  with the morphisms of interest being obtained by left extension to the dotted arrows.
  \begin{figure}
    \centering
    \begin{equation}
      \vcenter{\xymatrix@-.5pc{
          0_3 \otimes x' \ar[rr] \ar[dd] \ar[dr]
          && 0_3 \otimes x' \ar@{<-}[rr] \ar[dr] \ar'[d][dd]
          && x\otimes x' \ar'[d][dd] \ar[dr] \\
          & z\otimes x' \ar@{.>}[dd] \ar[rr]
          && (\Sigma x) \otimes x' \ar@{<-}[rr] \ar@{.>}[dd]
          && 0_4 \otimes x'\ar@{.>}[dd]\\
          0_3 \otimes y' \ar@{.>}[dr] \ar'[r][rr]
          && 0_3\otimes y' \ar@{.>}[dr] \ar@{<-}'[r][rr]
          && x \otimes y' \ar@{.>}[dr] \\
          & z\otimes x' \ar@{.>}[rr]_{\pm h\otimes 1_{x'}}
          && \Sigma (x \otimes x') \ar@{<.}[rr]_{\pm 1_x \otimes h'}
          && x\otimes z'.
        }}\label{eq:tc3-sign-cube}
    \end{equation}
    \caption{Identifying the anticommuting square in (TC3)}
    \label{fig:tc3ac}
  \end{figure}

  When determining the morphism labeled $h\otimes 1_{x'}$, we identified the middle object with $\Sigma(x\otimes x') \cong (\Sigma x)\otimes x'$ 
  by way of the top face of the right-hand cube in \autoref{fig:tc3ac} (together with the very middle face, which is trivial).
  This square is shown again on the left below:
  \begin{equation}\label{eq:tc3-sign-squares}
    \vcenter{\xymatrix{
        x\otimes x'\ar[r]\ar[d] &
        0_4 \otimes x'\ar[d]\\
        0_3 \otimes x'\ar[r] &
        (\Sigma x) \otimes x'
      }}
    \qquad
    \vcenter{\xymatrix{
        x\otimes x'\ar[r]\ar[d] &
        x\otimes y'\ar[r]\ar[d] &
        0_3 \otimes y'\ar[d]\\
        0_4 \otimes x'\ar[r] &
        x\otimes z'\ar[r] &
        x\otimes (\Sigma x')
      }}
  \end{equation}
  It is oriented as shown in order to match~\eqref{eq:Asquares}.
  On the other hand, when determining the morphism labeled $1_x \otimes h'$, we identified the middle object with 
  $\Sigma(x\otimes x') \cong x\otimes (\Sigma x')$ using the cofiber sequence shown on the right in~\eqref{eq:tc3-sign-squares}, 
  which consists of the right and bottom faces of the right-hand cube in \autoref{fig:tc3ac}.
  Note that the two zero objects $0_3$ and $0_4$ have been transposed between the two squares in~\eqref{eq:tc3-sign-squares}.
  Thus, these two identifications of the middle object with $\Sigma(x\otimes x')$ differ by a minus sign.

  It remains to construct the distinguished triangles.
  Consider the following diagrams of shape $(\bbthree\times\bbtwo\op)$, obtained from $U$ by restriction.
  \begin{equation}
    \vcenter{\xymatrix@R=1pc@C=3pc{
        (y\ot 0_1)\ar[r] &
        (y\ot x)\ar[r] &
        (0_4 \ot x)\\
        (0_2\ot \Sigma^{-1} z)\ar[r] &
        (y\ot x)\ar[r] &
        (y\ot y)\\
        (x\ot x)\ar[r] &
        (y\ot x)\ar[r] &
        (z\ot 0_3)
      }}
  \end{equation}
  Applying the functor $(- \otimes_{[\bbtwo]} X')$ to these diagrams yields the three $\bbthree$-shaped diagrams $(\bullet \xto{p_i} \bullet \xto{j_i} \bullet)$ in \autoref{fig:tc3}.
  Since $U$ is a restriction of \autoref{fig:A}, so is each of these diagrams.
  Moreover, each one also underlies a cofiber sequence in $\shift\D{\bbtwo\op}$ which may also be obtained from \autoref{fig:A} by restriction; these cofiber sequences are shown in \autoref{fig:tc3cofibers}.
  \begin{figure}
    \begin{gather*}
    \vcenter{\xymatrix@-.5pc{
        (y\ot 0_1)\ar[r]\ar[d] &
        (y\ot x)\ar[r]\ar[d] &
        (0_5 \ot 0_5)\ar[d]\\
        (0_4 \ot 0_1)\ar[r] &
        (0_4\ot x)\ar[r] &
        (\Sigma y \ot 0_5)
      }}\\
    \vcenter{\xymatrix@-.5pc{
        (0_2 \ot \Sigma^{-1}z)\ar[r]\ar[d] &
        (y\ot x)\ar[r]\ar[d] &
        (0_5 \ot 0_3)\ar[d]\\
        (0_2 \ot 0_2)\ar[r] &
        (y\ot y)\ar[r] &
        (0_5 \ot z)
      }}\\
    \vcenter{\xymatrix@-.5pc{
        (x\ot x)\ar[r]\ar[d] &
        (y\ot x)\ar[r]\ar[d] &
        (0_4 \ot 0_4)\ar[d]\\
        (0_3 \ot 0_3)\ar[r] &
        (z\ot 0_3)\ar[r] &
        (\Sigma x\ot \Sigma x)
      }}
    \end{gather*}
    \caption{Cofiber sequences for (TC3)}
    \label{fig:tc3cofibers}
  \end{figure}
  If we tensor these cofiber sequences with $X'$, we obtain three cofiber sequences in \D that will give rise to our desired distinguished triangles.
  It remains to identify the third morphism in each such sequence.

  The morphism $x\otimes z'\to y\otimes \Sigma x'$ %that should be $f\otimes h'$ 
  is obtained by tensoring
  \begin{equation}
    (0_4\ot x) \to (\Sigma y \ot 0_5)\label{eq:tc3-p1j1}
  \end{equation}
  with $X'=(x'\to y')$.
  But~\eqref{eq:tc3-p1j1} factors as
  \begin{equation}
    (0_4\ot x) \to (\Sigma x \ot 0_3) \to (\Sigma y \ot 0_5),
  \end{equation}
  and when tensored with $X'$ these two morphisms yield $1_x\otimes h'$ and $\Sigma f \otimes 1_{x'}$, by the sort of arguments given above.
  Thus, their composite is
  \[ (\Sigma f \otimes 1_{x'}) \circ (1_x\otimes h') =
  (f \otimes 1_{\Sigma x'}) \circ (1_x\otimes h') =
  f\otimes h'.
  \]
  Similarly, $(y\ot y) \to (0_5\ot z)$ factors through $(0_5\ot y)$, yielding $1_y\otimes g'$ and $g\otimes 1_{z'}$, while $(z\ot 0_3) \to (\Sigma x\ot \Sigma x)$ factors through $(\Sigma x\ot 0_3)$, yielding $h\otimes 1_{x'}$ and $1_{\Sigma x}\otimes f'$.
  There is a minus sign in the second distinguished triangle because the outer rectangle of the second sequence in \autoref{fig:tc3cofibers} involves the square
  \begin{equation}
  \vcenter{\xymatrix@-.5pc{
      \Sigma^{-1} z\ar[r]\ar[d] &
      0_3\ar[d]\\
      0_2\ar[r] &
      z
      }}
  \end{equation}
  which is transposed from the square used in~\eqref{eq:Asquares} to identify $\Sigma^{-1}z$.
\end{proof}

Henceforth, we will use the notations $v$, $j_1$, $p_1$, etc.\ to refer to the \emph{specific} objects and morphisms 
constructed in the proof of \autoref{thm:tc3}, which come equipped with certain coherent diagrams and cocartesian squares.
We will also abuse notation by denoting the lifts $(y\xleftarrow{f}x)\in \D(\bbtwo\op)$ and $(x'\xto{f'}y')\in\D(\bbtwo)$ by $f$ and $f'$, so that we can write $v = f\otimes_{[\bbtwo]} f'$.

We stress again that $v$, $j_1$, $p_1$, etc.\ are determined up to unique canonical isomorphism as soon as we choose these objects of $\D(\bbtwo\op)$ and $\D(\bbtwo)$.
This is in contrast to the situation in~\cite[\S4]{may:traces}, where the axioms can only assert that some objects and morphisms exist, without any uniqueness.
This is the main advantage of derivators over triangulated categories.
 
The following lemma, combined with the Mayer-Vietoris sequence, \cite[Theorem 6.1]{gps:stable}, implies that we actually have the ``stronger'' form of (TC3) from~\cite[Definition 4.11]{may:traces}.

\begin{lem}\label{thm:tc3-six-squares}
  In the coherent diagram \autoref{fig:tc3} constructed in \autoref{thm:tc3}, the six squares that have $v$ as a vertex are cocartesian.
\end{lem}
\begin{proof}
  Since the tensor product $\otimes_{\bbtwo}$ is cocontinuous in both variables, it suffices to show that the corresponding squares in \autoref{fig:tc3pf} are all cocartesian.
  But all of these either occur in a cofiber sequence or are constant in one direction.
\end{proof}

\autoref{thm:tc3} does not require the monoidal structure of \D to be symmetric.
However, if it \emph{is} symmetric, the symmetry $\symm$ induces an isomorphism $\gamma$ between $f\otimes_{[\bbtwo]} f'$ and $f'\otimes_{[\bbtwo]} f$.
(Technically, as $f$ and $f'$ denote objects of $\D(\bbtwo\op)$ and $\D(\bbtwo)$, the isomorphism $\bbtwo\op\cong\bbtwo$ is involved here as well.)
If $f\otimes_{[\bbtwo]} f'$ is denoted by $v$ as above, then again following~\cite{may:traces}, we denote $f'\otimes_{[\bbtwo]} f$ by $\vbar$.

Likewise, we denote the maps $p_i$ and $j_i$ associated to $\vbar$ as $\pbar_i$ and $\jbar_i$.
The asymmetry of the proof of \autoref{thm:tc3} means that an argument is required to deduce that $\gamma$ and~$\symm$ are compatible with $p_i$, $j_i$, $\pbar_i$, and $\jbar_i$.

\begin{lem}\label{thm:involution}
  For objects and morphisms as described in the paragraphs above, we have
  \begin{alignat}{3}
    \gamma \circ p_1 &= \pbar_3 \circ \symm, &\qquad
    \gamma \circ p_2 &= \pbar_2 \circ \symm, &\qquad
    \gamma \circ p_3 &= \pbar_1 \circ \symm,\\
    \symm \circ j_1 &= \jbar_3 \circ \gamma, &\qquad
    \symm \circ j_2 &= \jbar_2 \circ \gamma, &\qquad
    \symm \circ j_3 &= \jbar_1 \circ \gamma.
  \end{alignat}
\end{lem}
\begin{proof}
  Consider the first equation, involving $p_1$ and $\pbar_3$.
  By two-variable functoriality of $\otimes_{[\bbtwo]}$, the following square commutes
  \begin{equation}
    \vcenter{\xymatrix{
      (y\ot 0) \otimes_{[\bbtwo]} (x'\to x')\ar[r]\ar[d] &
      (y\ot x) \otimes_{[\bbtwo]} (x'\to x')\ar[d]\\
      (y\ot 0) \otimes_{[\bbtwo]} (x'\to y')\ar[r] &
      (y\ot x) \otimes_{[\bbtwo]} (x'\to y').
      }}
  \end{equation}
  Up to isomorphism we can identify this with the left-hand square below
  \begin{equation}
    \vcenter{\xymatrix{
        y\otimes x'\ar[r]\ar[d] &
        y\otimes x'\ar[d] \ar[r]^{\symm} &
        x'\otimes y \ar[d]^{\pbar_3}\\
        y\otimes x'\ar[r]_-{p_1} &
        v \ar[r]_\gamma &
        \vbar
      }}
  \end{equation}
  in which the left vertical map and top left horizontal map are the identity.
  This can be proven by considering some cocartesian squares in $\D(\bbtwo)$, analogously to how we identified all the relevant morphisms in the (TC3) diagram.
  Thus, we have $\gamma \circ p_1 = \pbar_3 \circ \symm$ as desired.
  The other five proofs are essentially identical, the only difference being in the details of why the unlabeled maps are identities.
\end{proof}

In the language of~\cite{may:traces}, \autoref{thm:involution} says that $\vbar$ with its specified (TC3) diagram is an \textbf{involution} of $v$ with its specified (TC3) diagram.
In~\cite{may:traces}, the approach to this compatibility was slightly different: starting from arbitrary (TC3) diagrams for $f$ and $f'$ and 
for $f'$ and $f$, by choosing the isomorphism $\gamma\colon v\toiso \vbar$ cleverly and modifying some of the morphisms in the latter diagram, 
it was transformed into a possibly-different (TC3) diagram for $f'$ and $f$ which satisfied the conclusions of \autoref{thm:involution}.
In our case, by contrast, we have canonically specified diagrams of all sorts and a canonically specified $\gamma\colon v\toiso \vbar$, 
which automatically satisfy the desired compatibility conditions.

\section{Rotated pushout products}\label{sec:rotated}
We may also consider rotating one or the other of the input triangles in (TC3).
The following lemma implies that it doesn't matter which one we rotate.
Recall the functor $\cof\colon\D(\bbtwo)\to\D(\bbtwo)$ from \eqref{eq:cofibersquare}.

\begin{lem}\label{thm:tc3-rotate}
  Given $(y\xleftarrow{f}x)\in \D(\bbtwo\op)$ and $(x'\xto{f'}y')\in\D(\bbtwo)$, we have a canonical isomorphism
  \begin{equation}\label{eq:tc3-rotate}
    \cof(f) \otimes_{[\bbtwo]} f'  \;\cong\; f\otimes_{[\bbtwo]} \cof(f').
  \end{equation}
  This isomorphism identifies $p_1$, $p_2$, and $p_3$ on the left with $p_2$, $p_3$, and $p_1$ on the right, respectively, and similarly for the $j$'s.
\end{lem}
\begin{proof}
  We extend $f$ and $f'$ to cocartesian squares
  \begin{equation}
    \vcenter{\xymatrix@C=3pc{
        z\ar@{<-}[r]^{\cof(f)}\ar@{<-}[d] &
        y \ar@{<-}[d]^{f}\\
        0\ar@{<-}[r] &
        x
      }}
    \qquad\text{and}\qquad
    \vcenter{\xymatrix@C=3pc{
        x' \ar[r]^{f'}\ar[d] &
        y' \ar[d]^{\cof(f')}\\
        0\ar[r] &
        z'.
      }}
  \end{equation}
  Denote these objects by $Q\in\D(\Box\op)$ and $Q'\in\D(\Box)$, respectively.
  Now since $Q'$ is cocartesian, we have $Q' \cong (i_\ulcorner)_!(i_\ulcorner)^*Q'$.
  Thus, by \autoref{thm:tpadj} we have
  \begin{align}
    Q\otimes_{[\Box]} Q'
    &\cong Q \otimes_{[\Box]} (i_\ulcorner)_!(i_\ulcorner)^*Q'\\
    &\cong (i_\ulcorner\op)^*Q \otimes_{[\ulcorner]} (i_\ulcorner)^*Q'.
  \end{align}
  On the other hand, using \cite[Remark 4.2]{gps:stable} we also have $(i_\ulcorner\op)^*Q \cong (i\op)_! (\cof(f))$, where $i\colon\bbtwo\into\ulcorner$ is the inclusion of $(0,0) \to (0,1)$.
  Thus, by \autoref{thm:tpadj} again, we have
  \begin{align}
    (i_\ulcorner\op)^*Q \otimes_{[\ulcorner]} (i_\ulcorner)^*Q'
    &\cong
    (i\op)_! (\cof(f)) \otimes_{[\ulcorner]} (i_\ulcorner)^*Q'\\
    &\cong
    \cof(f) \otimes_{[\bbtwo]} i^* (i_\ulcorner)^*Q'\\
    &\cong
    \cof(f) \otimes_{[\bbtwo]} f'.
  \end{align}
  A dual argument shows $Q\otimes_{[\Box]} Q' \cong f \otimes_{[\bbtwo]} \cof(f')$, yielding~\eqref{eq:tc3-rotate}.

  To identify the $p$'s and $j$'s, we use the functoriality of this identification in~$f$.
  According to the proof of \autoref{thm:tc3}, $p_1$ and $j_1$ for $f\otimes_{[\bbtwo]} \cof(f')$ are obtained by tensoring $\cof(f')$ with the diagram on the left below
  \begin{equation}
    \vcenter{\xymatrix@C=3pc{
        0\ar@{<-}[r]\ar@{<-}[d] &
        x \ar@{<-}[d]\\
        y\ar@{<-}[r]^-f\ar@{<-}[d] &
        x \ar@{<-}[d]\\
        y\ar@{<-}[r] &
        0,
      }}
    \qquad\qquad
    \vcenter{\xymatrix@C=3pc{
        \Sigma x\ar@{<-}[r]\ar@{<-}[d] &
        0 \ar@{<-}[d]\\
        z\ar@{<-}[r]^{\cof(f)}\ar@{<-}[d] &
        y \ar@{<-}[d]\\
        y\ar@{<-}[r] &
        y.
      }}
  \end{equation}
  But by the above argument, this is isomorphic to what we obtain by tensoring the diagram on the right with~$f'$.
  Again, the proof of \autoref{thm:tc3} shows this yields~$p_3$ and~$j_3$ for~$\cof(f)\otimes_{[\bbtwo]}f'$, as intended.
  The other two cases are established in a similar way.
\end{proof}

Following~\cite{may:traces}, we denote the common value of $\cof(f) \otimes_{[\bbtwo]} f'$ and $f\otimes_{[\bbtwo]} \cof(f')$ by $w$.
Using \autoref{thm:tc3-rotate} and the natural isomorphism from  $\cof^3\colon \D(\bbtwo)\to\D(\bbtwo)$ to the suspension functor 
$\Sigma$ of $\D^\bbtwo$ in \cite[Lemma 5.13]{gps:stable}, $w$ can equivalently be defined as
\begin{align}
  w &= \cof(f) \otimes_{[\bbtwo]} f'\\
  &\cong \cof^2(\fib(f)) \otimes_{[\bbtwo]} \cof(\fib(f'))\\
  &\cong \cof^3(\fib(f)) \otimes_{[\bbtwo]} \fib(f')\\
  &\cong \Sigma(\fib(f)) \otimes_{[\bbtwo]} \fib(f')\\
  &\cong \Sigma(\fib(f) \otimes_{[\bbtwo]} \fib(f')).
\end{align}
The last is how $w$ is defined in~\cite{may:traces}: by rotating both triangles \emph{backwards}, obtaining $(-\Sigma^{-1}h,f,g)$ 
and $(-\Sigma^{-1}h',f',g')$, constructing their ordinary pushout product, then suspending the result once.
The object $w$ comes with distinguished triangles
\begin{gather}
  \DT{x\otimes z'}{k_1}{w}{q_1}{z\otimes y'}{h\otimes g'}{\Sigma(x\otimes z'),}\\
  \DTl{y\otimes y'}{k_2}{w}{q_2}{\Sigma(x\otimes x')}{-\Sigma(f\otimes f')}{\Sigma(y\otimes y'),}\label{eq:tc4dt2}\\
  \DT{z\otimes x'}{k_3}{w}{q_3}{y\otimes z'}{g\otimes h'}{\Sigma(z\otimes x')}
\end{gather}
equipped with coherent diagrams similar to those in (TC3).

In~\cite{may:traces} the existence of $w$ and its associated data is called (TC3$'$).
His next axiom is about the interaction of (TC3) and (TC3$'$).

\begin{thm}[TC4]\label{thm:tc4}
  With all the data chosen as above, there is a cocartesian square
  \[\vcenter{\xymatrix{
      v\ar[r]^{j_2}\ar[d]_{(j_1,j_3)} &
      y\otimes y'\ar[d]^{k_2}\\
      (x\otimes z') \oplus (z\otimes x')\ar[r]_-{[k_1,k_3]} &
      w.
    }}
  \]
\end{thm}
May's (TC4), in the strong form of his Definition 4.11, follows from this using the Mayer-Vietoris sequence  \cite[Theorem 6.1]{gps:stable}.
\begin{proof}
  Consider the following square in $\shift\D{\bbtwo\op}$
  \begin{equation}\label{eq:tc4input}
  \vcenter{\xymatrix{
      (y\ot x)\ar[r]\ar[d] &
      (y\ot y)\ar[d]\\
      (z\ot x)\ar[r] &
      (z\ot y).
      }}
  \end{equation}
  This is obtained by restriction from the canonical square
  \begin{equation}\label{eq:tc4cof}
  \vcenter{\xymatrix@-.5pc{
      x\ar[r]\ar[d] &
      y\ar[d]\\
      0\ar[r] &
      z.
      }}
  \end{equation}
  Moreover,~\eqref{eq:tc4input} is cocartesian, since its two underlying squares in \D are so (by~\cite[Corollary~2.6]{groth:ptstab}).
  The tensor product of~\eqref{eq:tc4input} with $(x'\to y') \in\D(\bbtwo)$ gives a cocartesian square in \D
  \begin{equation}\label{eq:tc4output}
  \vcenter{\xymatrix{
      v\ar[r]^-{j_2}\ar[d] &
      y\otimes y'\ar[d]^{k_2}\\
      ?\ar[r] &
      w.
      }}
  \end{equation}
  The objects $v$ and $y\otimes y'$ and the morphism $j_2$ in this diagram are identified using their definition in the proof of \autoref{thm:tc3} (TC3).
  Similarly, the lower-right corner is $\cof(f)\otimes_{[\bbtwo]} f'=w$.
  And since $k_2$ is by definition the morphism $p_3$ for $\cof(f)$ and $f'$, the definition of $p_3$ in the proof of \autoref{thm:tc3} identifies the morphism $k_2$ in~\eqref{eq:tc4output}.
  
  To identify the missing object and the two morphisms connected to it first note
   that the morphism $(z\ot x)$ at the lower-left corner of~\eqref{eq:tc4input} is the zero morphism.
  Moreover, we can obtain  a $\lrcorner$-diagram in $\D(\bbtwo\op)$
  \begin{equation}
  \vcenter{\xymatrix{
      & (0\ot x) \ar[d] \\
      (z\ot 0)\ar[r] &
      (z\ot x)
      }}
  \end{equation}by restriction from~\eqref{eq:tc4cof}. 
  Since its two underlying $\lrcorner$-diagrams in \D are coproduct diagrams, this diagram is a coproduct in $\D(\bbtwo\op)$.
  Thus, since $\otimes_{[\bbtwo]}$ is cocontinuous in each variable, the object in~\eqref{eq:tc4output} labeled ? is the coproduct
  \begin{equation}
    \Big((0\ot x) \otimes_{[\bbtwo]} (x'\to y')\Big) \oplus
    \Big((z\ot 0) \otimes_{[\bbtwo]} (x'\to y')\Big)
    \quad\cong\;
    (x\otimes z') \oplus (z\otimes x').
  \end{equation}
  We can now identify the two missing morphisms by composing them with the projections and coprojections of this biproduct, 
  which in turn must be obtained from those of the biproduct $(z\ot x) \cong (0\ot x) \oplus (z\ot 0)$.
  We leave it to the reader to verify that this yields precisely the definitions of $j_1$ and $j_3$ from the proof of \autoref{thm:tc3}, 
  and those of $k_1$ and $k_3$ coming from the construction of $w$.
\end{proof}

Of course, when \D is symmetric, we can also apply the symmetry to pushout products of rotated triangles.
Following~\cite{may:traces}, we denote the resulting object by $\wbar$ and its associated morphisms by $\kbar_i$ and $\qbar_i$; these satisfy analogous equations to those in \autoref{thm:involution}.

\section{Two-variable adjunctions and closed monoidal derivators}
\label{sec:tvas}

We have not yet introduced the notion of \emph{duality} which is essential to the definition of traces.
Dual pairs can be defined in any monoidal category, but they are easier to work with when the monoidal category is \emph{closed}.
Thus, in this section we to define what it means for a monoidal derivator to be closed, by using a notion of ``two-variable adjunction''.

In classical category theory, we say that a functor $\oast\colon\bC_1\times \bC_2\to\bC_3$ is a \emph{two-variable left adjoint} if each functor $(X\oast -)$ and $(-\oast Y)$ is a left adjoint.
This is equivalent to the existence of functors $\homrbare\colon \bC_2\op \times \bC_3 \to \bC_1$ and $\homlbare\colon\bC_3\times\bC_1\op \to \bC_2$ and natural isomorphisms
\[ \bC_3\big(X\oast Y,Z\big) \cong \bC_1\big(X, \homre{Y}{Z}\big) \cong \bC_2\big(Y, \homle{X}{Z}\big). \]
Our notational convention is chosen so that these isomorphisms preserve the cyclic ordering of $X,Y,Z$.
In this case we say that we have a \emph{two-variable adjunction}
\[ (\oast,\homrbare,\homlbare)\colon \bC_1 \times \bC_2 \tva \bC_3. \]

In particular, a monoidal category is closed (sometimes called \emph{biclosed} in the non-symmetric case) if and only if its tensor product is a two-variable left adjoint.
Thus, in order to define closed monoidal derivators, we must generalize two-variable adjunctions to derivators.
For the same reason as \autoref{warning:intccts-nope}, we must formulate this notion using the external rather than the internal products.

\begin{defn}\label{def:2va}
  A morphism $\oast \colon \D_1 \times \D_2 \to \D_3$ of derivators is a \textbf{two-variable left adjoint} if each external component
  \begin{equation}
    \oast\colon \D_1(A) \times\D_2(B) \to \D_3(A\times B)\label{eq:2vaextcpt}
  \end{equation}
  is a two-variable left adjoint, with adjoints
  \begin{align*}
    \homrbaree{B} &\colon \D_2(B)\op \times \D_3(A\times B) \xto{} \D_1(A)\\
    \homlbaree{A} &\colon \D_3(A\times B) \times \D_1(A)\op \xto{} \D_2(B)
  \end{align*}
  such that
  \begin{itemize}
  \item For any $Y\in\D_2(B)$ and $Z\in \D_3(A\times B)$ and functor $u\colon A'\to A$, the canonical mate-transformation
    \begin{equation}
      u^*\big(\homr{B}{Y}{Z}\big) \to \big(\homr{B}{Y}{(u\times 1)^* Z}\big)\label{eq:homrpsnat}
    \end{equation}
    is an isomorphism, and %dually
  \item For any $X\in\D_1(A)$ and $Z\in \D_3(A\times B)$ and functor $v\colon B'\to B$, the canonical mate-transformation
    \begin{equation}
      v^*\big(\homl{A}{X}{Z}\big) \to \big(\homl{A}{X}{(1\times v)^* Z}\big)\label{eq:homlpsnat}
    \end{equation}
    is an isomorphism.
  \end{itemize}
  In this case we say that $(\oast,\homrbare,\homlbare)$ is a \textbf{two-variable adjunction} $\D_1 \times \D_2 \tva \D_3$.
\end{defn}

\begin{defn}
  A monoidal derivator is \textbf{closed} if its tensor product $\otimes$ is a two-variable left adjoint.
\end{defn}

We have denoted these right adjoints with a subscript in brackets, analogously to the tensor product of functors, to indicate an indexing category that gets ``canceled'' (by an \emph{end} construction; see~\eqref{eq:repr-homr} and~\eqref{eq:homr-as-end}) and no longer indexes the output.
Thus for $X\in\D_1(A)$, $Y\in\D_2(B)$, and $Z\in\D_3(A\times B)$, we have adjunction isomorphisms
\begin{equation}
  \D_3(A\times B)\big(X\oast Y,Z\big) \;\cong\;
  \D_1(A)\big(X, \homr{B}{Y}{Z}\big) \;\cong\;
  \D_2(B)\big(Y,\homl{A}{X}{Z}\big).\label{eq:tva-adj}
\end{equation}

\begin{lem}
  If $\oast \colon \D_1 \times \D_2 \to \D_3$ is a morphism such that each component~\eqref{eq:2vaextcpt} is a two-variable left adjoint, 
  then the transformations~\eqref{eq:homrpsnat} and~\eqref{eq:homlpsnat} make the functors $\homrbaree{B}$ (for fixed $B$) and 
  $\homlbaree{A}$ (for fixed $A$) into lax natural transformations.
\end{lem}
\begin{proof}
  We consider $\homrbaree{B}$; the case of $\homlbaree{A}$ is dual.
  In this case the transformation~\eqref{eq:homrpsnat} can be regarded as living in the square
  \begin{equation}
    \vcenter{\xymatrix@C=3pc{
        \D_2(B)\op \times \D_3(A\times B) \ar[r]^{1\times (u\times 1)^*}\ar[d]_{\homrbaree{B}} \drtwocell\omit &
        \D_2(B)\op \times \D_3(A'\times B) \ar[d]^{\homrbaree{B}}\\
        \D_1(A)\ar[r]_{u^*} &
        \D_1(A')
      }}
  \end{equation}
  Strict naturality of $\homrbaree{B}$ would require that this square commute strictly; lax naturality means it is inhabited by a transformation (in this case,~\eqref{eq:homrpsnat}) satisfying functoriality axioms, e.g.\ that pasting the squares for $u\colon A\to A'$ and $u'\colon A'\to A''$ yields the square for $u'u\colon A\to A''$.
  It suffices to verify these axioms for a fixed $Y\in\D_2(B)$, in which case we can use the fact that the mate correspondence is functorial (i.e.\ the mate of a pasting is the pasting of the mates).
  This reduces the question to pseudofunctoriality of the adjoint $(-\oast Y)$, which holds by definition.
\end{proof}

Thus, a morphism $\oast$ as in the lemma is a two-variable left adjoint if and only if the lax transformations~\eqref{eq:homrpsnat} and~\eqref{eq:homlpsnat} are pseudonatural.

\begin{lem}\label{lem:2var}
  If each functor $\oast\colon \D_1(A) \times\D_2(B) \to \D_3(A\times B)$ is a two-variable left adjoint, then 
  $\oast\colon \D_1\times \D_2\to\D_3$ is a two-variable left adjoint if and only if it is cocontinuous in each variable.
\end{lem}
\begin{proof}
  In this situation, for fixed $Y\in\D_2(B)$ and $u\colon A\to A'$, the transformations~\eqref{eq:homrpsnat} and~\eqref{eq:coctsmate} are both mates of the pseudofunctoriality isomorphism $(u^* X) \oast Y \cong (u\times 1)^*(X\oast Y)$: one with respect to the adjunctions $(-\oast Y)\dashv (\homr{B}{Y}{-})$ (at $A$ and $A'$), and one with respect to the adjunctions $u_! \dashv u^*$ and $(u\times 1)_! \dashv (u\times 1)^*$.
  Therefore,~\eqref{eq:homrpsnat} and~\eqref{eq:coctsmate} are also mates of each other with respect to the composite adjunctions $(u_!(-)\oast Y)\dashv (\homr{B}{Y}{(u\times 1)^*(-)})$ and $(u\times 1)_!(-\oast Y)\dashv u^*(\homr{B}{Y}{-})$.
  
  In general, it need not be the case that the mate of an isomorphism is an isomorphism, but here we are in the special case where it is: when there are no functors involved other than the adjoints, the mate of an isomorphism $F\cong F'$ under adjunctions $F\dashv G$ and $F'\dashv G'$ is an isomorphism $G'\cong G$, expressing the uniqueness of adjoints.
  Thus,~\eqref{eq:homrpsnat} is an isomorphism if and only if~\eqref{eq:coctsmate} is.
\end{proof}

\begin{eg}\label{eg:repr-2va}
  If $\oast\colon\bC_1\times \bC_2\to\bC_3$ is an ordinary two-variable left adjoint between complete and cocomplete categories, 
  then the induced morphism $y(\bC_1) \times y(\bC_2) \to y(\bC_3)$ is a two-variable left adjoint of derivators.
  We remarked in \autoref{eg:repr-2vccts} that it is cocontinuous in each variable, so by \autoref{lem:2var} it suffices to exhibit adjoints to its external components.
  Using the usual notation for ends, we define the adjoint $\homrbaree{B}$ by
  \begin{equation}\label{eq:repr-homr}
    (\homr{B}{Y}{Z})_a \coloneqq \int_{b\in B} (\homre{Y_b}{Z_{(a,b)}})
  \end{equation}
  and similarly for $\homlbaree{A}$.
  In particular, if \bC is a complete and cocomplete closed monoidal category, then $y(\bC)$ is a closed monoidal derivator.
\end{eg}

\begin{eg}\label{thm:mmc-cmder}
  Suppose $\oast\colon\bC_1\times \bC_2\to\bC_3$ is a two-variable Quillen left adjoint between \emph{combinatorial} model categories.
  Then we can equip the diagram categories $\bC_1^A$, $\bC_2^B$, and $\bC_3^{A\times B}$ with the \emph{injective} model structures, in which the cofibrations and weak equivalences are levelwise.
  (Injective model structures on diagram categories generally only exist in the combinatorial case.)
  Since pushouts in functor categories are also levelwise, it follows that the induced functor
  \begin{equation}\label{eq:mmcten}
    \oast\colon\bC_1^A \times \bC_2^B \to \bC_3^{A\times B}
  \end{equation}
  is also a two-variable Quillen left adjoint.
  (A similar observation can be found in~\cite[Remark A.3.3.4]{lurie:higher-topoi}.)

  Therefore,~\eqref{eq:mmcten} induces a two-variable adjunction on homotopy categories.
  Together with \autoref{thm:mmc-mder}, this implies that the induced morphism of derivators $\ho(\bC_1) \times\ho(\bC_2) \to\ho(\bC_3)$ is a two-variable left adjoint.
  In particular, if \bC is a combinatorial monoidal model category, then $\ho(\bC)$ is a closed monoidal derivator.
  In fact, the hypothesis of combinatoriality can be weakened; see~\autoref{thm:mmc-cmder-noncomb}.
\end{eg}

The following lemma is also useful for constructing adjunctions of two variables.

\begin{lem}\label{lem:2varcom}
Let~$\oast\colon\D_1\times\D_2\to\D_3$ be a two-variable left adjoint and let $L_1\colon\E_1\to\D_1,$ $L_2\colon\E_2\to\D_2,$ and $L_3\colon\D_3\to\E_3$ be left adjoints. Then
\[L_3\circ\oast\circ(L_1\times L_2)\colon\E_1\times\E_2\to\E_3\]
is also a two-variable left adjoint.
\end{lem} 

\begin{proof}
The functoriality of mates implies that the given composite is cocontinuous in each variable.
 The result then follows easily from the corresponding result in ordinary category theory  
and \autoref{lem:2var}.
\end{proof}

Using Lemmas~\ref{lem:2var} and~\ref{lem:2varcom}, we now observe that if $\oast\colon\D_1\times\D_2\to\D_3$ is a two-variable left adjoint, 
then so are each of the ``liftings'' of its internal, external, and canceling components to morphisms of derivators.

\begin{eg}\label{thm:shifted-2va}
  If $\oast\colon\D_1\times\D_2\to\D_3$ is a two-variable left adjoint, then so is the induced morphism $\oast_B \colon \shift{\D_1}{B} \times \shift{\D_2}{B} \to \shift{\D_3}{B}$
  from~\eqref{eq:paramintern}.
  The external component of $\oast_B$ at $A$ and $C$ is the composite
  \[ \D_1(B\times A) \times \D_2(B\times C) \xto{\oast} \D_3(B\times A \times B\times C) \xto{(\symm\Delta_B)^*} \D_3(B\times A\times C). \]
  By the version of \autoref{lem:2varcom} for ordinary categories, this composite has a right adjoint in each variable.
  The functoriality of mates, together with cocontinuity of $\oast$ and $(\symm\Delta_B)^*$, implies that $\oast_B$ is also cocontinuous in each variable.
  Thus, by \autoref{lem:2var}, $\oast_B$ is a two-variable left adjoint.

  In particular, it follows that the ordinary functor $\oast_B\colon \D_1(B) \times \D_2(B) \to \D_3(B)$ is a two-variable left adjoint for all $B$.
  One might thus hope to be able to reformulate the notion of two-variable adjunction of derivators in terms of the internal monoidal structures only, 
  but the obvious version of the ``naturality'' condition for these adjunctions fails, for the same reason as \autoref{warning:intccts-nope}.

  It does follow, however, that if \D is a closed monoidal derivator, then so is $\shift{\D}{B}$.
\end{eg}

\begin{eg}\label{eg:paramextern-2va}
  If $\oast\colon \D_1 \times \D_2 \to \D_3$ is a two-variable left adjoint, then so is the induced morphism 
  $\shift{\D_1}{A} \times \shift{\D_2}{B} \to \shift{\D_3}{A\times B}$ from~\eqref{eq:paramextern}.
  This follows immediately from its definition using \autoref{thm:shifted-2va} and \autoref{lem:2varcom}.
\end{eg}

\begin{eg}\label{thm:tensor-adj}
  If $\oast\colon \D_1\times \D_2 \to \D_3$ is a two-variable left adjoint, then so is the induced morphism $\shift{\D_1}{B\op} \times\shift{\D_2}{B} \to \D_3$ from~\eqref{eq:tensor}.
  This follows immediately from its definition as the composite
  \[\shift{\D_1}{B\op} \times \shift{\D_2}{B} \xto{\oast}
  \shift{\D_3}{B\op\times B} \xto{(t\op,s\op)^*}
  \shift{\D_3}{\tw(B)\op} \xto{(\pi_{\tw(B)\op})_!}
  \D_3
  \]
  along with \autoref{eg:paramextern-2va} and \autoref{lem:2varcom}.
\end{eg}

\section{Cycling two-variable adjunctions}
\label{sec:cycling-tvas}

In an ordinary two-variable adjunction
\[ (\oast,\homrbare,\homlbare)\colon \bC_1\times \bC_2\tva\bC_3, \]
we can ``cycle'' the categories involved to obtain two other two-variable adjunctions
\begin{align}
  (\homlbareop,\oast\op,\homrbare)&\colon \bC_3\op\times \bC_1\tva\bC_2\op,\\
  (\homrbareop,\homlbare,\oast\op)&\colon \bC_2\times \bC_3\op\tva\bC_1\op.
\end{align}
This is a two-variable version of the fact that an adjunction $(F,G)\colon\bC_1 \rightleftarrows \bC_2$ can equally be regarded 
as an adjunction $(G\op,F\op)\colon\bC_2\op \rightleftarrows \bC_1\op$.
(The placement of the opposites can also be made more symmetrical; see~\cite{cgr:mvadj}.)

By contrast, in our definition of a two-variable adjunction for derivators, the morphisms $\homrbare$ and $\homlbare$ 
appear to play very different roles from $\oast$ (e.g.\ their components have different forms).
We now show that this asymmetry is only apparent: our two-variable adjunctions can be ``cycled'' just like the ordinary ones.

\begin{thm}\label{thm:cycle}
  From any two-variable adjunction of derivators
  \begin{equation}
    (\oast,\homrbare,\homlbare) \colon \D_1 \times \D_2 \tva \D_3\label{eq:2vatocycle}
  \end{equation}
  we can construct a ``cycled'' two-variable adjunction
  \begin{equation}
    (\homlbareop,\oast\op,\homrbare) \colon \D_3\op \times \D_1 \tva \D_2\op.\label{eq:cycled2va}
  \end{equation}
\end{thm}
\begin{proof}
  Invoking the definition of opposite derivators, we see that the external component of~\eqref{eq:cycled2va} at $A$ and $B$ should have the form
  \begin{equation}
    \homlbareop \colon \D_3(A\op)\op \times \D_1(B) \too \D_2(A\op\times B\op)\op\label{eq:cycledext}
  \end{equation}
  and thus its right adjoints should have the forms
  \begin{align}
    \oast_{[B\op]}\op &\colon \D_1(B)\op \times \D_2(A\op\times B\op)\op \too \D_3(A\op)\op,\\
    \homrbaree{A\op} &\colon \D_2(A\op\times B\op)\op \times \D_3(A\op) \too \D_1(B).\label{eq:cycledradj2}
  \end{align}
  We have used our standard notation with canceling subscripts for these right adjoints, which immediately tells us how to define them.
  Namely, we let $\oast_{[B\op]}\op$ be the opposite of an instance of the canceling product for $\oast$ with $A\op$ as an external parameter, i.e.\ the opposite of the composite
  \begin{equation}
    \D_1(B) \times \D_2(A\op\times B\op) \xto{\oast}
    \D_3(B \times A\op\times B\op) \xto{\int^{B\op}}
    \D_3(A\op)\label{eq:cycledradj1}
  \end{equation}
  so that for $X\in\D_1(B)$ and $Y\in\D_2(A\op\times B\op)$ we have
  \begin{equation}
    X \oast_{[B\op]} Y = (\pi_{\tw(B\op)\op})_!(t\op,s\op)^*(X\oast Y).
  \end{equation}
  By \autoref{thm:tensor-adj}, the functor~\eqref{eq:cycledradj1} is a two-variable left adjoint.
  Thus, its right adjoints give us definitions for $\homlbareop$ and $\homrbaree{A\op}$.
  Explicitly, for $X\in \D_1(B)$, $Y\in\D_2(A\op\times B\op)$, and $Z\in \D_3(A\op)$ we have
  \begin{align}
    \homleop{X}{Z} &= \homl{B\op}{X}{\big((t\op,s\op)_* (\pi_{\tw(B\op)\op})^* Z\big)}
    \label{eq:cycle-homleop}\\
    \homr{A\op}{Y}{Z} &= \homr{A\op\times B\op}{Y}{\big((t\op,s\op)_* (\pi_{\tw(B\op)\op})^* Z\big)}.
    \label{eq:cycle-homr}
  \end{align}

  According to \autoref{def:2va}, it remains to show that $\homlbareop$ is pseudonatural in both variables, that $\oast_{[B\op]}\op$ 
  is pseudonatural in $A$, and that $\homrbaree{A\op}$ is pseudonatural in $B$.
  Pseudonaturality of $\homlbareop$ in $A$ is equivalent to pseudonaturality of $\homlbare$ in $A\op$, which follows from cocontinuity of $\oast_{[B\op]}$ in $A\op$, as in \autoref{lem:2var}.
  Pseudonaturality of $\homlbareop$ in $B$ means an isomorphism
  \begin{equation}\label{eq:psnathomolbareopB}
    (\homleop{u^* X}{Z}) \;\cong\; (1\times u\op)^* (\homleop{X}{Z})
  \end{equation}
  for $u\colon  B\to B'$, $X\in \D_1(B')$, and $Z\in \D_3(A\op)$.
  To obtain this, note that we have an isomorphism
  \begin{equation}\label{eq:psnathomolbareopB'}
    (u^*X) \oast_{[B\op]} Y \;\cong\; X\oast_{[(B')\op]} (1\times u\op)_! Y
  \end{equation}
  for any $Y\in \D_2(A\op\times B\op)$, by \autoref{thm:tpadj} applied to the shifted two-variable adjunction $\D_1 \times \shift{\D_2}{(A\op)} \to \shift{\D_3}{(A\op)}$.
  Since we have adjunctions
  \begin{align}
    ((u^*X) \oast_{[B\op]} -) &\;\dashv\; (\homleop{u^* X}{-}) \qquad\text{and}\\
    (X\oast_{[(B')\op]} (1\times u\op)_! -) &\;\dashv\; (1\times u\op)^* (\homleop{X}{-}),
  \end{align}
  the mate of~\eqref{eq:psnathomolbareopB'} is the desired isomorphism~\eqref{eq:psnathomolbareopB}.

  Pseudonaturality of $\oast_{[B\op]}\op$ in $A$, of course, is equivalent to pseudonaturality of $\oast_{[B\op]}$ in $A\op$.
  Finally, pseudonaturality of $\homrbaree{A\op}$ in $B$ means an isomorphism
  \begin{equation}\label{eq:cycleriso}
     u^* (\homr{A\op}{Y}{Z}) \;\cong\; \homr{A\op}{((1\times u\op)^*Y)}{Z}
  \end{equation}
  for $u\colon B\to B'$, $Y\in \D_2(A\op\times (B')\op)$, and $Z\in \D_3(A\op)$.
  As with~\eqref{eq:psnathomolbareopB}, we take this to be a mate of the isomorphism
  \begin{equation}
    (u_! X) \oast_{[(B')\op]} Y \;\cong\; X \oast_{[B\op]} (1\times u\op)^* Y
  \end{equation}
  which exists for $X\in \D_1(B)$, again by \autoref{thm:tpadj}.
  We leave it to the reader to verify that the induced isomorphism~\eqref{eq:cycleriso} is, in fact, the canonical transformation~\eqref{eq:homrpsnat} which we require to be an isomorphism.
\end{proof}

Analogously to~\eqref{eq:tva-adj}, the cycled two-variable adjunction has isomorphisms
\begin{equation}
  \small
  \D_2(A\op\times B\op)\op\big(\homleop{X}{Z},Y\big) \;\cong\;
  \D_1(A)\big(X, \homr{B}{Y}{Z}\big) \;\cong\;
  \D_3(B\op)\op\big(Z,X \oast_{[A\op]} Y\big).
\end{equation}
If we iterate \autoref{thm:cycle}, we obtain its dual version, which gives a two-variable adjunction
\begin{equation}\label{eq:cycle2}
  (\homrbareop,\homlbare,\oast\op) \colon \D_2 \times \D_3\op \tva \D_1\op.
\end{equation}
From the proof of \autoref{thm:cycle}, we see that the functor $\oast\op$ occurring in~\eqref{eq:cycle2}, which has components
\begin{align}
  \oast_{[A\op]}\op &\colon \D_1(A\op\times B\op)\op \times \D_2(A)\op \too \D_3(B\op)\op,
\end{align}
is defined in terms of the original functor $\oast$ by
\begin{align}
  X\oast_{[A\op]}\op Y &= X \oast_{[A\op\times B\op]} (t,s)_! (\pi_{\tw(B\op)})^* Y\\
  &\cong X\oast_{[A\op]} Y \mathrlap{\hspace{2cm}\text{(by \autoref{thm:bicatunit}).}}
\end{align}
By uniqueness of adjunctions, it follows that the functors $\homrbareop$ and $\homlbare$ in~\eqref{eq:cycle2} must be the right adjoints of this, as constructed in \autoref{thm:tensor-adj}.
Namely, their components
\begin{align}
  \homrbareop &\colon \D_2(A) \times \D_2(B\op)\op \too \D_1(A\op\times B\op)\op\\
  \homlbaree{B} &\colon \D_2(B\op) \times \D_1(A\op\times B\op)\op \too \D_2(A)
\end{align}
must be isomorphic to
\begin{align}
  \homreop{Y}{Z} &= \homr{A}{Y}{\big((t\op,s\op)_* (\pi_{\tw(A\op)\op})^* Z\big)}
  \label{eq:cycle2-homreop}\\
  \homl{B\op}{X}{Z} &= \homl{A\op\times B\op}{X}{\big((t\op,s\op)_* (\pi_{\tw(A\op)\op})^* Z \big)}.
  \label{eq:cycle2-homl}
\end{align}
Note that these definitions are exactly symmetrical to those obtained in \autoref{thm:cycle}.
Finally, applying \autoref{thm:cycle} one more time, we obtain a two-variable adjunction of the original form
\begin{equation}\label{eq:cycle3}
  (\oast,\homrbare,\homlbare) \colon \D_1 \times \D_2 \tva \D_3
\end{equation}
whose functor $\oast$ is defined, for $X\in\D_1(A)$ and $Y\in\D_2(B)$, by
\begin{align}
  X\oast Y &= (t,s)_!(\pi_{\tw(B)})^* X \oast_{[B]} Y.
\end{align}
Applying \autoref{thm:bicatunit} again, we see that this is isomorphic to the original functor $\oast$, and hence its adjoints must  be isomorphic to the original adjoints.
Thus, up to isomorphism, \autoref{thm:cycle} describes a cyclic action on two-variable adjunctions.
Abstractly, we could say that derivators and two-variable adjunctions form a ``pseudo cyclic double multicategory''~\cite{cgr:mvadj}.

Also, by the construction in \autoref{thm:cycle}, the functor $\homrbare$ appearing in~\eqref{eq:cycle3} should be given by a canceling version of~\eqref{eq:cycle2-homreop}, which is to say an \emph{end} in $\D_1$:
\begin{equation}\label{eq:homr-as-end}
  \homr{B}{Y}{Z} \;=\; \int_B (\homre{Y}{Z})
  \;=\; (\pi_{\tw(B)\op})_* (t\op,s\op)^* (\homre{Y}{Z}).
\end{equation}
Since this is isomorphic to the original adjoint $\homrbaree{B}$ of $\oast$, we obtain a version of the formula~\eqref{eq:repr-homr} that holds in any derivator.

One application of \autoref{thm:cycle} is to deduce the rest of May's axiom (TC2).

\begin{thm}[the rest of (TC2)]\label{thm:tc2b}
  For any distinguished triangle \[\DT x f y g z h {\Sigma x}\]
  in $\D(\bbone)$ and any $w\in\D(\bbone)$, the following triangles are distinguished
  \begin{gather*}
    \DT{\homre w x}{\homre 1 f}{\homre w y}{\homre 1 g}{\homre w z}{\homre 1 h}{\Sigma(\homre w x),}\\
    \DT{\Sigma^{-1} (\homre x w)}{-\homre h 1}{\homre z w}{\homre g 1}{\homre y w}{\homre f 1}{\homre x w}
  \end{gather*}
  along with the analogous triangles involving $\homlbare$ (which, in the symmetric case, are isomorphic to the two above).
\end{thm}
\begin{proof}
  By assumption, we have a cofiber sequence
  \[\vcenter{\xymatrix@-.5pc{
      x\ar[r]\ar[d] &
      y\ar[r]\ar[d] &
      0\ar[d]\\
      0\ar[r] &
      z\ar[r] &
      \Sigma x.
    }}\]
  Now by \autoref{thm:cycle},
  \[ \homrbare\op \colon \D \times \D\op \to \D\op \]
  is itself a two-variable left adjoint, and hence cocontinuous in each variable.
  In particular, the morphism $(\homre w -)\colon \D\op\to\D\op$ preserves cofiber sequences, which (since $\D\op$ is stable) is equivalent to preserving fiber sequences.
  But fiber sequences in $\D\op$ are the same as cofiber sequences in $\D$, so preservation of these yields the first triangle above.

  The functor $(\homre - w)\colon \D\to \D\op$ also preserves cofiber sequences, which is to say that it takes cofiber sequences in \D to \emph{fiber} sequences in \D.
  As was observed in \cite[Lemma 5.22]{gps:stable} the fiber sequences induce the negative of the triangulation on \D, we get the second triangle above together with the minus sign as indicated.
\end{proof}

Another application of \autoref{thm:cycle} is to extend \autoref{thm:mmc-cmder} to the non-combinatorial case.

\begin{thm}\label{thm:mmc-cmder-noncomb}
  If \bC is a cofibrantly generated monoidal model category, then $\ho(\bC)$ is a closed monoidal derivator.
\end{thm}
\begin{proof}
  If \bC is not combinatorial, the injective model structure on $\bC^A$ may not exist, but the \emph{projective} model structure always does, in which the \emph{fibrations} and weak equivalences are objectwise.
  Moreover, under the isomorphism
  \[\left(\bC^{B\op}\right)\op \cong (\bC\op)^B \]
  the dual of the projective model structure on the left-hand side becomes identified with an injective model structure on the right-hand side.
  Since projective cofibrations are in particular objectwise cofibrations, it follows that the induced two-variable adjunction
  \[ \homrbareop \colon \bC^A \times \left(\bC^{B\op}\right)\op \to \left(\bC^{A\op\times B\op}\right)\op
  \]
  is a two-variable Quillen left adjoint.
  As in the combinatorial case, therefore, we obtain a two-variable adjunction of derivators
  \begin{equation}
    (\homrbareop,\homlbare,\otimes\op) \colon \ho(\bC) \times \ho(\bC)\op \too \ho(\bC)\op
  \end{equation}
  which we can therefore cycle forwards into a two-variable adjunction
  \begin{equation}
    (\otimes',\homrbare,\homlbare)\colon \ho(\bC) \times \ho(\bC) \too \ho(\bC).
  \end{equation}
  Thus, to make $\ho(\bC)$ into a closed monoidal derivator, 
  it suffices to identify this functor $\otimes'$ with the monoidal structure $\otimes$ on $\ho(\bC)$ constructed in \autoref{thm:mmc-mpder}.
  By definition, $\otimes'$ is given by~\eqref{eq:cycle-homleop} with $\otimes'_{[B]}$ replacing $\homlbaree{B\op}$:
  \begin{equation}\label{eq:mmc-cycled}
    Z \otimes' X = \big((t,s)_! (\pi_{\tw(B\op)})^* Z\big) \otimes'_{[B]} X.
  \end{equation}
  Here $\tw(B\op)\op$ and $(t\op,s\op)_*$ have become $\tw(B\op)$ and  $(t,s)_!$, respectively, because of passage to the opposite derivator.
  Moreover, $\otimes'_{[B]}$ denotes, not the canceling tensor product $\otimes_{[B]}$ constructed in \S\ref{sec:ends-coends}, 
  but the left derived functor of the point-set-level canceling tensor product
  \[ \bC^{A\times B\times B\op} \times \bC^B \too \bC^{A\times B} \]
  obtained by applying it to projectively cofibrant diagrams.
  However, because this point-set-level functor is the composite
  \[ \bC^{A\times B\times B\op} \times \bC^B \xto{\otimes} \bC^{A\times B\times B\op \times B}
  \xto{(t\op,s\op)^*} \bC^{A\times B\times \tw(B\op)}
  \xto{(\pi_{\tw(B\op)})_!} \bC^{A\times B},
  \]
  its left derived functor is naturally isomorphic to the composite of the left derived functors of each of these ingredients.
  But these left derived functors are precisely the corresponding functors in $\ho(\bC)$ which went into the construction of $\otimes_{[B]}$, so we have $\otimes_{[B]}' \cong \otimes_{[B]}$.
  Finally, applying \autoref{thm:bicatunit}, we obtain $\otimes'\cong \otimes$.
\end{proof}
\noindent
The proof actually shows that a two-variable Quillen adjunction between cofibrantly generated model categories induces a two-variable adjunction of derivators.

\section{Duality in closed monoidal derivators}
\label{sec:duality}

With closed monoidal derivators in hand, we move on to study duality.
Following~\cite{may:traces}, for an object $x\in\D(\bbone)$, we refer to the internal-hom $\homre{x}{\lS}$ as the \textbf{canonical dual} of $x$ and write it as $\dual x$.
There is a canonical \textbf{evaluation} map $\epsilon\colon \dual x \otimes x \to \lS$, defined by adjunction from the identity of $\dual x$.

We say that $x$ is \textbf{dualizable} if the canonical map
\begin{equation}\label{eq:dualmap}
  \mu_{x,u}\colon u\otimes \dual x \too \homre x u
\end{equation}
(whose adjunct\footnote{In general, if $F\colon \bC\to\bD$ is left adjoint to $G\colon\bD\to\bC$ and $\phi\colon F x \to y$ is a morphism in \bD, then by the \emph{adjunct} of $\phi$ we mean its image $\bar{\phi}\colon x\to G y$ under the adjunction isomorphism $\bD(F x ,y) \cong \bC(x,Gy)$.} is $u\otimes \dual x \otimes x \xto{1_u \otimes \epsilon} u\otimes \lS \cong u$)
is an isomorphism for all objects $u$.
It is sufficient to require this for $u=x$. 

More generally, the two-variable morphism $\homrbare \colon  \D\op\times \D \to\D$ induces a one-variable morphism
\[ \dual{(-)} \coloneqq (\homre{-}{\lS}) \colon  \D\op \too \D.\]
If $X\in\D(A\op)$, we will refer to $\dual{X}\in\D(A)$ as its \textbf{pointwise (canonical) dual}, since by naturality we have $(\dual X)_a \cong \dual{(X_a)}$.
We have a generalized evaluation map $\epsilon\colon \dual X \otimes_{[A]} X \to \lS$ defined analogously.
The functoriality of pointwise duals allows us to prove the following.

\begin{lem}\label{thm:dtduals-3for2}
  In a stable monoidal derivator, given a distinguished triangle
  \[ \DT{x}{f}{y}{g}{z}{h}{\Sigma x} \]
  if any two of $x$, $y$, and $z$ are dualizable, then so is the third.
\end{lem}
\begin{proof}
  Suppose given a cofiber sequence $X\in\D(\boxbar)$
  \begin{equation}
    \vcenter{\xymatrix{
        x\ar[r]^f\ar[d] &
        y\ar[r]\ar[d]^g &
        0\ar[d]\\
        0\ar[r] &
        z\ar[r]_-h &
        \Sigma x.
      }}
  \end{equation}
  Then we have the evaluation map of the pointwise dual, $\epsilon\colon \dual X \otimes_{[\boxbar]} X \to \lS$, which we can tensor with $u\in\D(\bbone)$ to obtain a morphism
  \[ (u\otimes \dual X) \otimes_{[\boxbar]} X \toiso
  u\otimes (\dual X \otimes_{[\boxbar]} X) \xto{\epsilon}
  u \otimes \lS \toiso u
  \]
  in $\D(\bbone)$.
  The adjunct of this is a morphism
  \[ \mu_{X,u}\colon u\otimes \dual X \too \homre{X}{u} \]
  in $\D(\boxbar)$.
  By naturality, we have $(\mu_{X,u})_{(1,2)} = \mu_{x,u}$ and so on; thus the underlying morphism of $\mu_{X,u}$ in $\D(\bbone)^\boxbar$ gives a morphism of distinguished triangles
  \begin{equation}
    \vcenter{\xymatrix{
        \Sigma^{-1}(u\otimes \dual x)
        \ar[r]^-{-\dual h} \ar[d]_{\Sigma^{-1}\mu_{x,u}} &
        u\otimes \dual z\ar[r]^-{\dual g}\ar[d]^{\mu_{z,u}} &
        u\otimes \dual y\ar[r]^-{\dual f}\ar[d]^{\mu_{y,u}} &
        u\otimes \dual x\ar[d]^{\mu_{x,u}}\\
        \Sigma^{-1}(\homre x u) \ar[r]_-{-\homre h 1} &
        \homre z u\ar[r]_-{\homre g 1} &
        \homre y u\ar[r]_-{\homre f 1} &
        \homre x u.
      }}
  \end{equation}
  However, both the domain and codomain of $\mu_{u,X}$ are cofiber sequences, and a morphism of bicartesian squares is an isomorphism 
  as soon as it restricts to an isomorphism on $\ulcorner$ or $\lrcorner$, since $(i_\ulcorner)_!$ and $(i_\lrcorner)_*$ are fully faithful.
  The result follows.
\end{proof}

The last step in this proof can be regarded as a stable-derivator version of the ``five-lemma'' for triangulated categories, which says that a morphism of distinguished triangles which is an isomorphism at two places is also an isomorphism at the third.
The derivator version refers instead to morphisms of (coherent) cofiber sequences, and is true whether or not the derivator is strong.

\begin{cor}
  The pushout of any span of dualizable objects in a stable monoidal derivator is dualizable.
\end{cor}
\begin{proof}
  By \autoref{thm:dtduals-3for2} together with the Mayer-Vietoris sequence \cite[Theorem 6.1]{gps:stable}, 
and the fact that finite coproducts of dualizable objects in an additive monoidal category are dualizable.
\end{proof}

We also have a compatibility between duality and pushout products.
To express this, note first that since $\homrbare\op$ is cocontinuous in both variables, $\dual{(-)}$ takes colimits in \D to limits in \D.
In particular, it takes cofiber sequences to fiber sequences, and when $\D$ is stable, it preserves bicartesianness of squares.
Expressed at the level of $\D(\bbone)$, this says that the functor $\dual{(-)}\colon \D(\bbone)\op \to \D(\bbone)$ preserves distinguished triangles, 
in the sense that if
\[\DT x f y g z h {\Sigma x}, \]
is distinguished then so is
\[\DT{\dual z}{\dual g}{\dual y}{\dual f}{\dual x}{\dual \Sigma^{-1} h}{\dual \Sigma^{-1} z \cong \Sigma(\dual z)}. \]
(This is a special case of \autoref{thm:tc2b} (TC2).)
By \autoref{thm:tc3-rotate}, therefore, we have
\[ \dual f \otimes_{[\bbtwo]} f
\;\cong\;
\cof(\dual g) \otimes_{[\bbtwo]} f
\;\cong\;
\dual g \otimes_{[\bbtwo]} \cof(f)
\;\cong\;
\dual g \otimes_{[\bbtwo]} g.
\]
Similarly, this same object is isomorphic to $\dual h  \otimes_{[\bbtwo]} h$.
Following May, we denote it by $\wbar$; it is the object occurring in (TC3\/$'$) for the triangles $(\dual{g}, \dual{f}, \dual{\Sigma^{-1}h})$ and $(f, g, h)$.
The following theorem, which is the first half of May's axiom (TC5), says that the evaluation morphisms $\epsilon$ jointly induce a unique morphism $\overline\epsilon\colon \dual f \otimes_{[\bbtwo]} f \to\lS$.

\begin{thm}[TC5a]\label{thm:tc5a}
  With $\wbar\coloneqq \dual g \otimes_{[\bbtwo]} g$ as above, there is a map $\overline{\epsilon}\colon \wbar \rightarrow \lS$ such that the following (incoherent) diagrams commute:  \begin{equation}\label{eq:tc5a}
    \xymatrix@C=3pc{\dual{x}\otimes x
      \ar[r]^-{\overline{k}_3}\ar[dr]_{\epsilon} &
      \wbar\ar[d]^{\overline{\epsilon}}\\
      &\lS}\hspace{10pt}
    \xymatrix@C=3pc{ \dual{y}\otimes y\ar[r]^-{\overline{k}_2}\ar[rd]_{\epsilon}&
      \wbar\ar[d]^{\overline{\epsilon}}
     \\
      &\lS}\hspace{10pt}
    \xymatrix@C=3pc{\dual{z}\otimes z 
      \ar[r]^-{\overline{k}_1}\ar[dr]_{\epsilon} &
      \wbar\ar[d]^{\overline{\epsilon}}\\
      &\lS}
  \end{equation}
%  \begin{equation}\label{eq:tc5a}
%    \xymatrix@C=3pc{(\dual{z}\otimes z) \oplus (\dual{x}\otimes x)
%      \ar[r]^-{[\overline{k}_1, \overline{k}_3]}\ar[dr]_{[\epsilon,\epsilon]} &
%      \wbar\ar[d]^{\overline{\epsilon}}&
%      \dual{y}\otimes y\ar[l]_-{\overline{k}_2}\ar[ld]^{\epsilon}\\
%      &\lS}
%  \end{equation}
\end{thm}
\begin{proof}
  As in the proof of \autoref{thm:tc3-rotate}, we can obtain $\wbar$ as the canceling tensor product over $\Box$ of the two canonical diagrams
  \begin{equation}
    \vcenter{\xymatrix{
        \dual x \ar@{<-}[r]^{\dual f}\ar@{<-}[d] &
        \dual y \ar@{<-}[d]^{\dual g}\\
        0 \ar@{<-}[r] &
        \dual z
      }}
    \qquad\text{and}\qquad
    \vcenter{\xymatrix{
        x \ar[r]^f\ar[d] &
        y \ar[d]^g\\
        0 \ar[r] &
        z.
      }}
  \end{equation}
  The second of these is the standard cofiber square for $(x\xto{f}y)\in\D(\bbtwo)$; let us denote it by $X$.
  The first is its pointwise canonical dual.
  But now, the counit of the adjunction
  \[ \D(\bbone)(M \otimes_{[\Box]} X,N) \cong
  \D(\Box\op)(M, \homre{X}{N})
  \]
  at $\lS$ is a morphism
  \begin{equation}
    \wbar = \dual{X} \otimes_{[\Box]} X \too \lS,\label{eq:tc5a0}
  \end{equation}
  which we may denote by $\overline{\epsilon}$.

  It remains to show that $\overline{\epsilon} \circ \kbar_i = \epsilon$ for $i=1,2,3$.
  Let $Y$ denote the square
  \begin{equation}
    \vcenter{\xymatrix@-.5pc{
        y \ar[r]\ar[d] &
        y \ar[d]\\
        0\ar[r] &
        0.
      }}
  \end{equation}
  Then by the proof of \autoref{thm:tc3-rotate}, $\kbar_2$ is induced by tensoring $X$ with the map of squares
  $\dual{Y} \to \dual{X}$ which is
  induced by the evident map of squares $X\to Y$.
  Thus, by extraordinary naturality \cite{ek:gen-funct-calc} of the counits of adjunctions, we have a commutative square
  \begin{equation}
    \vcenter{\xymatrix{
        \dual{Y} \otimes_{[\Box]} X \ar[r]\ar[d]_{\kbar_2} &
        \dual{Y} \otimes_{[\Box]} Y \ar[d]^{\epsilon}\\
        \dual{X} \otimes_{[\Box]} X\ar[r]_-{\overline{\epsilon}} &
        \lS.
      }}
  \end{equation}
  It is easy to see that the top map here is an isomorphism, both objects being isomorphic to $\dual y \otimes y$, and that $\epsilon$ is simply the counit of $y$.
  Thus, we have $\overline{\epsilon}\circ \kbar_2 = \epsilon$, as desired.
  The cases $i=1,3$ are nearly identical.
\end{proof}

\section{Duality for profunctors}\label{prof-dual}

To complete the proof of additivity, we need to consider not only the symmetric monoidal duality discussed in \S\ref{sec:duality}, but
also the generalization to bicategorical duality in the special case of the bicategory $\cProf(\D)$ from \S\ref{sec:ends-coends}.  We will recall all of the results we 
require here; see e.g.~\cite[Ch.~16]{maysig:pht} for the general theory of duality in bicategories.

Recall that a bicategory is said to be \textbf{closed} if its composition functors are two-variable left adjoints.

\begin{prop}\label{thm:bicat-closed}
  If \D is a closed monoidal derivator, then $\cProf(\D)$ is closed.
\end{prop}
\begin{proof}
  By \autoref{thm:tensor-adj}.
\end{proof}

Thus, we have natural isomorphisms
\begin{align}
  \D(A\times C\op)(X\otimes_{[B]} Y,Z)
  &\cong
  \D(A\times B\op)(X, \homr{C}{Y}{Z})\\
  &\cong
  \D(B\times C\op)(Y, \homl{A}{X}{Z}).
\end{align}

Following the symmetric monoidal case,
 we say that $X\in\cProf(\D)(A,B) = \D(A\times B\op)$ is \textbf{right dualizable} if the canonical map
\begin{equation}\label{eq:bicatdualmap}
  \mu_{X,U}\colon U\otimes_{[B]} (\homr{B}{X}{\lI_B}) \too \homr{B}{X}{U}
\end{equation}
is an isomorphism, for all $U\in \cProf(\D)(C,B)$.
As before, it suffices to require this when $C=A$ and $U=X$.

We write $\dualr X = \homr{B}{X}{\lI_B}\in\cProf(\D)(B,A)$ and call it the \textbf{canonical right dual}.
Similarly, we have the canonical left dual $\duall X = \homl{A}{X}{\lI_A}$.
If $X$ is right dualizable, then $\dualr X$ is left dualizable and $X\cong \duall {\dualr X}$.
To give some intuition, we describe some easy ways to obtain right dualizable objects.

\begin{lem}\label{thm:bcodual}
  For any $f\colon A\to B$, the diagram $(f\times 1)^*\lI_B \in \cProf(\D)(A,B)$ is right dualizable.
\end{lem}
\begin{proof}
  Recall that \autoref{thm:tpadj}\ref{item:tpadj2} gives isomorphisms
  \begin{align*}
    Y \otimes_{[A]} (f\times 1)^*\lI_B &\cong (1\times f\op)_! Y \otimes_{[B]} \lI_B\\
    &\cong (1\times f\op)_! Y
  \end{align*}
  natural in $Y\in \cProf(\D)(C,A)$.
  Therefore, the right adjoints of the functors
  \[ (-\otimes_{[A]} (f\times 1)^*\lI_B) \qquad\text{and}\qquad (1\times f\op)_!,\]
  namely
  \[(\homr{B}{(f\times 1)^*\lI_B}{-}) \qquad\text{and}\qquad (1\times f\op)^*,\]
  are also isomorphic, by isomorphisms which respect the adjunction units and counits.
  Now recall that~\eqref{eq:bicatdualmap} is by definition the adjunct of the composite
  \begin{align*}
    \big(U\otimes_{[B]} (\homr{B}{X}{\lI_B})\big)  \otimes_{[A]} X
    &\toiso U\otimes_{[B]} \big((\homr{B}{X}{\lI_B})  \otimes_{[A]} X\big)\\
    &\to\; U \otimes_{[B]} \lI_B\\
    &\toiso U
  \end{align*}
  in which the middle map is $U$ tensored with the counit of the adjunction
  \begin{equation}
    (-\otimes_{[A]} X)\dashv (\homr{B}{X}{-}).
  \end{equation}
  Thus, in the case $X=(f\times 1)^*\lI_B$, by passing across the above isomorphism of adjunctions and using \autoref{rmk:tpadj-assoc}, we can identify~\eqref{eq:bicatdualmap} with the adjunct of the composite
  \begin{align*}
    (1\times f\op)_!\big(U\otimes_{[B]} (1\times f\op)^*\lI_B\big)
    &\toiso U\otimes_{[B]} (1\times f\op)_!(1\times f\op)^*\lI_B\\
    &\to\; U \otimes_{[B]} \lI_B\\
    &\toiso U.
  \end{align*}
  in which the middle map is $U$ tensored with the counit of the adjunction $(1\times f\op)_!\dashv (1\times f\op)^*$.
  But by functoriality of mates, this composite is equal to
  \begin{align*}
    (1\times f\op)_!\big(U\otimes_{[B]} (1\times f\op)^*\lI_B\big)
    &\toiso (1\times f\op)_!(1\times f\op)^*(U\otimes_{[B]} \lI_B)\\
    &\to\; U \otimes_{[B]} \lI_B\\
    &\toiso U
  \end{align*}
  in which the counit is no longer tensored with anything.
  Since the adjunct of a counit is an identity,~\eqref{eq:bicatdualmap} is an isomorphism.
\end{proof}

In particular, the proof of \autoref{thm:bcodual} shows that the canonical right dual of $(f\times 1)^*\lI_B$ is $(1\times f\op)^*\lI_B$.
These are analogues of the classical \emph{representable profunctors} $B(-,f-)$ and $B(f-,-)$.

\begin{lem}\label{thm:duals-compose}
  If $X\in\cProf(\D)(A,B)$ and $Y\in\cProf(\D)(B,C)$ are right dualizable, so is $X\otimes_{[B]} Y$.
\end{lem}
\begin{proof}
  $\mu_{X\otimes_{[B]} Y,U}$ factors as $\mu_{X,Y\rhd_{[C]} U}$ followed by $\homr{B}{1_X}{\mu_{Y,U}}$.
\end{proof}

We can also identify the canonical right dual of $X\otimes_{[B]} Y$ as $\dualr Y \otimes_{[B]} \dualr X$.

\begin{lem}\label{thm:dual-objwise}
  A coherent diagram $X\in\cProf(\D)(A,B)$ is right dualizable if and only if $X_a \in \cProf(\D)(\bbone,B)$ is right dualizable for every $a\in A$.
\end{lem}
\begin{proof}
  By (Der2), the map~\eqref{eq:bicatdualmap} is an isomorphism just when it is so when restricted to each $a\in A$.
  All the functors involved are pseudonatural in the extra variables, so the restriction of $\mu_{X,U}$ to $a\in A$ is $\mu_{X_a,U}$.
  Thus, $\mu_{X,U}$ is an isomorphism if and only if each $\mu_{X_a,U}$ is so.
\end{proof}

In particular, \autoref{thm:dual-objwise} tells us that $X\in\D(A)$ is right dualizable, when regarded as a morphism from $A$ to \bbone in $\cProf(\D)$, 
if and only if each object in its underlying diagram is dualizable in $\D(\bbone)$.
For instance, an object $(x\xto{f}y) \in\D(\bbtwo)$, regarded as a morphism from $\bbtwo$ to \bbone in $\cProf(\D)$, 
is right dualizable just when $x$ and $y$ are dualizable, 
and in this case its right dual is just its ``pointwise dual'' $(\dual x \xot{\dual f} \dual y) \in\D(\bbtwo\op)$.

On the other hand, in general, right dualizability of $X$ regarded as a morphism in the other direction, from $\bbone$ to $A$, depends not just on the objects of $X$ but on $A$.
In the stable case, we have the following.

\begin{lem}\label{thm:totaldual}
  Let $X\in \cProf(\D)(\bbone,\bbtwo) = \D(\bbtwo\op)$. If the underlying objects of~$X$ are dualizable, then $X$ is right dualizable.
\end{lem}
\begin{proof}
  Suppose $X = (y\xot{f} x)$.
  We may extend $X$ to a bicartesian square of the form
  \begin{equation}
    \vcenter{\xymatrix@-.5pc{
      u\ar[r]^k\ar[d] &
      x\ar[d]^f\\
      0\ar[r] &
      y.
      }}
  \end{equation}
  By restriction, we obtain a cube
  \begin{equation}\label{eq:totaldualcube}
    \vcenter{\xymatrix@-1pc{u \ar[dr] \ar@{<-}[rr] \ar[dd] && 0 \ar[dr] \ar'[d][dd] \\
      & x \ar@{<-}[rr] \ar[dd] && x \ar[dd] \\
      0 \ar@{<-}'[r][rr] \ar[dr] && 0 \ar[dr] \\
      & y \ar@{<-}[rr]_f && x.
    }}
  \end{equation}
  Its left and right face are bicartesian, so by \cite[Corollary~2.6]{groth:ptstab} 
  it is bicartesian in $\shift{\D}{\bbtwo\op}$.
  Therefore, by the bicategorical version of \autoref{thm:dtduals-3for2} (whose proof is identical), it suffices to show that $(u\ot 0)$ and $(x\ot x)$ are right dualizable.
  But we have $(u\ot 0) \cong u\otimes (\lS\ot 0)$ and $(x\ot x) \cong x\otimes (\lS\ot\lS)$, so since $x$ and $u$ are dualizable 
  (the latter by \autoref{thm:dtduals-3for2}), by \autoref{thm:duals-compose} it suffices to show that $(\lS \ot 0)$ and $(\lS\ot\lS)$ are right dualizable.
  But since $\lI_\bbtwo\in\cProf(\D)(\bbtwo,\bbtwo)=\D(\bbtwo\times\bbtwo\op)$ is given by
  \[\vcenter{\xymatrix@-.5pc{
      \lS \ar@{<-}[r]\ar[d] &
      0 \ar[d]\\
      \lS\ar@{<-}[r]& \lS
    }}\]
these are just the ``representable profunctors'' from \autoref{thm:bcodual} for the two functors $\bbone\to\bbtwo$.
\end{proof}

The proof of \autoref{thm:totaldual} also enables us to identify the bicategorical right dual of $(y\xot{f} x)$.
Namely, applying $\dualr{}$ to~\eqref{eq:totaldualcube} regarded as an object of $\cProf(\D)(\Box,\bbtwo)$, 
we obtain an object of $\cProf(\D)(\bbtwo,\Box)$, which we claim must look like
\begin{equation}\label{eq:totaldualcube2}
  \vcenter{\xymatrix@-1pc{\dual u \ar@{<-}[dr] \ar[rr] \ar@{<-}[dd]
    && \dual u \ar@{<-}[dr]^{\dual k} \ar@{<-}'[d][dd] \\
    & 0 \ar[rr] \ar@{<-}[dd] && \dual x \ar@{<-}[dd]^{\dual f} \\
    0 \ar'[r][rr] \ar@{<-}[dr] && 0 \ar@{<-}[dr] \\
    & \dual z \ar[rr]_{\dual g} && \dual y
  }}
\end{equation}
where $(y\xto{g}z) = \cof(x\xto{f} y)$.
By \autoref{thm:bcodual}, we have $\dualr(\lS\ot\lS) \cong (0 \to\lS)$ and $\dualr(0\ot\lS) \cong (\lS\to\lS)$.
By \autoref{thm:duals-compose}, taking tensor products of these with $x$ and $u$ respectively corresponds in the dual to taking tensor products with $\dual x$ and $\dual u$.
This (together with its functoriality) identifies the top face of~\eqref{eq:totaldualcube2}.

Since~\eqref{eq:totaldualcube2} is also bicartesian, its right and left faces must also be bicartesian.
This identifies the arrow $\dual f$, since it must be the fiber of $\dual k$.
Since the back face of~\eqref{eq:totaldualcube2} is trivially bicartesian, its front face must  be as well.
This identifies the arrow $\dual g$, since it must be the fiber of $\dual f$.
(Of course, we have $z\cong \Sigma u$.)
Thus, the right dual of $f$ is canonically isomorphic to the pointwise dual of its cofiber:
\begin{equation}\label{eq:totaldual}
  \dualr(f) \;\cong\; \dual{(\cof(f))}.
\end{equation}
We can now construct what in~\cite{may:traces} are called \textbf{duals of diagrams}.

\begin{lem}\label{thm:almost-tc5b}
  For $x\xto{f} y$ and $x'\xto{f'}y'$, where $x$ and $y$ are dualizable, we have
  \begin{equation}\label{eq:almost-tc5b}
    \dual{(f \otimes_{[\bbtwo]} f')} \;\cong\; \dual{f'} \otimes_{[\bbtwo]} \dual{(\cof(f))}.
  \end{equation}
  Moreover, under this isomorphism we have the following left-to-right correspondences between the morphisms from (TC3):
  \begin{alignat}{3}
    p_1 &\leftrightarrow j_2, &\qquad
    p_2 &\leftrightarrow j_3, &\qquad
    p_3 &\leftrightarrow j_1, \\
    j_1 &\leftrightarrow p_2, &\qquad
    j_2 &\leftrightarrow p_3, &\qquad
    j_3 &\leftrightarrow p_1.
  \end{alignat}
\end{lem}
\begin{proof}
  By \autoref{thm:totaldual}, the chosen object $(y\xot{f} x) \in \D(\bbtwo\op) = \cProf(\D)(\bbone,\bbtwo)$ is right dualizable.
  We now have the following chain of isomorphisms, natural in $f$ and $f'$.
  \begin{align}
    \dual{(f \otimes_{[\bbtwo]} f')}
    &= \homr{\bbone}{(f\otimes_{[\bbtwo]} f')}{\lI_\bbone}\\
    &\cong \homr{\bbtwo}{f}{(\homr{\bbone}{f'}{\lI_\bbone})}\\
    &= \homr{\bbtwo}{f}{\dualr f'}\\
    &\cong \dualr f' \otimes_{[\bbtwo]} \dualr f\\
    &\cong \dual{f'} \otimes_{[\bbtwo]} \dual{(\cof(f))}.
  \end{align}
  Here the equality in the first line is the definition of the canonical dual $\dual{(f \otimes_{[\bbtwo]} f')}$, the isomorphism in the second line is an instance of a general fact in any closed bicategory, and the equality in the third line is the definition of the canonical right dual $\dualr {f'}$.
  The isomorphism in the fourth line is again an instance of a general fact about dualizable 1-morphisms in a closed bicategory (applied here to $f$).
  Finally, the isomorphism in the last line uses the observation after \autoref{thm:dual-objwise} that the right dual of $f'$ is its pointwise dual, along with the identification~\eqref{eq:totaldual} of $\dualr f$.

  It is easy to verify that when $f$ or $f'$ has one of the ``degenerate'' forms appearing in the outer parts of \autoref{fig:tc3pf}, this isomorphism is essentially an identity.
  Thus, we can easily trace through the construction of the $p$'s and $j$'s to find the desired identifications.
\end{proof}

In particular, if we let $f'$ be $\dual{(\cof(f))}$, then $f \otimes_{[\bbtwo]} f'$ in~\eqref{eq:almost-tc5b} occurs in a (TC3) diagram for 
distinguished triangles $(f,g,h)$ and $(\dual g,\dual f,\dual{\Sigma^{-1} h})$, while $\dual{f'} \otimes_{[\bbtwo]} \dual{(\cof(f))}$ 
occurs in a (TC3$'$) diagram for these same triangles in the reverse order (using the identification of $f$ with its double pointwise dual).
Following~\cite{may:traces}, we call the (TC3$'$) diagram for $(\dual g,\dual f,\dual{\Sigma^{-1} h})$ and $(f,g,h)$ the \textbf{dual} of the (TC3) diagram for $(f,g,h)$ and 
$(\dual g,\dual f,\dual{\Sigma^{-1} h})$.

We can now observe that the second half of May's final axiom holds.
The tautologous nature of the proof shows the advantage of derivators, where objects are specified by true universal properties, over triangulated categories, where they can only be asserted to exist.

\begin{thm}[TC5b]\label{thm:tc5b}
  If $x$, $y$, and $z$ are dualizable, then the (TC3\/$'$) diagram specified in (TC5a) for the triangles $(\dual{g}, \dual{f}, \dual{\Sigma^{-1}h})$ 
  and $(f, g, h)$ is isomorphic to the dual of a (TC3) diagram for the same triangles in reverse order, and satisfies the additivity axiom (TC4) with respect to an involution of the latter diagram.
\end{thm}
\begin{proof}
  Since all of our diagrams are canonically specified by universal properties, we can say that \emph{the} dual of \emph{the} (TC3) 
  diagram for $(f, g, h)$ and $(\dual{g}, \dual{f}, \dual{\Sigma^{-1}h})$ is \emph{the} (TC3$'$) diagram for $(\dual{g}, \dual{f}, \dual{\Sigma^{-1}h})$ 
  and $(f, g, h)$, while \emph{the} involution of the former diagram is \emph{the} (TC3) diagram for $(\dual{g}, \dual{f}, \dual{\Sigma^{-1}h})$ and $(f, g, h)$.
  Thus, since \autoref{thm:tc4} (TC4) is a statement about \emph{the} (TC3) and (TC3$'$) diagrams for any pair of triangles, it holds for these.
\end{proof}

Finally, we follow May in concluding the dual of \autoref{thm:tc5a} (TC5a).
For any object $x$, we have a morphism $\lS\too \homre x x$ which is the adjunct of the unit isomorphism $\lS \otimes x \cong x$.
If $x$ is dualizable, the composite of this with $(\mu_{x,x})^{-1}$ defines a map
\[\eta\colon \lS\rightarrow x\otimes \dual x\]
which is called the \textbf{coevaluation} of $x$.
(Dualizability can equivalently be defined as the existence of such a coevaluation, satisfying the ``triangle'' or ``zig-zag'' identities relating 
it to the  evaluation map $\epsilon\colon\dual x \otimes x \to \lS$.)  See \cite{dp:duality, lms:equivariant} for further details. 

Now,
combining \autoref{thm:tc5b} with \autoref{thm:tc5a}, we obtain the following result, whose statement and proof are both identical to \cite[Lemma 4.14]{may:traces}.

\begin{lem}\label{lem:tc5av}
  For $v$ as defined in (TC3) for the distinguished triangles $(f,g,h)$ and $(Dg, Df, D\Sigma^{-1}h)$, there is a map $\bar{\eta}\colon \lS\rightarrow v$ such that the following diagrams commute
  \[
\xymatrix@C=3pc{
    \lS\ar[dr]^\eta\ar[d]^{\bar{\eta}}\\
    v\ar[r]_-{j_1}&x\otimes \dual{x}
    }\hspace{10pt}
\xymatrix@C=3pc{
    \lS\ar[dr]^\eta\ar[d]^{\bar{\eta}}\\
    v\ar[r]_-{j_2}&y\otimes \dual{y}
    }\hspace{10pt}
\xymatrix@C=3pc{
    \lS\ar[dr]^\eta\ar[d]^{\bar{\eta}}\\
    v\ar[r]_-{j_3}&z\otimes \dual{z}.
    }
  \]
%  \[\xymatrix@C=3pc{
%    &\lS\ar[dr]^\eta\ar[d]^{\bar{\eta}}\ar[dl]_{(\eta,\eta)}\\
%    (z\otimes \dual{z})\oplus (x\otimes \dual{x})&v\ar[l]^-{(j_3,j_1)}\ar[r]_-{j_2}&y\otimes \dual{y}.
%  }\]
\end{lem}

\section{The additivity of traces in stable monoidal derivators}
\label{sec:introadditivity}

Finally, we are ready to consider traces.
We begin by recalling the definitions.
When $x$ is dualizable and $f\colon p\otimes x\rightarrow x\otimes q$ is any morphism, the \textbf{trace} of $f$ is defined to be the composite 
\[\xymatrix{p\cong p\otimes \lS\ar[r]^-{\id\otimes \eta}&p\otimes x\otimes \dual{x}\ar[r]^-\symm
&\dual{x}\otimes p\otimes x\ar[r]^-{\id \otimes f}&
\dual{x}\otimes x\otimes q\ar[r]^{\epsilon\otimes \id}&\lS\otimes q\cong q.
}\]
An important special case is when $p$ and $q$ are the unit object \lS, in which case the trace of $f\colon x\to x$ is an endomorphism of $\lS$.
In the even more special case when $f$ is the identity morphism $1_x$, its trace is called the \textbf{Euler characteristic} of $x$.

The desired additivity theorem for traces says that given a distinguished triangle
\[ \DT x f y g z h {\Sigma x} \]
with $x$, $y$, and hence $z$ dualizable (see \autoref{thm:dtduals-3for2}), and a diagram of the form
\[\xymatrix@C=3pc{ 
  p\otimes x\ar[r]^-{1_p\otimes f}\ar[d]^{\phi_x} &
  p\otimes y\ar[r]^-{1_p\otimes g}\ar[d]^{\phi_y} &
  p\otimes z\ar[r]^-{1_p\otimes h}\ar[d]^{\phi_z} &
  p\otimes \Sigma x \ar[d]^{\Sigma \phi_x}\\
  x\otimes q\ar[r]_-{f\otimes 1_q} &
  y\otimes q\ar[r]_-{g\otimes 1_q} &
  z\otimes q\ar[r]_-{h\otimes 1_q} &
  \Sigma x \otimes q
} \]
then we have
\[\tr(\phi_x)+\tr(\phi_z) = \tr(\phi_y).\]
However, if by ``diagram'' we mean an \emph{incoherent} diagram, i.e.\ a diagram in a triangulated category, then this might not be true (see for instance~\cite{ferrand:nonadd}).
We need $(\phi_x,\phi_y,\phi_z)$ to be coherent in an appropriate sense; this is exactly what the machinery of derivators is designed to identify.

We will first state the theorem for an explicitly coherent choice of $\phi$'s.
Then we will deduce a corollary whose statement, at least, makes sense in the language of triangulated categories.

\begin{thm}\label{thm:additivity}
  Let $\D$ be a closed symmetric monoidal stable derivator.
  Suppose we have a bicartesian square
  \begin{equation}
  \vcenter{\xymatrix{
      x\ar[r]^f\ar[d] &
      y\ar[d]^g\\
      0\ar[r] &
      z
      }}
  \end{equation}
  denoted $X\in \D(\Box)$, objects $p,q\in\D(\bbone)$, and a morphism $\phi\colon p\otimes X \to X\otimes q$ in $\D(\Box)$.
  If any two of $x$, $y$ and  $z$ are dualizable in $\D(\bbone)$, then
  \[\tr(\phi_x)+\tr(\phi_z)=\tr(\phi_y)\]
  as morphisms $p\to q$ in $\D(\bbone)$.
\end{thm}

\begin{proof}[Proof of \autoref{thm:additivity} following \cite{may:traces}]
  We will construct the (incoherent) commutative diagram in $\D(\bbone)$ shown in \autoref{fig:additivity}.
  This will prove the theorem, since the composite around the left is $\tr(\phi_x)+\tr(\phi_z)$ and the composite around the right is $\tr(\phi_y)$.
  \begin{figure}
    \centering
    \[\xymatrix@C=5pc{
      &p\otimes\lS\ar[d]^{1\otimes \overline{\eta}}\ar[dr]^{1\otimes \eta}\ar[dl]_{(1\otimes \eta,1\otimes \eta)}\\
      (p\otimes z\otimes \dual{z})\oplus(p\otimes x\otimes \dual{x})
      \ar[dd]_{\symm\oplus\symm}\ar[dr]_{[1\otimes p_3,1\otimes p_1]}
      & p\otimes v \ar[l]^-{(1\otimes j_3,1\otimes j_1)} \ar[r]_-{1\otimes j_2}
      & p\otimes y\otimes \dual{y}\ar[dd]^{\symm}\ar[dl]^{1\otimes p_2}\\
      &p\otimes w\ar@{.>}[d]\\
      (\dual{z}\otimes p\otimes z)\oplus (\dual{x}\otimes p\otimes x)
      \ar[r]\ar[d]_{(1\otimes \phi_z)\oplus (1\otimes \phi_x)}
      &\wtil\ar@{.>}[d]
      &\dual{y}\otimes p\otimes y\ar[l]\ar[d]^{1\otimes \phi_y}\\
      (\dual{z}\otimes z\otimes q)\oplus (\dual{x}\otimes x\otimes q)
      \ar[dr]_{[\epsilon\otimes 1, \epsilon\otimes 1]} \ar[r]^-{[\kbar_1\otimes 1,\kbar_3\otimes 1]}
      &\wbar\otimes q\ar[d]^{\overline{\epsilon}\otimes 1}
      &\dual{y}\otimes y\otimes q\ar[l]_-{\kbar_2\otimes 1}\ar[dl]^{\epsilon\otimes 1}\\
      &\lS\otimes q}
    \]
    \caption{The additivity of traces}
    \label{fig:additivity}
  \end{figure}
  As before, we are using the notations:
  \begin{itemize}
  \item $v\coloneqq f \otimes_{[\bbtwo]} \dual g$ occurs in the (TC3) diagram for $(f,g,h)$ and $(Dg, Df, D\Sigma^{-1}h)$.
  \item $w\coloneqq g \otimes_{[\bbtwo]} \dual g$ occurs in the (TC3$'$) diagram for $(f,g,h)$ and $(Dg, Df, D\Sigma^{-1}h)$.
  \item $\wbar \coloneqq \dual g \otimes_{[\bbtwo]} g$ occurs in the (TC3$'$) diagram for $(Dg, Df, D\Sigma^{-1}h)$ and $(f,g,h)$.
  \end{itemize}
  Thus, the lower triangles in \autoref{fig:additivity} commute by \autoref{thm:tc5a} (TC5a) tensored on the right with $q$, 
  while the upper triangles commute by \autoref{lem:tc5av} tensored on the left with $p$.
  Similarly, the quadrilateral involving $p\otimes v$ and $p\otimes w$ commutes by \autoref{thm:tc4} (TC4) tensored on the left with $p$.
  
  Now since the tensor product is cocontinuous in each variable, it preserves all the constructions in (TC3) and (TC3$'$); 
  thus $p\otimes w$ is (canonically isomorphic to) the object occurring in the (TC3$'$) diagram for $(1_p\otimes f,1_p\otimes g,1_p\otimes h)$ and $(Dg, Df, D\Sigma^{-1}h)$.
  (Technically, of course, the former means the bicartesian square $p\otimes X$.)
  Similarly, $\wbar\otimes q$ is the object occurring in the (TC3$'$) diagram for $(Dg, Df, D\Sigma^{-1}h)$ and $(f\otimes 1_q,g\otimes 1_q,h\otimes 1_q)$.

  We define $\wtil$ to be the object defined in the (TC3$'$) diagram for $(Dg, Df, D\Sigma^{-1}h)$ and $(1_p\otimes f,1_p\otimes g,1_p\otimes h)$.
  Then \autoref{thm:involution} yields the dotted map $p\otimes w \to \wtil$ making the two adjacent trapezoids commute.
  Finally, since $\phi$ is a map of bicartesian squares, the functoriality of the (TC3$'$) construction yields the second dotted map 
  $\wtil \to \wbar\otimes q$ making the remaining squares commute.
\end{proof}

We can now deduce a corollary that makes sense in the language of triangulated categories.

\begin{cor}\label{thm:triang-additivity} 
  Let \D be a closed symmetric monoidal stable derivator which is additionally strong.
  Suppose we have a distinguished triangle in $\D(\bbone)$
  \[ \DT x f y g z h {\Sigma x} \]
  in which $x$ and $y$ (hence also $z$) are dualizable, and morphisms $\phi_x$ and $\phi_y$ making the following diagram of solid arrows commute in $\D(\bbone)$
  \[\xymatrix@C=3pc{ 
    p\otimes x\ar[r]^-{1_p\otimes f}\ar[d]^{\phi_x} &
    p\otimes y\ar[r]^-{1_p\otimes g}\ar[d]^{\phi_y} &
    p\otimes z\ar[r]^-{1_p\otimes h}\ar@{.>}[d]^{\phi_z} &
    p\otimes \Sigma x \ar[d]^{\Sigma \phi_x}\\
    x\otimes q\ar[r]_-{f\otimes 1_q} &
    y\otimes q\ar[r]_-{g\otimes 1_q} &
    z\otimes q\ar[r]_-{h\otimes 1_q} &
    \Sigma x \otimes q.
  } \]
  Then there exists a morphism $\phi_z$ making both squares commute, such that
  \[\tr(\phi_x)+\tr(\phi_z)=\tr(\phi_y).\]
\end{cor}
\begin{proof}
  By definition, a distinguished triangle in $\D(\bbone)$ is isomorphic to one arising from a cofiber sequence in \D.
  Thus, we may assume given a bicartesian square $X$ as in \autoref{thm:additivity}.
  Let $X_0$ denote the restriction of $X$ to an object $(x\xto{f} y) \in \D(\bbtwo)$.

  Since $\D$ is strong, the ``underlying diagram'' functor $\D(\bbtwo) \to \D(\bbone)^\bbtwo$ is full.
  Thus, since $(\phi_x,\phi_y)$ defines a morphism $1_p\otimes f\to f \otimes 1_q$ in $\D(\bbone)^\bbtwo$, it must be in the image of some morphism $\phi_0\colon p \otimes X_0 \to X_0 \otimes q$ in $\D(\bbtwo)$.
  Applying extension by zero and a left Kan extension to $\phi_0$, we obtain $\phi\colon p\otimes X\to X\otimes q$.
  Thus, we can apply \autoref{thm:additivity}.
\end{proof}

Of course, when $\phi_x$ and $\phi_y$ are identities, we may choose $\phi_0$ to be the identity of $X_0$, so that $\phi_z$ is also an identity.
Thus we recover the additivity of Euler characteristics as a special case.

\appendix

\section{On coends in derivators}
\label{sec:coends}

Recall that in \S\ref{sec:ends-coends} we defined the coend of $X\in \D(A\op\times A)$ to be 
\[ \int^A X \coloneqq (\pi_{\tw(A)\op})_! (t\op,s\op)^* X, \]
This reproduces the usual definition of coend in a represented derivator.
However, there are other ways in which one might try to generalize the construction of coends to derivators.
For instance, if \bC is a cocomplete ordinary category, then the coend $\int^{a\in A} X(a,a)$ is equivalently a coequalizer
\[ \coprod_{\alpha\colon a_1 \to a_2} X(a_2,a_1) \toto \coprod_a X(a,a). \]
Our coends in derivators can also be rephrased in a similar way, but (as usual in a homotopical context) coequalizers must be generalized to 
``geometric realizations of simplicial bar constructions''.
In a derivator, ``geometric realization'' just means the colimit of a diagram of shape $\bbDelta\op$; we now construct such a diagram whose colimit is the coend of $X\in\D(A\op\times A)$.

Let $(\bbDelta/A)\op$ denote the opposite of the category of simplices of $A$.
Its objects are functors $[n] \to A$, where $[n]$ denotes the $(n+1)$-element totally ordered set regarded as a category, 
and its morphisms from $[n] \to A$ to $[m]\to A$ are functors $[m]\to [n]$ over $A$.
(The opposite of the category $A^{\S}$ from~\cite[\S IX.5]{maclane} is contained in $(\bbDelta/A)\op$ as a final, but not homotopy final, subcategory.
This reflects the fact that mere coequalizers suffice to define coends in non-homotopical category theory.)

There is a functor $c\colon (\bbDelta/A)\op \to \tw(A)\op$ which simply composes up a string of arrows; we claim it is homotopy final.
Using the characterization of homotopy final functors in \cite[Corollary 3.13]{gps:stable}, it suffices to show that for any 
$\alpha\colon a_1\to a_2$ in $A$, the category $(\alpha/c)$ is homotopy contractible.
To see this, note first that $(\alpha/c)$ is isomorphic to $(\bbDelta/(a_1/A/a_2)_\alpha)\op$, where $(a_1/A/a_2)_\alpha$ is the category of pairs of morphisms in $A$ with composite $\alpha$.

Now for any category $B$, there is a functor $s\colon (\bbDelta/B)\op \to B$ which picks out the first object in a string of arrows.
For any $b\in B$, the fiber $s^{-1}(b)$ has a terminal object, hence is homotopy contractible.
Moreover, $s^{-1}(b)$ is a coreflective subcategory of $(b/s)$, so the latter is also homotopy contractible.
Thus $s$ is homotopy final, and in particular a homotopy equivalence.

Therefore, to show that $(\alpha/c) \cong (\bbDelta/(a_1/A/a_2)_\alpha)\op$ is homotopy contractible, it suffices to show that $(a_1/A/a_2)_\alpha$ is so.
But by~\cite[Theorem 3.15]{gps:stable}, this follows from homotopy exactness of 
the square \begin{equation}
  \vcenter{\xymatrix@-.5pc{
      A\ar[r]^{1_A}\ar[d]_{1_A} &
      A\ar[d]^{1_A}\\
      A\ar[r]_{1_A} &
      A.
      }}
\end{equation}

This completes the proof that $c\colon (\bbDelta/A)\op \to \tw(A)\op$ is homotopy final.
Thus, the coend $\int^A X$ can be identified with the colimit of the restriction of $X$ along
\[\xymatrix@C=3pc{(\bbDelta/A)\op \ar[r]^{c} &
  \tw(A)\op \ar[r]^{(t\op,s\op)} &
  A\op\times A.}
\]
Moreover, since $\pi_{(\bbDelta/A)\op}$ factors through $\bbDelta\op$ by a functor $p\colon (\bbDelta/A)\op \to \bbDelta\op$, this colimit is also isomorphic to the colimit of $p_!c^*(t\op,s\op)^*X$;
this is our ``simplicial bar construction''.
To see that it deserves the name, note that $p$ is a discrete opfibration, so that each pullback square
\begin{equation}
  \vcenter{\xymatrix@-.5pc{
      P_n\ar[r]\ar[d] &
      (\bbDelta/A)\op\ar[d]\\
      \bbone\ar[r]_{[n]} &
      \bbDelta\op
      }}
\end{equation}
is homotopy exact, where $P_n$ is the discrete category on the set of composable strings of $n$ arrows in $A$.
Thus, $p_!c^*(t\op,s\op)^*X$ looks like
\begin{equation}
  \xymatrix{
    \raisebox{-3mm}{$\cdots$} \ar@<1.6mm>[r] \ar@{}[r]|-{\vdots} \ar@<-3mm>[r] &
    \displaystyle\coprod_{a_1 \to a_2 \to a_3}^{\vphantom{a_n}} X(a_3,a_1)
    \ar@{<-}@<1mm>[r] \ar@{<-}@<-1mm>[r] \ar[r] \ar@<2mm>[r] \ar@<-2mm>[r] &
    \displaystyle\coprod_{a_1 \to a_2}^{\vphantom{a_n}} X(a_2,a_1)
    \ar@<1mm>[r] \ar@<-1mm>[r] \ar@{<-}[r] &
    \displaystyle\coprod_a^{\vphantom{a_n}} X(a,a)
  }
\end{equation}
as we expect of a bar construction.

A third definition of the coend, familiar in enriched category theory, is as the weighted colimit of $X$ weighted by the hom-functor $\hom_A\colon A\op\times A\to \bSet$.
There is no direct notion of ``weighted colimit'' in a derivator, but in this case we can mimic one using left Kan extensions into \emph{collages} (see also~\cite{gs:enriched}).
Specifically, consider the category $\ncoll(\hom_A)$ which contains $A\op\times A$ as a full subcategory, 
together with one additional object $\infty$, and with the morphisms from $(a_1,a_2)$ to $\infty$ being the homset $A(a_1,a_2)$ and no nonidentity morphisms with domain~$\infty$.
Classically, we can obtain the $\hom_A$-weighted colimit of $X$ by left Kan extending from $A\op\times A$ to $\ncoll(\hom_A)$, then evaluating at $\infty$.
However, it is easy to see that there is a comma square
\begin{equation}\label{eq:coendcollage}
  \vcenter{\xymatrix@C=3pc{
      \tw(A)\op \ar[r]^-{(t\op,s\op)} \ar[d]_{\pi_{\tw(A)\op}} \drtwocell\omit &
      A\op\times A\ar[d]\\
      \bbone \ar[r]_-{\infty} &
      \ncoll(\hom_A)
      }}
\end{equation}
so that by (Der4), the same is true in any derivator.

\section{Unitality of the bicategory of profunctors}
\label{sec:bicatunit}

In this appendix we will prove \autoref{thm:bicatunit}.
Recall the statement:

\begin{lem}
  For any diagrams
  \[X \in \D(A\times B\op) \qquad\text{and}\qquad Y\in \D(C) \]
  we have a natural isomorphism
  \[ X\otimes_{[B]} \big((t,s)_! (\pi_{\tw(B)})^* Y\big) \cong X\otimes Y. \]
\end{lem}

\begin{proof}
Since $(t,s)\colon \tw(B)\xto{}B\times B\op$ is a Grothendieck opfibration, the following pullback square is homotopy exact
\[
\xymatrix@C=4pc{
  \tw(B)\op\times_{B} \tw(B)\ar[r]^-{t\op\times 1}\ar[d]_{1\times s}&
  B\op\times \tw(B)\ar[d]^{1\times (t,s)}\\
  \tw(B)\op\times B\op\ar[r]_-{(t\op,s\op)\times 1}&
  B\op\times B\times B\op.
}
\]
Thus, in the following diagram, the square commutes up to isomorphism for any derivator \E:
\begin{equation}
\small\vcenter{\xymatrix@C=2.5pc{
  \E(\tw(B)\op\times_{B} \tw(B))\ar[d]_{s_!}
  &\E(B\op\times \tw(B))\ar[l]_-{(t\op)^*}\ar[d]^{(t,s)_!}
  &\E(B\op)\ar[l]_-{(\pi_{\tw (B)})^\ast}\\
  \E(\tw(B)\op\times B\op)\ar[d]_{(\pi_{\tw (B)\op})_!} &
  \E(B\op\times B\times B\op)\ar[l]^{(t\op,s\op)^\ast}\\
  \E(B\op)&
}}\label{eq:unitalitysq}
\end{equation}
Letting $\E = \shift\D{A\times C}$ and using the fact that $\otimes$ preserves colimits in each variable, for $X$ and $Y$ as in the statement of the lemma we have
\begin{align}
  X\otimes_{[B]} \big((t,s)_! (\pi_{\tw(B)})^* Y\big)
  &=
  (\pi_{\tw(B)\op})_! \, (t\op,s\op)^* (X\otimes (t,s)_! (\pi_{\tw(B)})^* Y) \\
  &\cong
  (\pi_{\tw(B)\op})_! \, (t\op,s\op)^* (t,s)_! (\pi_{\tw(B)})^* (X\otimes Y)\\
  &\cong
  (\pi_{\tw(B)\op})_! \,s_!\,(t\op)^* \, (\pi_{\tw(B)})^* \, (X\otimes Y)\label{eq:unitiso-pullout}
\end{align}
Thus, it suffices to show that the upper-left composite in~\eqref{eq:unitalitysq} is isomorphic to the identity functor.
For this, it will suffice to show that the following square is homotopy exact
\begin{equation}\label{eq:unithoexsq}
  \vcenter{\xymatrix{
      \tw(B)\op\times_{B} \tw(B) \ar[r]^-{t\op}\ar[d]_{s} \drtwocell\omit &
      B\op\ar@{=}[d]\\
      B\op\ar@{=}[r] &
      B\op.
    }}
\end{equation}
The objects of $\tw(B)\op\times_{B} \tw(B)$ are composable pairs $(b_1\xto{\beta}b_2 \xto{\gamma} b_3)$ of morphisms in $B$, and its morphisms from $(\gamma,\beta)$ to $(\gamma',\beta')$ are commutative diagrams:
\begin{equation}
  \vcenter{\xymatrix@-.7pc{
      \ar[r]^\beta\ar@{<-}[d] &
      \ar[r]^\gamma\ar[d] &
      \ar@{<-}[d]\\
      \ar[r]_{\beta'} &
      \ar[r]_{\gamma'} &
    }}
\end{equation}
We have $t\op(\gamma,\beta)=b_3$ and $s(\gamma,\beta) = b_1$, and the natural transformation in~\eqref{eq:unithoexsq} simply composes $\beta$ and $\gamma$ (this goes in the direction shown because its target is $B\op$ rather than $B$).
Now using the characterization of homotopy exact squares in \cite[Theorem 3.15]{gps:stable}, it suffices to show that for any $b_0, b_4 \in B$ and $\ph\colon b_0 \to b_4$, the category
\[ \Big( b_4 \,/\, \big(\tw(B)\op\times_{B} \tw(B)\big) \,/ \, b_0 \Big)_\ph \]
is homotopy contractible.
Denote this category by $E_4$: its objects are composable quadruples $(b_0 \xto{\alpha} b_1 \xto{\beta} b_2 \xto{\gamma} b_3 \xto{\delta} b_4)$ in $B$ whose composite is $\ph$, where the objects $b_0$ and $b_4$ are fixed but the other three can vary.
Its morphisms are commutative diagrams
\begin{equation}
  \vcenter{\xymatrix@C=1.5pc@R=1.2pc{
      &\ar[r]^\beta\ar@{<-}[dd] &
      \ar[r]^\gamma\ar[dd] &
      \ar@{<-}[dd] \ar[dr]^{\delta}\\
      b_0 \ar[ur]^{\alpha} \ar[dr]_{\alpha'}&&&&
      b_4 .\\
      &\ar[r]_{\beta'} &
      \ar[r]_{\gamma'} &
      \ar[ur]_{\delta'}
    }}
\end{equation}
Now the full subcategory $E_3\subset E_4$ of objects where $\delta$ is an identity is coreflective in $E_4$.
Then the further subcategory $E_2\subset E_3$ where both $\gamma$ and $\delta$ are identities is reflective in $E_3$.
Finally, the third subcategory $E_1 \subset E_2$ where $\beta$, $\gamma$, and $\delta$ are all identities is coreflective in $E_2$, and $E_1$ contains only the single object
\[(b_0 \xto{\ph} b_4 = b_4 = b_4 = b_4).\]
Thus, $E_4$ is connected by a chain of adjunctions to $\bbone$, hence is homotopy contractible.
\end{proof}

Taking $Y = \lS_{\bbone}$ in \autoref{thm:bicatunit}, we have an isomorphism $\rho\colon X\otimes_{[B]} \lI_B \cong X$.
Dually, of course, we have $\lambda\colon\lI_A \otimes_{[A]} X \cong X$.

\begin{lem}
  These associativity and unitality isomorphisms satisfy the unit axiom
  \begin{equation}\label{eq:profcoh}
    \vcenter{\xymatrix{
        (X \otimes_{[B]} \lI_B) \otimes_{[B]} Y \ar[rr] \ar[dr] &&
        X \otimes_{[B]} (\lI_B \otimes_{[B]} Y) \ar[dl] \\
        & X\otimes_{[B]} Y.
        }}
  \end{equation}
\end{lem}
\begin{proof}
  Consider coherent diagrams $X\in\D(A\times B\op)$ and $Y\in\D(B\times C\op)$, and $\lI_B\in\D(B\times B\op)$.
  By passing to the derivator $\shift\D{A\times C\op}$, we can suppress the additional parameters~$A$ and~$C\op$.
  We will construct a commutative diagram
  \begin{equation}\label{eq:coh}
    \vcenter{\xymatrix{
        (X \otimes_{[B]} \lI_B) \otimes_{[B]} Y \ar[rr]^-\cong \ar[d]_\cong &&
        X \otimes_{[B]} (\lI_B \otimes_{[B]} Y) \ar[d]^\cong \\
        \int^B (X \otimes \lS_\bbone) \otimes Y \ar[rr]_-\cong \ar[dr] &&
        \int^B X \otimes (\lS_\bbone \otimes Y) \ar[dl] \\
        & \int^BX\otimes Y
        }}
  \end{equation}
  in which the unlabeled vertical isomorphisms are given by \autoref{thm:bicatunit} and the horizontal ones are the respective associativity constraints.
  Of course, the lower triangle commutes by the unit axiom for the external monoidal structure $\otimes$, so it suffices to produce the upper square.

  First of all, since $\otimes$ is pseudonatural and cocontinuous in both variables, we may pull all restriction and 
  left Kan extension functors out of all tensor products, as in~\eqref{eq:unitiso-pullout}.
  By the functoriality of mates, this does not affect the commutativity of any diagrams.
  Thus, we may consider $(X \otimes_{[B]} \lI_B) \otimes_{[B]} Y$ to be obtained from $(X\otimes\lS_\bbone  )\otimes Y$ 
  by transporting along the diagonal in the following diagram, restricting to the left and left Kan extending downwards.
  \begin{equation}
    \tiny\vcenter{\xymatrix@C=1.5pc{
        \tw(B)\op\times_{B} \tw(B)\times_{B}\tw(B)\op \ar[d]_{\pi_{\tw(B)}}\ar[r]^-s
        \ar@{}[dr]|{\fbox{0}}
        &\tw(B)\op\times_{B} \tw(B)\times B\ar[d]^s\ar[r]^-{t\op}
        \ar@{}[dr]|{\fbox{2}}
        &B\op\times\tw(B)\times B\ar[r]^-{\pi_{\tw(B)}}\ar[d]^{(t,s)}
        &B\op\times B\ar[dd]^-=\\
        \tw(B)\op\times\tw(B)\op\ar[d]_{{\pi_{\tw(B)\op}}\times 1}\ar[r]|-{(t\op,s\op)}
        \ar@{}[dr]|{\fbox{1}}
        &\tw(B)\op\times B\op\times B\ar[d]^-{\pi_{\tw(B)\op}}\ar[r]_-{(t\op,s\op)}
        &B\op\times B\times B\op\times B
        \ar@{}[dr]|{\fbox{3}}
        &\\
        \tw(B)\op\ar[r]_-{(t\op,s\op)}\ar[d]_-{\pi_{\tw(B)\op}}
        \ar@{}[drrr]|{\fbox{4}}
        &B\op\times B\ar[rr]_-=
        &
        &B\op\times B \ar[d] \\
        \bbone \ar[rrr]_-{\infty}&&&
        \ncoll(\hom_B).
      }}\label{eq:rho}
  \end{equation}
  The three small rectangles~\fbox{\tiny 0}--\fbox{\tiny 2} are homotopy exact since all of them are pullbacks and in each case one of the relevant functors is a Grothendieck (op)fibration.
  Moreover, we know from the proof of \autoref{thm:bicatunit} that the large rectangle on the right (composed of regions~\fbox{\tiny 2} and~\fbox{\tiny 3}) 
  endowed with the same 2-cell as in that proof is homotopy exact.
  Finally, in the rectangle~\fbox{\tiny 4} we have the 2-cell~\eqref{eq:coendcollage} rendering it a comma square, hence homotopy exact.

  For conciseness, let us denote by $\mu_k$ the mate-transformation associated to a composite of regions~\fbox{\tiny$k$}, perhaps whiskered with functors on one side or the other.
  Then by definition, the left vertical isomorphism of~\eqref{eq:coh} is the composite $\mu_{23} \mu_2^{-1}$.
  By the functoriality of mates, we have
  \begin{align}
    \mu_{23} \mu_2^{-1}
    &= \mu_{0123} \mu_0^{-1}\mu_1^{-1}\mu_2^{-1}\\
    &= \mu_{0123} \mu_0^{-1}\mu_2^{-1}\mu_1^{-1}\\
    &= \mu_{0123} \mu_{02}^{-1}\mu_1^{-1}.
  \end{align}

  Similarly, we may consider $X \otimes_{[B]} (\lI_B \otimes_{[B]} Y)$ to be obtained from $X\otimes(\lS_\bbone  \otimes Y)$ by transporting along the diagonal in the following diagram:
  \begin{equation}
    \tiny\vcenter{\xymatrix@C=1.3pc{
        \tw(B)\op\times_{B} \tw(B)\times_{B}\tw(B)\op \ar[d]_{\pi_{\tw(B)}}\ar[r]^-{t\op}
        \ar@{}[dr]|{\fbox{5}}
        &B\op\times \tw(B)\times_{B} \tw(B)\op\ar[d]^t\ar[r]^-s
        \ar@{}[dr]|{\fbox{7}}
        &B\op\times\tw(B)\times B\ar[r]^-{\pi_{\tw(B)}}\ar[d]^{(t,s)}
        &B\op\times B\ar[dd]^-=\\
        \tw(B)\op\times\tw(B)\op\ar[d]_{1\times{\pi_{\tw(B)\op}}}\ar[r]|-{(t\op,s\op)}
        \ar@{}[dr]|{\fbox{6}}
        &B\op\times B\times\tw(B)\op\ar[d]^-{\pi_{\tw(B)\op}}\ar[r]_-{(t\op,s\op)}
        &B\op\times B\times B\op\times B
        \ar@{}[dr]|{\fbox{8}}
        &\\
        \tw(B)\op\ar[r]_-{(t\op,s\op)}\ar[d]_-{\pi_{\tw(B)\op}}
        \ar@{}[drrr]|{\fbox{9}}
        &B\op\times B\ar[rr]_-=
        &
        &B\op\times B \ar[d]\\
        \bbone \ar[rrr]_-{\infty}&&&
        \ncoll(\hom_B)
      }}\label{eq:lambda}
  \end{equation}
  As before, the rectangles~\fbox{\tiny 5}--\fbox{\tiny 7} are homotopy exact, as is the rectangle~\fbox{\tiny 78} with an analogous 2-cell, while \fbox{\tiny 9} is literally equal to \fbox{\tiny 4}.
  The right vertical composite in~\eqref{eq:coh} is then by definition $\mu_{78} \mu_7^{-1}$, which as before is equal to $\mu_{5678} \mu_{57}^{-1} \mu_6^{-1}$.

  However, from the proof of the Fubini Theorem~\ref{lem:Fubini}, the associativity isomorphism from \autoref{thm:bicatassoc} is composed of the associativity of $\otimes$ together with $\mu_6 \mu_1^{-1}$.
  Thus, it will suffice to show that $\mu_{5678} \mu_{57}^{-1} = \mu_{0123} \mu_{02}^{-1}$.
  But the commutative rectangles~\fbox{\tiny02} and~\fbox{\tiny57} are obviously identical, so it remains to show $\mu_{5678} = \mu_{0123}$.

  Now the 2-cells that live in regions~\fbox{\tiny 0123} and~\fbox{\tiny 5678} are \emph{not} equal, and indeed it doesn't even make sense to ask whether they are, since their codomains are different.
  The equality $\mu_{5678} = \mu_{0123}$ only makes sense after postwhiskering with $(\pi_{\tw(B)\op})_!$.
  However, since regions~\fbox{\tiny4} and~\fbox{\tiny9} are homotopy exact, to show the equality of these two whiskerings, it suffices to show the equality of the pasted 2-cells in the regions~\fbox{\tiny01234} and~\fbox{\tiny56789}, whose domains and codomains \emph{are} equal.
  It is straightforward to check that both of these 2-cells just ``compose up'' a composable string of three arrows.
  Hence they are the same, and~\eqref{eq:coh} commutes.
\end{proof}

\bibliographystyle{alpha}
\bibliography{derivbicat}

\end{document}